\documentclass{amsart}
\usepackage[english]{babel}
\usepackage{amsfonts}
\usepackage{color}
\usepackage{authblk}
\usepackage{graphicx}
\usepackage{tikz}
\usepackage{stmaryrd}
\usepackage{mathrsfs,upgreek}
\usepackage{eucal}
\usepackage{enumerate}
\usepackage{changes}
\usepackage{amsaddr}
\usepackage[hmargin = 2.8cm,vmargin = 2.8cm]{geometry}
\usepackage{amsthm,amssymb,latexsym,mathabx}
\usepackage{amsmath}
\usepackage{todonotes}

\newtheorem{theorem}{Theorem}[section]
\newtheorem{lemma}[theorem]{Lemma}
\newtheorem{proposition}[theorem]{Proposition}
\newtheorem{corollary}[theorem]{Corollary}

\theoremstyle{definition}
\newtheorem{definition}[theorem]{Definition}
\newtheorem{example}[theorem]{Example}

\newtheorem{notation}{Notation}

\newtheorem{remark}[theorem]{Remark}

\numberwithin{equation}{section}

\newcommand{\join}{\vee}

\newcommand{\meet}{\wedge}

\newcommand{\alg}[1]{{\textbf{\upshape #1}}}  %
\newcommand{\vv}[1]{\mathsf {#1}}

\newcommand{\sse}{\subseteq}

\newcommand{\II}{{\mathbf I}} 

\newcommand{\FL}{\vv F\vv L}

\newcommand{\cc}[1]{\mathcal{#1}}
\newcommand{\si}{{\rm SI}}
\newcommand{\back}{\backslash}
\newcommand{\VS}{\mathcal{VS}}
\newcommand{\sell}{\sigma^\ell}
\newcommand{\sr}{\sigma^r}
\newcommand{\gell}{\gamma^\ell}
\newcommand{\gr}{\gamma^r}

\newcommand{\Sell}[2]{\sigma^\ell_{#1#2}}
\newcommand{\Sellx}[3]{\sigma^\ell_{#1#2}(#3)}
\newcommand{\Gell}[2]{\gamma^\ell_{#1#2}}
\newcommand{\Gellx}[3]{\gamma^\ell_{#1#2}(#3)}
\newcommand{\Sr}[2]{\sigma^r_{#1#2}}
\newcommand{\Srx}[3]{\sigma^r_{#1#2}(#3)}
\newcommand{\Gr}[2]{\gamma^r_{#1#2}}
\newcommand{\Grx}[3]{\gamma^r_{#1#2}(#3)}
\newcommand{\betaij}{\beta_{ij}}

\newcommand{\leftr}{\bbslash}
\newcommand{\rightr}{\sslash}
\newcommand{\fil}{\mathbf{Fil}}

\newcommand{\dfil}{{\rm Fil}}

\newcommand{\betafil}{\mathbf{\beta Fil}}
\newcommand{\dbetafil}{{\rm \beta Fil}}

\newcommand{\arc}[1]{{\vv{arc}}(#1)}
\newcommand{\BlockG}{\vv{ArcG}}
\newcommand{\BG}{\vv{AG}}

\newcommand{\BlG}{\vv{BlockG}}
\newcommand{\ssBG}{\vv{sAG}}

\begin{document} 
\title[Blockwise Gluings And Amalgamation Failures in IRLs]{Blockwise Gluings And Amalgamation Failures\\ in Integral Residuated Lattices}

\author{Valeria Giustarini and Sara Ugolini}
\address{Artificial Intelligence Research Institute (IIIA), CSIC, Barcelona, Spain}
\email{valeria.giustarini@iiia.csic.es, sara@iiia.csic.es}

\begin{abstract}
We introduce the {\em blockwise gluing} construction. This describes residuated integral chains which can be decomposed into (possibly) partial algebras, stacked one on top of the other, and such that elements in a certain component multiply {\em in blocks}, i.e., in the same way, with respect to lower components. This construction generalizes that of 1-sums (or ordinal sums). As first main results, we provide finite axiomatizations for varieties generated by particular chains that are gluings of their archimedean components. For such varieties we also prove the finite embeddability property, and as a consequence, the decidability of their universal theory.

Moreover, we solve in the negative several longstanding open problems in the literature about the amalgamation property (AP). Indeed, we provide denumerably many new examples of varieties lacking the AP, including: semilinear (commutative) integral residuated lattices, semilinear FL$_w$-algebras, MTL-algebras, involutive and pseudocomplemented MTL-algebras, and all their $n$-potent subvarieties for $n\geq 2$. For the commutative varieties, this also entails that the associated substructural logics do not have the deductive interpolation property.
\end{abstract}
\maketitle

\section{Introduction}
Residuated lattice-ordered monoids, or residuated lattices for short, are algebraic structures in the signature $(\lor, \land, \cdot, \backslash, /, 1)$, where the monoidal reduct and the lattice order $\leq$ are connected by the {\em residuation law}:
$$x \cdot y \leq z \;\;\Leftrightarrow\;\; y \leq x \backslash z \;\;\Leftrightarrow\;\; z \leq z /y,$$
in which the two divisions $(\backslash, /)$ intuitively allow one to solve inequalities, in much of the same way as one would use inverses to solve equations in groups. First introduced by Ward and Dilworth to study ideals of rings \cite{WardDilworth}, residuated lattices have caught the interest of researchers in the field of algebraic logic. This is due in particular to the fact that they constitute a natural semantics for a wide framework of nonclassical logics, namely {\em substructural logics} \cite{GJKO,MPT2023}. These are a large container of systems which include many of the interesting examples of logics amenable to the algebraic study: classical logic, intuitionistic logic, intermediate logics, many-valued logics, relevance logics to name a few. In this context, the monoidal operation acts as a conjunction, the two divisions play the role of implications, and the residuation law then reads as an algebraic version of a deduction theorem.

Residuated lattices comprise a variety of different structures; in fact, they do not satisfy any special purely lattice-theoretic nor monoid-theoretic identity, since any lattice and any monoid can be embedded into some residuated lattice \cite{Bahlsetal,JipsenTsinakis}. 
This wide spectrum of structures makes the study of residuated lattices both interesting and complicated, and while the literature on the subject is quite vast and many deep results have been obtained in the last decades (see e.g. the monographs \cite{GJKO,MPT2023}), at the present moment large classes of algebras still lack an effective structural understanding. One of such classes is given by commutative integral chains, i.e., totally ordered structures where the
unit of the monoid is the top element of the lattice, and the monoidal operation is commutative. This class is of particular interest in the field of nonclassical logics, since they give semantics to many-valued logics such as \L ukasiewicz logic, G\"odel-Dummett logic, H\'ajek Basic Logic, and Esteva-Godo's monoidal t-norm based logic \cite{Hajek98,EstevaGodo2001}. 

This work aims to further our understanding of the structure of integral residuated lattices in general, and totally ordered ones in particular. A meaningful way of tackling this issue is to decompose these algebras in ``simpler'' components, which ideally are subalgebras or subreducts, and understand constructions that allow to put such components back together. 
This approach has been particularly successful in varieties generated by totally ordered structures satisfying {\em divisibility}, i.e., such that if $x \leq y$, there are $z, z'$ such that $x =  yz = z'y$. The results in \cite{AM2003}, later extended to the non-commutative setting in \cite{Dvurevcenskij}, use {\em 1-sums} (or {\em ordinal sums}) to yield a description of totally ordered divisible chains. These results have been succesfully utilized to study, for instance, free algebras \cite{AguzzoliBova}, subvariety lattices \cite{AM2003,AglianoUgolini22,AguzzoliBianchi}, and interesting algebraic properties such as amalgamation \cite{FussnerSantschi25}.

While 1-sums have shown to be effective in divisible structures, they do not give a useful description for all integral chains. Indeed, the class of integral chains that are not decomposable in 1-sums is quite large --in the $0$-bounded cases, all involutive structures belong to it \cite{Montagnaetal2006}-- and it is not closed under subalgebras \cite{GalatosUgolinimanuscript}, hence cannot be described by equations, quasiequations, nor universal sentences. It then becomes relevant to consider more general constructions, capable of effectively describing larger classes of chains. An example in this direction is in \cite{GalatosUgolini}; the authors first introduce a construction that {\em glues} together structures intersecting not just at $1$, like in 1-sums, but in a congruence filter $F$, and then obtain a construction involving partial algebras by removing said filter $F$. 

Building on these latter ideas, we introduce the {\em blockwise gluing}, with which we are able to describe integral chains whose components are made of possibly partial algebras, stacked one on top of the other and intersecting at $1$, and such that elements in a certain component multiply {\em in blocks}, i.e., in the same way, with respect to lower components. This construction generalizes that of 1-sum and the one introduced in \cite{GalatosUgolini}. As first main results, we show that the blockwise gluing can be effectively used to study varieties generated by chains which are gluings of their archimedean components. We provide finite axiomatizations for $n$-potent varieties of this kind. Moreover, using {\em residuated frames} \cite{GalatosJipsen2013} as a technical tool, we show that all the $n$-potent varieties we analyze have the finite embeddability property, as a result of the fact that the blockwise gluing construction is preserved by the residuated frames we make use of. This entails that all such varieties have decidable universal theories.

Finally, we show that, within the class of chains describable by means of blockwise gluings, one can construct a V-formation that does not have an amalgam in residuated chains. This solves in the negative several longstanding open problems in the literature about the amalgamation property (see the table in page 205 of \cite{MPT2023}). In particular we establish that the following lack amalgamation: the varieties of semilinear (commutative) integral residuated lattices, semilinear FL$_w$-algebras and MTL-algebras, involutive and pseudo-complemented MTL-algebras, and all their $n$-potent subvarieties for $n\geq 2$. For the commutative varieties, this also entails that the associated logics do not have the deductive interpolation property.

\section{Preliminaries}
In this preliminary section we introduce the main notions of interest. First, we recall the basic facts about residuated lattices which will be of use in this paper. Secondly, we recall in some detail the residuated frames constructions, which will be a key tool to prove some interesting properties of the varieties generated by blockwise gluings, and will also be instrumental in some axiomatization results.

\subsection{Residuated lattices}A residuated lattice is an algebra $\alg A = ( A,\lor,\land,\cdot,\back, /, 1)$ of type $(2, 2, 2, 2, 2, 0)$ such that:
\begin{enumerate}
\item $(A, \lor, \land )$ is a lattice;
\item $(A, \cdot, 1)$ is a monoid;
\item $\back$ and $/$ are the right and left division of $\cdot$: for all $x,y, z \in A$,
$$
x \cdot y \leq z \;\;\Leftrightarrow\;\; y \leq x \back z \;\;\Leftrightarrow\;\; x \leq z /y,
$$
where $\leq$ is  the lattice ordering.
\end{enumerate}
As usual, in what follows we write $xy$ for $x \cdot y$. 
Importantly, divisions in a residuated lattice can be characterized by means of the product and the order. Specifically, given a residuated lattice $\alg A$,
\begin{equation}
	x \backslash z = \max \{y \in A: xy \leq z\}; \quad z / y = \max\{x \in A: xy \leq z\}.
\end{equation}
A fact that will sometimes be used in this work is that a monoid with a complete lattice order is residuated if and only if multiplication distributes on both sides over arbitrary joins (\cite[Corollary 3.11]{GJKO}). For a finite totally ordered structure, this is equivalent to the product preserving the order and the bottom element being absorbing\footnote{We observe for the sake of the reader that this last condition is wrongly omitted in Corollaries 3.12 and 3.13 of \cite{GJKO}.} (since the bottom element is the empty join).

Residuated lattices form a variety, denoted by $\mathsf{RL}$, as residuation can be expressed equationally; see \cite{BlountTsinakis2003}. When the monoidal identity is the top element of the lattice we say that the residuated lattice is \emph{integral} or an \emph{IRL}; we call the corresponding variety $\mathsf{IRL}$.
Residuated lattices with an additional constant $0$ are called \emph{pointed}. \emph{$0$-bounded} residuated lattices are pointed residuated lattices that satisfy the identity $0 \leq x$. We call a residuated lattice {\em bounded} if it is integral and $0$-bounded.
 The variety of bounded residuated lattices is called $\mathsf{FL_w}$, referring to the fact that it is the equivalent algebraic semantics of the Full Lambek calculus with the structural rule of {\em weakening} (see \cite{GJKO,MPT2023}). 

A residuated lattice is called \emph{commutative} if the monoidal operation is commutative. In this case the two divisions coincide, and we write $x \to y$ for $x \back y = y / x$. We write $\mathsf{CRL}$ and $\mathsf{CIRL}$, respectively, for the commutative subvarieties of $\mathsf{RL}$ and $\mathsf{IRL}$, and refer to commutative $\mathsf{FL_{w}}$-algebras as $\mathsf{FL_{ew}}$-algebras, since commutativity of the monoidal operation corresponds to the structural rule of {\em exchange}. In $\vv{FL_{ew}}$-algebras we also write $\neg x$ for $x \to 0$. 

We call \emph{chain} a totally ordered structure; hence a \emph{residuated chain} is a residuated lattice which is totally ordered as a lattice. 
Given a class of residuated lattices $\vv K$, we denote with $\vv K_{\rm c}$ the class of structures in $\vv K$ that are chains:
\begin{equation}
	\vv K_{\rm c} = \{\alg A \in \vv K: \alg A \mbox{ is a chain}\}.
\end{equation}
We say that a residuated lattice is \emph{semilinear} if it is a subdirect product of chains. The variety of semilinear residuated lattices is called $\vv{SemRL}$ and it is axiomatized by (see for instance \cite{MPT2023}): 
\begin{equation}
	\tag{sem} [u\backslash((y \backslash x \land 1)u) \meet 1] \lor  [(v(x \backslash y \land 1))/ v \meet 1]=1,
\end{equation}
 or in the commutative case by \emph{prelinearity}:
\begin{equation}
	\tag{prel} (x \to y \land 1) \join (y \to x \land 1) = 1.
\end{equation} 
We call the variety of semilinear commutative residuated lattices $\vv{SemCRL}$ and the subvariety of those satisfying also integrality $\vv{SemCIRL}$.
In particular the subvariety of prelinear $\vv {FL_{ew}}$-algebras is called $\vv{MTL}$ since it is the equivalent algebraic semantics of the \emph{monoidal t-norm based logic}, MTL for short,  introduced by Esteva and Godo \cite{EstevaGodo2001}. 

A relevant subvariety of $\vv{MTL}$ is the variety of BL-algebras, the equivalent algebraic semantics of  H\'ajek's Basic Logic \cite{Hajek98}, which is axiomatized with respect to $\vv{MTL}$ by the divisibility identity: 
\begin{equation}
	\tag{div} x \meet y= x(x\backslash y)=(y/x)x.
\end{equation}
The previous equation is equivalent to the following condition in chains: given a chain $\alg A$, if $a \leq b$ there are $c,d \in A$ such that $a=cb=bd$. The variety of the $0$-free subreducts of BL-algebras is the variety of \emph{basic hoops} and it corresponds to divisible, semilinear CIRLs.
Among the most well known subvarieties of BL-algebras one finds \emph{MV-algebras} and {\em G\"odel algebras}, which are the equivalent algebraic semantics of, respectively, \L ukasiewicz infinite-valued logic and G\"odel-Dummett logic (which is the intermediate logic complete with respect to totally ordered Heyting algebras). 
MV-algebras are BL-algebras satisfying \emph{involution}:
\begin{equation*}\tag{inv}
    \neg \neg x=x; 
\end{equation*}
while G\"odel-algebras are BL-algebras satisfying \emph{idempotency} of the product:
\begin{equation*}\tag{id}
	x^2=x
\end{equation*}
Finite totally ordered MV-algebras and G\"odel algebras are fully determined by their cardinality.
  
With respect to their structure theory, residuated lattices are quite well-behaved. Indeed, they are {\em $1$-regular} and {\em ideal determined}, meaning that congruences are in one-one correspondence with subsets that we call {\em congruence filters}, which correspond to the equivalence classes of $1$.
In a residuated lattice $\alg A$, a congruence filter $F$ is a non-empty upset of $A$, closed under products (if $x, y \in F$, then $x y \in F$) and under conjugates, i.e. if $x \in F$ , then $yx/y, y\backslash xy \in F$ for every $y \in A$. We will denote by $\alg {Fil}(\alg A)$ the lattice of congruence filters of $\alg A$. 
 $\alg {Fil}(\alg A)$ is isomorphic to the congruence lattice of $\alg A$, $\alg{Con}(\alg A)$, via the maps:
$$F \mapsto \theta_{F} = \{(x, y) \in A \times A : x \back y, y \back x \in F\} = \{(x, y) \in A \times A: x / y, y / x \in F\},$$
$$\theta \mapsto F_{\theta} = \{x \in A:  (x, 1) \in \theta\}$$
for all $F \in {\rm Fil}(\alg A), \theta \in {\rm Con}(\alg A)$. In what follows, given an algebra $\alg A$ and a congruence filter $F$ of $\alg A$, we may write $\alg A/F$ in place of $\alg A/\theta_F$ for the quotient structure. Moreover, we will often denote elements in the quotient $\alg A/F$ as $[a]_F$, for $a \in A$; when there is no danger of confusion, we simply write $[a]$. 

 The following universal algebraic notion is going to be relevant for some results. An algebra $\alg A$ has the {\em congruence extension property} (for short, CEP) if, for any subalgebra $\alg B$ of $\alg A$ and $\Theta \in \mathsf{Con}(\alg B)$, there exists $\Phi \in  \mathsf{Con}(\alg A)$ such that $\Phi \cap B^2 = \Theta$. A class of algebras has the CEP if each of its members has the CEP. It is well known (see \cite{GJKO,MPT2023}) that all commutative subvarieties of residuated lattices or $\mathsf{FL}$-algebras have the CEP.

	Finally, given that {\em partial} IRLs play a role in this manuscript, let us give a precise definition. 
	We call a \emph{partial IRL} a partially ordered partial algebra $(\alg A, \leq)$ in the language of residuated lattices, such that:
	\begin{enumerate}
	\item $1 \in A$ and $\alg A$ is integral: $x \leq 1$ for all $x \in A$;
		\item the three axioms of RLs are satisfied whenever they can be applied, in the following sense: \begin{enumerate}
		\item $x \lor y$ is the least common upper bound of $x$ and $y$ whenever it exists, and similarly $x \land y$ is the largest lower bound whenever it exists;
		\item $x 1 = 1 x = 1$, and if $xy, (xy)z, yz, x(yz)$ are defined then $(xy)z = x(yz)$;
		\item if $xy$, $z / y$, and $x \back z$ are defined, then $xy \leq z$ if and only if $x \leq z / y$ if and only if $y \leq x \back z$.
		\end{enumerate}
		\item multiplication is order preserving when defined:  $a \leq b$ and $ac, bc$ defined implies $ac \leq bc$, and likewise for left multiplication.
		\item whenever defined, the division operations are order-preserving in the numerator and order-reversing in the denominator. That is:
		$x \leq y $ and $z \backslash x, z \backslash y$ defined implies $z \backslash x \leq z \backslash y$; $x \leq y $ and $ y \backslash z, x \backslash z$ defined implies $ y \backslash z \leq x \backslash z$; likewise for right division. 
	\end{enumerate}
	While the above definition is quite general, most partial residuated lattices we will consider in this work will only possibly lack some divisions and some joins.

\subsection{Residuated frames}
In order to establish key properties of some varieties under consideration, specifically, the finite model property (FMP) and the finite embeddability property (FEP), we will rely on the framework of residuated frames as developed in \cite{GalatosJipsen2013}.  We recall the main facts formulated as needed in this work. For more details, we refer the reader to \cite{GalatosJipsen2013}.

In essence, a residuated frame is a multi-sorted structure $(W,W',N, \circ, \leftr, \rightr)$ that abstracts key characteristics of a residuated lattice.
The ternary relations $\circ, \leftr, \rightr$ are analogous to the operations $(\cdot, \backslash, /)$, in so far that the former satisfy a residuation law with respect to the relation $N$, as the latter do with respect to the lattice order $\leq$. In a nutshell, residuated frames allow one to construct residuated structures with controlled properties by appropriately defining the above relations.

 \begin{definition}
 	A {\em residuated frame} is a structure of the form $\alg W =(W,W',N, \circ, \leftr, \rightr)$ such that $\circ \subseteq W^3$, $\leftr \subseteq W \times W' \times W'$, $\rightr \subseteq W' \times W \times W'$, and  $N\subseteq W \times W'$ is a {\em nuclear relation}, i.e.:
 		$$(u \circ v)N w \;\;\mbox{ iff }\;\; v N (u \leftr w) \;\;\mbox{ iff }\;\; u N (w \rightr v)$$
 		for all $u,v \in W$ and $w \in W'$, where $x * y= \{z : (x,y,z) \in *\}$ for $* \in \{\circ, \leftr,\rightr\}$.
 \end{definition}
 The relation $N$ is a Galois relation and it induces a Galois connection for $\mathcal{P}(W)$ and $\mathcal{P}(W')$, given by the maps $^{\triangleright}: \mathcal{P}(W) \to \mathcal{P}(W')$ and $^{\triangleleft}: \mathcal{P}(W') \to \mathcal{P}(W)$ defined as follows:
 $$
 X^{\triangleright}= \{ z \in W' : \forall x \in X, x N z \} \qquad Z^{\triangleleft}=\{ x \in W : \forall z \in Z, x N z\},
 $$for $X \subseteq W$ and $Z \subseteq W'$. Thus one has  $X \subseteq Z^\triangleleft$ if and only if $Z \subseteq X^\triangleright$. One can then consider the associated closure operator $\gamma_N(X)=X^{\triangleright \triangleleft}$ on the powerset $\mathcal{P}(W)$. Importantly, the subsets $\{\{u\}^\triangleleft : u \in W'\}$ are a basis for $\gamma_N$, that is to say, the elements in $\gamma_N[\mathcal{P}(W)]$ are exactly the intersections of elements in $\{\{u\}^\triangleleft : u \in W'\}$. Let 
 $$
 X \circ Y = \{x \circ y : x \in X, y \in Y\},
 $$
 then $\gamma_N$ is a nucleus on $(\mathcal{P}(W), \circ)$, i.e., a closure operator such that $\gamma_N(X)\circ \gamma_N(Y)\subseteq \gamma_N (X \circ Y)$. We are ready to define the {\em Galois algebra} of $\alg W$ as $\alg W^+=(\gamma_N[\mathcal{P}(W)], \cap, \cup_{\gamma_N}, \circ_{\gamma_N}, \backslash, /)$ where 
\begin{eqnarray*}
	& X\circ_{\gamma_N} Y = \gamma_N(X \circ Y), &X \cup_{\gamma_N} Y = \gamma_N(X \cup Y)\\
&  X \backslash Y = \{ z : \forall x \in X, x \circ z \in Y\}, &X / Y =\{ z :\forall x \in X, z \circ x \in Y \}.
\end{eqnarray*}
$\alg W^+$ is in general a residuated lattice-ordered groupoid. In order to get a residuated lattice, we need to add further properties.

A {\em unital residuated frame} (ru-frame for short) $\alg W=(W, W', N, \circ, \leftr, \rightr, E)$ is a residuated frame with a set $E \subseteq W$ such that $(x \circ E)^\triangleright= \{ x \}^\triangleright=(E \circ x)^\triangleright$ for all $x \in W$. A {\em ruz-frame} is a unital frame with an extra subset $D\subseteq W$. In both cases, $\alg W^+$ has a unit $\gamma_N(E)$ and in the latter case $\gamma_N(D)$ is the interpretation of the constant $0$. Moreover, a residuated frame is {\em associative} if for all $x,y,z \in W$, $[(x \circ y)\circ z]^\triangleright=[x \circ (y \circ z)]^\triangleright$. The Galois algebra of an associative ru-frame is a residuated lattice, while the one of an associative ruz-frame is an FL-algebra. We also observe that in what follows, $\circ$ will always be a ternary relation obtained by a binary monoidal operation. 

A typical application of residuated frames is to prove generation of a variety or quasivariety by its finite algebras. We say that a class of algebras $\vv K$ has the {\em finite model property} (FMP) if any equation that fails in $\vv K$, fails in a finite member of $\vv K$. For a variety, this means being generated by its finite algebras. Moreover, we say that $\vv K$ has the {\em finite embeddability property} (FEP) when for any given finite partial subalgebra $\alg B$ of an algebra $\alg A \in \vv K$, there exists a finite algebra $\alg D \in \vv K$ into which $\alg B$ can be embedded. The FEP clearly implies the FMP, and for a variety, it coincides with being generated as a quasivariety by its finite members. We recall how residuated frames are used towards FMP and FEP, since we will use it in what follows.
 
 Let $\alg A$ be a residuated lattice and $\alg B$ a partial subalgebra of $\alg A$. Then $(W, \circ, 1)$ is taken to be the submonoid of $\alg A$ generated by $\alg B$. Moreover, define a map $u$ on $(W, \circ, 1)$ as $u(x)=v \circ x \circ w$ for $v, w \in W$; we call such map a {\em section}, and we denote by $\alg S_W$ the set of all sections. Then set $W'=\alg S_W \times \alg B$. Finally, define the binary relation $N$ as $$x N (u,b) \mbox{ if and only if } u(x)\leq_{\alg A} b.$$ $N$ is nuclear with respect to $\leftr,\rightr$ defined as follows: 
 $$x \leftr (u,b)= \{(u(x \circ -), b)\} \;\;\mbox{ and }\;\; (u,b) \rightr y= \{(u(- \circ y), b)\},$$
 where $u(x \circ -): z \mapsto u(x \circ z)$ and $u(- \circ y): z \mapsto u(z \circ y)$.
Consider $\alg W_{\alg A, \alg B}=(W,W',N, \circ, \leftr, \rightr, \{1\})$, one has that: 
 \begin{proposition}[{\cite{GalatosJipsen2013}}]\label{prop:galatosjipsen}
$\alg W_{\alg A, \alg B}$ is a residuated frame and the map $b \mapsto \{(id, b)\}^\triangleleft$ embeds $\alg B$ into $\alg W_{\alg A, \alg B}^+$. 
\end{proposition}
Thus, residuated frames can be used to construct a residuated lattice into which the initial partial subalgebra embeds. This is often used to prove that a certain variety or quasivariety is generated by its finite members,  by showing that any (quasi)identity that fails in the (quasi)variety fails in a finite algebra. To this end, one needs to show that $\alg W_{\alg A, \alg B}^+$ ends up in the desired class of algebras. This step may require an ad hoc proof, since not all properties of the initial residuated lattice need be preserved by the construction. The following is a useful observation. Note that we assumed the constant $1$ to belong to all partial IRLs, so to all partial subalgebras in this context.
 
\begin{lemma}\label{lemma: finite chain}
	Let $\alg A$ be an integral residuated chain and let $\alg B$ be a finite partial subalgebra of $\alg A$; then $\alg {W}_{\alg A, \alg B}^+$ is a finite integral  residuated chain.
\end{lemma}
\begin{proof}
The fact that $\alg W_{\alg A, \alg B}^+$ is finite is shown in the proof of \cite[Theorem 3.18]{GalatosJipsen2013}, and integrality is easily seen to hold given that we consider $1 \in B$. Let us prove that $\alg W_{\alg A, \alg B}^+$ is a chain. Recall that $\{\{(u,b)\}^\triangleleft : (u,b) \in W' \}$ forms a basis for $\alg W_{\alg A, \alg B}^+$. Moreover, by direct application of the definitions, 
$$
\{(u,b)\}^\triangleleft= \{x \in W : x N (u,b) \}=\{x \in W : u(x) \leq_{\alg A} b \}.
$$
Notice then that $\{(u,b)\}^\triangleleft$ is a downset of $W$, which inherits the order of the chain $\alg A$. Hence, the downsets in the basis and their intersections are totally ordered by inclusion, which implies that $\alg W_{\alg A, \alg B}^+$ is a chain. 
\end{proof}

\section{Blockwise gluings}
The main results of this work aim to further our understanding of the structure of integral residuated lattices in general, and chains in particular. A meaningful way of tackling this issue is to decompose these algebras in ``simpler'' components; particularly, by seeing an IRL $\alg A$ as the union of two (or more) components, which ideally are its subalgebras or subreducts. Given that we are particularly interested in totally ordered structures, we shall consider constructions that assemble these components in a totally ordered fashion, effectively stacking one on top of another.

Say that a residuated lattice $\alg A$ can be seen as the union of two substructures $\alg B$ and $\alg C$, one below the other, where we assume that $B \cap C = \{1\}$ so that they share the constant. It has been shown in \cite[Proposition 2.1]{GalatosUgolini} that, if one wants both $\alg B$ and $\alg C$ to be subalgebras of $\alg A$, $\alg A$ is necessarily the {\em 1-sum} (often called {\em ordinal sum} in residuated structures\footnote{We choose the naming 1-sum in accordance with \cite{Olson1,Olson2}, to avoid confusion with the ordinal sum of posets, where no elements are glued together.}) of $\alg B$ and $\alg C$. That is to say, the products between elements of $\alg B$ and $\alg C$ coincide with their meet. Therefore, in order to be able to combine the components in different ways, we must give up the requirement that the components be subalgebras, and instead allow them to be subreducts. A natural choice is to retain the product and order from the components, while allowing to omit closure under divisions, since the latter are entirely determined by the other operations.

 Let us then consider an IRL $\alg A$, seen as the union of two substructures $\alg B$ and $\alg C$, one below the other, where we assume that: $B \cap C = \{1\}$, and both subreducts are closed under products, meets, and joins (except in the case where $1$ is not join-irreducible in $\alg B$). Then it is clear that $\alg C$ is actually a subalgebra (hence an IRL), while $\alg B$ is a possibly partial IRL, where some divisions and joins might be missing, as they end up in $\alg C$. See Figure \ref{figure:gluing} for a pictorial intuition.
 
\begin{figure}
\begin{center}
\begin{tikzpicture}
\footnotesize{
 \draw  [dashed](0,0) ellipse (0.5 and 0.7);
 \draw  (0,1.6) ellipse (0.5 and 0.7);
  \fill (0,2.3) circle (0.05);
  \node at (0,1.6) {$\alg C$}; 
  \node at (0,0) {$\alg B$}; 
   \node at (0.15,2.5) {$1$}; 
  \node at (0,-1.3) {$\alg A$};
   }
 \end{tikzpicture}\caption{}\label{figure:gluing}
   \end{center}  
\end{figure}
 
In such a scenario, products between elements of $\alg B$ and $\alg C$ are not constrained to be the meet, unlike the case of the 1-sum, and in fact seem to have quite a high level of freedom. To help the intuition, let us provide non-isomorphic algebras $\alg A$ and $\alg A'$ that have the same subcomponents $\alg B$ and $\alg C$, as it is shown in the next example.
\begin{example}
	Consider the following two chains:
	\begin{center}
\begin{tikzpicture}
\draw (0,3) -- (0,1);
 \fill (0,3) circle (0.05);
 \node at (0.15,3.1) {$1$};
 \fill (0,2.5) circle (0.05);
 \node at (0.67,2.55) {$x = x^2$};
 \fill (0,2) circle (0.05);
 \node at (0.67,2.05) {$y = y^2$};
 \fill (0,1.5) circle (0.05);
 \node at (0.25,1.55) {$z$};
  \fill (0,1) circle (0.05);
  \node at (1.15,1.05) {$z^2 = xz = yz$};
  \node at (0.1,0.2) {$\alg A$};
 
 \draw (4,3) -- (4,1); 
 \fill (4,3) circle (0.05);
 \node at (4.15,3.1) {$1$}; 
 \fill (4,2.5) circle (0.05);
 \node at (4.67,2.55) {$x = x^2$};
 \fill (4,2) circle (0.05);
 \node at (4.67,2.05) {$y = y^2$};
 \fill (4,1.5) circle (0.05);
 \node at (4.67,1.55) {$z = xz$};
  \fill (4,1) circle (0.05);
  \node at (4.75,1.05) {$z^2 = yz$};
  \node at (4.1,0.2) {$\alg A'$};
   \end{tikzpicture}\end{center}
   It is easy to see that the above operations identify uniquely two IRL chains, $\alg A$ and $\alg A'$.
 In this example, $\alg B$ is the partial IRL with domain $\{1, z, z^2\}$ and the operations restricted from the ones of $\alg A$ (or, equivalently, $\alg A'$), while $\alg C$ is the subalgebra of $\alg A$ (or, again equivalently, $\alg A'$) with domain $\{1, x, y\}$. Note that the products between $\alg B$ and $\alg C$ differ in $\alg A$ and $\alg A'$, indeed: $x \cdot_{\alg A} z = z^2$ while $x \cdot_{\alg A'} z = z$. Notice also that neither $\alg A$ nor $\alg A'$ are the 1-sum of $\alg B$ and $\alg C$.
\end{example}
We observe that every IRL chain can be seen as a union of subreducts closed under products and the lattice operations (it suffices to take as components the different archimedean classes). Therefore, if one were able to effectively describe all possible ways of redefining products between such components, this would yield a characterization of all IRL-chains. At the present moment, we are far from understanding this level of generality in a meaningful way. Nonetheless, in this work we take a step towards this direction. Specifically, we will consider the particular case in which elements belonging to a component multiply {\em in blocks} with respect to lower components. Looking at Figure \ref{figure:gluing} again, we mean that all elements of $\alg C$ multiply in the same way with respect to each element of $\alg B$: for all $c, c' \in C$ and $b \in B$, $cb = c'b$ and $bc = bc'$. 

An example of such a construction is of course the 1-sum, and more generally, the {\em partial gluing} construction introduced in \cite{GalatosUgolini}, which will be the starting point of our investigation. We recall it here in some detail. 
 The key ingredient for the partial gluing construction is given by the partial algebra acting as the lower component in the gluing, called in \cite{GalatosUgolini} a
  \emph{lower-compatible triple}. This is a triple $(\alg K, \sigma, \gamma)$, where, intuitively, $\sigma$ stores the information for the product with respect to the upper component, while $\gamma$ stores the information for the divisions. More precisely, a lower-compatible triple $(\alg K, \sigma, \gamma)$ is such that:
 \begin{enumerate}
\item  $\alg K$ is a partial IRL with all operations defined, except for $x \back y$ and $y /x$, which are undefined if and only if $\sigma(x) \leq y$ and $x \not\leq y$.\item\label{property2UPT}  $(\sigma, \gamma)$  is a residuated pair, i.e. $\sigma(x) \leq y$ if and only if $x \leq \gamma(y)$, such that:\begin{enumerate}
\item $\sigma$ is a \emph{strong conucleus}, i.e, an interior operator such that for $x, y \neq 1$, $x \sigma(y) = \sigma(xy) = \sigma(x) y$, and $\sigma(1) = 1$.
\item $\gamma$ is a closure operator on $\alg K$, and
\item $x y, y  x \leq \sigma(x)$ for all $x, y \in K, y \neq 1$.
\end{enumerate}
\end{enumerate}
The partial gluing then considers a a lower compatible triple $(\alg K, \sigma, \gamma)$, and an IRL $\alg L$ with a splitting coatom $c_L$ (i.e., $L=\{1\} \,\cup\! \downarrow\! c_L$), such that: $K \cap L = \{1\}$, and if $x \lor y = 1$ in $K$ for some $x,y \in K-\{1\}$, then $L$ has a bottom element $0_{L}$.
Set $\pi = (\sigma, \gamma)$, then the partial gluing $\alg K \oplus_{\pi}  \alg L$ is the structure whose operations extend those of $\alg K$ and $\alg L$, except if $x \lor y = 1$ in $K$ then we redefine $x \lor y = 0_{L}$, and moreover:
\begin{align*}
x y &=\left\{\begin{array}{ll}
\sigma(x) & \mbox{ if } y \in L - \{1\}, x \in K - \{1\}\\
\sigma(y) & \mbox{ if } x \in L - \{1\}, y \in K - \{1\}\end{array}\right.\\
x\back\, y &=\left\{\begin{array}{ll} 
c_{L} & \mbox{ if } x, y \in  K \mbox { and $x \back^{\alg K} y$ is undefined}\\ 
\gamma(y) & \mbox{ if } x \in L - \{1\}, y \in K - \{1\} \\
1 & \mbox{ if } x \in K - \{1\}, y \in L\\ \end{array}\right.\\
y/\, x &=\left\{\begin{array}{ll} 
c_{L} & \mbox{ if } x, y \in  K \mbox{ and $y/^{\alg K} x$ is undefined}\\
\gamma(y) & \mbox{ if } x \in L - \{1\}, y \in K - \{1\} \\
1 & \mbox{ if } x \in K - \{1\}, y \in L\\\end{array}\right.\\
x \land y &=\begin{array}{ll}
x & \mbox{ if } x \in K - \{1\}, y \in L\end{array}.\\
x \lor y &=\begin{array}{ll}
y & \mbox{ if } x \in K -\{1\}, y \in L\end{array}.
\end{align*}
Then 
$\alg K \oplus_{\pi}  \alg L$ is an IRL \cite[Proposition 2.9]{GalatosUgolini}. Note that by construction $\alg L$ is a subalgebra of $\alg K \oplus_{\pi}  \alg L$, and, if $1$ is join irreducible in $\alg K$, $\alg K$ is a subreduct with respect to all operations except the divisions.  
The 1-sum construction corresponds to the particular case where the lower compatible triple is a total IRL, and $\sigma$ and $\gamma$ are the identity maps.

The intuition behind this construction in \cite{GalatosUgolini} comes from the idea of gluing together two IRLs, say $\alg B$ and $\alg C$, which intersect not only at the constant $1$, but at a shared congruence filter (say $B \cap C = F$); then one ``forgets'' the filter, and what is left are the components of the partial gluing. Here we start from the partial gluing itself, and consider a more general version of it, where left and right products from $\alg C$ to $\alg B$ are allowed to differ. We then show that this construction allows us to describe exactly the chains where components multiply in blocks, in the sense described above. Later on, we will take these ideas one step further and introduce an iterated version of the construction which allows one to describe chains which are decomposable in an arbitrary (totally ordered) set of components.

Let us now be more precise. We start by defining the algebras we wish to characterize, which we will say are {\em blockwise gluings} of their components.

\begin{definition}\label{def: block IRL}
	Let $\alg A, \alg C$ be IRLs, and $\alg B$ a (possibly) partial IRL. We say that $\alg A$ is a {\em blockwise gluing of $\alg B$ and $\alg C$} if:
	\begin{enumerate}
		\item $A =B \cup C$ and $B \cap C = \{1\}$;
		\item the elements of $B-\{1\}$ are strictly below the elements of $C$;
		\item $\alg C$ is a subalgebra of $\alg A$;
		\item given $x, y \in C -\{1\}$ and $z \in B-\{1\}$, $xz = yz$ and $zx = zy$.
	\end{enumerate}

\end{definition}
We will see how to construct blockwise gluings by generalizing the partial gluing construction. Given a residuated structure $\alg A$, we call an element $c \in A$ a {\em splitting coatom} of $\alg A$ if for all $x \in A$, $x \neq 1$, we have that $x \leq c$.
\begin{lemma} \label{lemma: propblockwise gluing}
	If $\alg A$ is a blockwise gluing of $\alg B$ and $\alg C$, then:
	\begin{enumerate}
		\item $\alg B$, with the inherited operations from $\alg A$, is a partial IRL where all operations are defined except:\begin{enumerate} 
		\item $x \backslash y$ iff there is $z \in C$ such that $xz \leq y$ iff $x \backslash y$ is a splitting coatom of $\alg C$;  
		\item $y /x$ iff there is $z \in C$ such that $zx \leq y$ iff $y / x $ is a splitting coatom of $\alg C$; 
		\item $x \lor y$ iff $x \lor y \in C$, and then $\alg C$ has a bottom element $\perp _{\alg C}$ such that $x \join y = \perp_{\alg C}$.
		\end{enumerate}
		\item let $b \in B-\{1\}$ and $c,d \in C$, then $c \backslash b = d \backslash b$ and $b / c = b / d$.
	\end{enumerate}
\end{lemma}
	\begin{proof}
		(1) Notice that $\alg B$ is a partial IRL that contains $1$ and is closed under products and meets. Let then $x, y \in B$, then $x \backslash y$ and $y/x$ can lie either in $B$ or $C$. In particular, $x \backslash y \in C$ if and only if $\max \{k \in A : xk \leq y\}$ is the splitting coatom of $\alg C$, since any element of $C$ multiplies the same with respect to $\alg B$ by Definition \ref{def: block IRL}(4). This happens if and only if there exists $z \in C$ such that $xz \leq y$, which proves (1a). With a similar argument one can show (1b).
			For (1c), in $\alg B$ the only joins that  are missing lie in $\alg C$. Let then $x,y \in B-\{ 1\}$, if $x \join y \in \alg C$, necessarily $\alg C$ has a bottom element $\perp_\alg{C}$ and $x \join y = \perp_{\alg C}$. Indeed, all elements of $\alg C$ are strictly above $x$ and $y$, and $x \join y$ is the smallest element with this property.
		
       (2) Let $b \in B - \{1\}$ and $c,d \in C$, then $$c \backslash b= \max\{k \in A : ck \leq b\} = \max \{ k \in B : ck \leq b \}=max\{k \in B : dk \leq b\}=d \backslash b,$$ where the second equality follows from the fact that $\alg C$ is closed under products and the elements of $\alg C$ are above the ones in $B - \{1 \}$, while the third one follows from Definition \ref{def: block IRL} (4), since $k \in B $ and $c,d \in C$. Similarly one can show that also $b/c=b/d$. This concludes the proof.
   \end{proof}
   \begin{remark}\label{rem:div}
   	 We further observe that, if $\alg A$ is a blockwise gluing of $\alg B$ and $\alg C$, $\alg B$ can be seen as a partial IRL with only some division missing. Indeed, in the partial order inherited from $\alg A$, the joins of elements of $\alg B$ that end up in $C$ within $\alg A$, exist and are set to $1$ once restricting to $B$.
   \end{remark}
    
	The lemma above shows a close similarity with algebras that are obtainable by partial gluings. However, in the blockwise gluing, left and right products between elements of the top component with elements of the lower component may differ. This observation inspires the following definition, which aims at characterizing the blockwise gluing abstractly. We will need two distinct operators, $\sell$ and $\sr$, in order to store the information for left and right products respectively; accordingly, we shall also need two operators for the divisions, $\gell$ and $\gr$.
	
	\begin{definition} \label{def: lower block}
		We call {\em weak lower block} a tuple $(\alg B, \sell, \sr, \gell, \gr)$ such that: 
		\begin{enumerate}
			\item $\alg B$ is a partial IRL where all operations are defined except possibly for some divisions; 
			\item $(\sell, \gell)$ and $(\sr,\gr)$ are residuated pairs;
			\item  $\sell$ and $\sr$ are interior operators such that for $x, y \neq 1$:
				\begin{enumerate}
				\item $x\sr(y)=\sr(xy)$ and $\sell(x)y=\sell(xy)$;
				\item $x\sell(y)=\sr(x)y$;
				\item $\sr(\sell(x))=\sell(\sr(x))$;
				\item $\sell(1)=\sr(1)=1$;
				\end{enumerate}
			
				\item $\gell$ and $\gr$ are closure operators; 
				\item $xy \leq \sr(x), \sell(y)$ and $yx\leq \sr(y), \sell(x)$ for all $x,y \in B$, $x,y \neq 1$.
		\end{enumerate}
		We call {\em lower block} a weak lower block such that if $x \backslash y$ is undefined then necessarily $\sr(x)\leq y$ and $x \not \leq y$ and if $y / x$ is undefined then necessarily $\sell(x)\leq y$ and $x \not \leq y$.
	\end{definition}
	Note that both $\sell$ and $\sr$ are conuclei, i.e. $\sell(x)\sell(y)\leq \sell(xy)$ and $\sr(x)\sr(y)\leq \sr(xy)$. Indeed $\sell(xy)=\sell(x)y \geq \sell(x)\sell(y)$ since $y \geq \sell(y)$ and $\sr(xy)=x \sr(y)\geq \sr(x)\sr(y)$ since $x \geq \sr(x)$. 

	Let $\alg A $ be a blockwise gluing of $\alg B$ and $\alg C$, and fix an element $c \in C - \{1\}$; we define
	$$
	\sell_{\alg A}(x) = cx,\;\; \sr_{\alg A}(x) = xc,\;\; \gell_{\alg A}(x)= c \backslash x,\;\;\gr_{\alg A}(x) = x / c
	$$
	 for $x \in B-\{1\}$, and $\sell(1) = \sr(1) = \gr(1) = \gell(1) = 1$.
	\begin{proposition}\label{prop: lower blocks with c}
		 Let $\alg A$ be a blockwise gluing of $\alg B$ and $\alg C$. Then $(\alg B, \sell_{\alg A}, \sr_{\alg A}, \gell_{\alg A}, \gr_{\alg A})$ is a lower block.
	\end{proposition}
	\begin{proof} 
	We proceed to prove properties (1)--(5) of Definition \ref{def: lower block}. For simplicity, within this proof we write $\sell, \sr,\gell, \gr$ for $\sell_{\alg A}, \sr_{\alg A}, \gell_{\alg A}, \gr_{\alg A}$.

		(1) The fact that $\alg B$ is a partial IRL where all operations are defined except for some divisions follows from Lemma \ref{lemma: propblockwise gluing} and Remark \ref{rem:div}. 		
		
		(2) Let us consider the pair $(\sell, \gell)$; one can easily show that it is a residuated pair. Indeed, if $\sell(x)=cx\leq y$, then, by residuation in $\alg A$, $x \leq c \backslash y=\gell(y)$. Similarly, also $(\sr, \gr)$ is a residuated pair. 
		
		(3) We now verify that $\sell$ and $\sr$ are interior operators satisfying the properties (a)-(d) in Definition \ref{def: lower block}(3). First, notice that $\sell$ and $\sr$ are decreasing maps. Indeed, since $\alg A$ is integral, $\sell(x)=cx\leq x$ and $\sr(x)=xc\leq x$. To prove that they are idempotent, note that $c$ and $c^2$ multiply the same with respect to the elements in $\alg B$, by the definition of a blockwise gluing, and then we have that $$\sell(\sell(x))=c^2x=cx=\sell(x) \;\mbox{ and }\; \sr(\sr(x))=xc^2=xc=\sr(x).$$ Let us now show order preservation. Suppose $x \leq y$, then, by order preservation of the product in $\alg A$, $\sell(x)=cx \leq cy =\sell(y)$ and $\sr(x)=xc \leq yc=\sr(y)$; hence $\sell$ and $\sr$ are interior operators. 
		Properties (a)-(c) follow from the associativity of the product in $\alg A$. Finally, (d), i.e. the fact that $\sr(1)=\sell(1)$, holds by definition. 	
			
		(4) We now want to show that both $\gell$ and $\gr$ are closure operators. First note that, as a consequence of integrality, $cx \leq x$ and $xc \leq x$, hence by residuation $x \leq c\backslash x =\gell(x)$ and $x \leq x/c=\gr(x)$. By order preservation, if $x \leq y$, then $\gell(x)=c \backslash x \leq c \backslash y=\gell(y)$ and $\gr(x)=x/c \leq y/c =\gr(y)$. Moreover, by (2) of Lemma \ref{lemma: propblockwise gluing}, we get that $\gell(\gell(x))=c \backslash (c \backslash x)=c^2 \backslash x= c \backslash x= \gell(x)$ and $\gr(\gr(x))=(x /c)/c=x /c^2=x/c=\gr(x)$. 
		
(5) The fact that, for all $x,y \in B- \{1\}$, $xy \leq \sr(x), \sell(y)$ and $yx\leq \sr(y), \sell(x)$, follows by the order preservation of the product in $\alg A$. Indeed, $xy \leq xc=\sr(x)$, $xy \leq cy=\sell(y)$, $yx\leq cx=\sell(x)$ and $yx \leq yc=\sr(y)$. 

Hence $(\alg B, \sell, \sr, \gell, \gr)$ is a weak lower block. It is left to show that it is a lower block. Suppose that $x \backslash y$ is undefined in $\alg B$, then $x \backslash y \in C - \{1\}$. Now, $\sr(x) = x(x\backslash y) \leq y$, and also $x \not \leq y$ since otherwise $x \backslash y$ would be defined in $\alg B$ and equal to $1$.  Similarly one can prove that if $y / x$ is undefined then $\sell(x)\leq y$ and $x \not \leq y$. This concludes the proof.
  \end{proof}
	We now demonstrate how one can reconstruct a blockwise gluing starting from its abstract components. 
	
	\begin{definition}
		Let $(\alg B, \sell, \sr, \gell, \gr)$ be a lower block, and $\alg C$ be an IRL such that $B \cap C = \{1\}$. We assume that $\alg C$ has a splitting coatom $c_{\alg C}$ in case there are $x,y \in B - \{1\}$ such that $x \not\leq y$, and $\sell(x) \leq y$ or $\sr(x) \leq y$. Moreover, if $x \lor y = 1$ in $\alg B$ for some $x, y \in B - \{1\}$, then $\alg C$ has a bottom element $\perp_{\alg C}$. 
		Let $\beta = (\sell, \sr, \gell, \gr)$, we define $\alg B \boxplus_\beta \alg C$ to be the structure having domain $	B \cup C$ and with the operations
		defined as follows:
		\begin{align*}
x y &=\left\{\begin{array}{ll}
x \cdot^{\alg B} y &\mbox{ if } x, y \in B\\
x \cdot^{\alg C} y &\mbox{ if } x, y \in C\\
\sell (y) & \mbox{ if } y \in B - \{1\}, x \in C - \{1\}\\
\sr(x) & \mbox{ if } x \in B - \{1\}, y \in C - \{1\}\end{array}\right.\\
x\back\, y &=\left\{\begin{array}{ll} 
x \back^{\alg C} y  &\mbox{ if } x, y \in C\\
x \back^{\alg B} y &\mbox{ if } x, y \in B\mbox{ and } \sr(x)\not \leq y \mbox { or } x \leq y\\
c_{\alg C} & \mbox{ if } x, y \in  B-\{1\} \mbox { and } \sr(x)\leq y \mbox { and } x \not \leq y\\ 
\gell(y) & \mbox{ if } x \in C - \{1\}, y \in B - \{1\} \\
1 & \mbox{ if } x \in B - \{1\}, y \in C\\ \end{array}\right.\\
y/\, x &=\left\{\begin{array}{ll} 
y/\,^{\alg C} x  &\mbox{ if } x, y \in C\\
y/\,^{\alg B} x &\mbox{ if } x, y \in B\mbox{ and } \sell(x)\not \leq y \mbox{ or } x \leq y\\
c_{\alg C}& \mbox{ if } x, y \in  B -\{1\} \mbox{ and } \sell(x) \leq y \mbox{ and } x \not \leq y\\
\gr(y) & \mbox{ if } x \in C - \{1\}, y \in B - \{1\} \\
1 & \mbox{ if } x \in B - \{1\}, y \in C.\\\end{array}\right.\\
x \land y &=\left \{\begin{array}{ll}
x \land^\alg B y & \mbox{ if } x, y \in B\\
x \land^\alg C y &\mbox{ if } x, y \in C\\
x & \mbox{ if } x \in B - \{1\}, y \in C\\
y & \mbox{ if } y \in B - \{1\}, x \in C\end{array}\right .\\
x \lor y &=\left \{\begin{array}{ll}
x \join^\alg B y &\mbox{ if } x, y \in B  \mbox{ and $x \join y$ exists in } B\\
\perp_{\alg C} &\mbox{ if }   x, y \in B-\{1\} \mbox{ and $x \join y = 1$ in } B\\
x \join^\alg C y &\mbox{ if } x, y \in C\\
x & \mbox{ if } y \in B -\{1\}, x \in C\\
y & \mbox{ if } x \in B -\{1\}, y \in C\end{array} \right.
\end{align*}
We call $\alg B \boxplus_\beta \alg C$ the {\em blockwise gluing of $\alg B$ and $\alg C$ via $\beta = (\sell, \sr, \gell, \gr)$}.
	\end{definition}
	\begin{proposition}\label{prop: proving to be IRL}
		$\alg B \boxplus_\beta \alg C$ is an IRL .
	\end{proposition}
	\begin{proof}
		It is straightforward to check that $(\alg B \boxplus_\beta \alg C, \lor, \land, 1)$ is a lattice with top element $1$. Consider $(B \cup C, \cdot, 1)$, we show that it is a monoid. In particular we prove that the product is associative, i.e. for all $x, y, z$ $x(yz)=x(yz)$. First note that if $x,y,z$ all belong  to the same component, then associativity holds since $\alg B$ is a partial IRL with defined products, and $\alg C$ is an IRL. Thus, we need to consider the cases in which one of them belongs to a different component.
		\begin{enumerate}
			\item If $x,y \in B$ and $z \in C$, then:  $(xy)z=\sr(xy)=x\sr(y)=x(yz)$ where the second equality holds by Definition \ref{def: lower block} (3a). 
				
			\item If $x,y \in C$ and $z \in B$, then: 
			$(xy)z=\sell(z)=\sell(\sell(z))=x(yz)$, where the second equality holds since $\sell$ is an interior operator by Definition \ref{def: lower block} (3). 				
			\item If $x,z \in B$ and $y \in C$, then: 
			 $(xy)z=\sr(x)z=x\sell(z)=x(yz)$, where the second equality holds by Definition \ref{def: lower block} (3b). 
				
			\item If $x,z \in C$ and $y \in B$, then: 
			 $(xy)z=\sell(y)z=\sr(\sell(y))=\sell(\sr(y))=x(yz)$,where the third equality holds by Definition \ref{def: lower block}(3c); 
				
				\item If $x \in B$ and $y,z \in C$, then: 
				 $(xy)z=\sr(\sr(x))=\sr(x)=x(yz)$, where the second equality holds since $\sr$ is an interior operator by Definition \ref{def: lower block} (3); 
					
				\item If $x \in C$ and $y,z \in B$, then: 
				 $(xy)z=\sell(y)z=\sell(yz)=x(yz)$,where the second equality holds by Definition \ref{def: lower block} (3a); 
			\end{enumerate}
	It is left to show that the residuation law holds, i.e., for all $x,y,z$ in the domain: $$xy \leq z \mbox{ if and only if } y \leq x \backslash z \mbox{ if and only if }x \leq z/y.$$ 
	Notice that if at least one of the elements is $1$, residuation is easily seen to hold. We then need to consider the following different cases, in which we assume $x, y, z$ all different from $1$. 
	\begin{enumerate}
	\item If $x,y,z \in C$, then the residuation law holds since $\alg C$ is an IRL.
		\item Let $x,y,z \in B$, we verify that $xy \leq z$ if and only if $y \leq x \backslash z$, the proof for the remaining equivalence being similar. If $\sr(x) \not\leq z$ or $x \leq z$, then $x \backslash z$ is defined in $\alg B$ by the definition of a lower block; hence the equivalence holds since $\alg B$ is a partial IRL. If instead $\sr(x) \leq z$ and $x \not\leq z$,	 then $y \leq x \backslash z=c_{\alg C}$ and $xy \leq \sr(x)\leq z$, which follows from the Definition \ref{def: lower block} (5) of a lower block.
		\item Let $x,y \in B$ and $z \in C$. Then $xy \in B$, hence $xy \leq z$, $y \leq x \backslash z = 1$, and $x \leq z/y = 1$. Hence the three inequalities all hold. The same happens if $x,z \in C$ and $y \in B$ and if $x \in B$ and $y,z \in C$.
		\item Let $x,y \in C$ and $z \in B$. Then $xy \in C$ while $x \backslash z, z / y \in B$, hence $xy \not\leq z$, $y \not\leq x \backslash z$, and $x \not\leq z/y$. Thus, the three inequalities are all false. 
		\item Let $x,z \in B$ and $y \in C$. Let us start by the first equivalence. Suppose that $xy=\sr(x)\leq z$, then either $x \not\leq z$ and $x \backslash z= c_{\alg C}\geq y$, or $x \leq z$ and then $x \backslash z=1 \geq y$. Conversely, suppose that $y \leq x\backslash z$, then $x\backslash z$ is not in $B-\{1\}$, hence either $xy = \sr(x)\leq z$ or $xy = \sr(x) \leq x \leq z$.
	
		For the other equivalence, notice that since $(\sr, \gr)$ is a residuated pair, $xy = \sr(x)\leq z$ iff $x \leq \gr(z)=z/y$.  
		\item Let $x \in C$ and $y,z \in B$. Then residuation can be proven similarly to the previous step, taking into account that $(\sell, \gell)$ forms a residuated pair, and the definition of the divisions in $\alg B \boxplus_\beta \alg C$.
     \end{enumerate}		
     
	We can conclude that $\alg B \boxplus_\beta \alg C$ is an IRL. 
	\end{proof}
	Recall that given a blockwise gluing $\alg A$ of $\alg B$ and $\alg C$, we defined, fixing any $c \in C - \{1\}$, $\sell_{\alg A}(x) = cx,$ $\sr_{\alg A}(x) = xc$, $\gell_{\alg A}(x)= c \backslash x$, $\gr_{\alg A}(x) = x / c$.	
	\begin{theorem}\label{thm:blockwise}
		If $\alg A$ is a blockwise gluing of $\alg B$ and $\alg C$, then $\alg A = \alg B\boxplus_\beta \alg C$, with $\beta = (\sell_{\alg A}, \sr_{\alg A}, \gell_{\alg A}, \gr_{\alg A})$.
	\end{theorem}
	\begin{proof}
If $\alg A$ is a blockwise gluing of $\alg B$ and $\alg C$, then $(\alg B, \sell_{\alg A}, \sr_{\alg A}, \gell_{\alg A}, \gr_{\alg A})$ is a lower block by Proposition \ref{prop: lower blocks with c}. Given that $B \cap C = \{1\}$ by definition, and given Lemma \ref{lemma: propblockwise gluing}, one can define $\alg B \boxplus_\beta \alg C$. Notice that the domains of $\alg B \boxplus_\beta \alg C$ and $\alg A$ coincide; let us verify that the operations of $\alg B \boxplus_\beta \alg C$ coincide with those of $\alg A$ as well. 
This is clear for the lattice operations. For the product, it follows from the fact that $\alg B$ and $\alg C$ are $\cdot$-subreduct of both $\alg A$ and $\alg B \boxplus_\beta \alg C$, and the observation that products between elements in $B$ and $C$ coincide in both algebras with: $$bc = \sr_{\alg A}(b),\;\; cb = \sell_{\alg A}(b),$$ for any $b \in B - \{1\}$ and $c \in C-\{1\}$. 

Let us finally verify that the divisions coincide. If $x,y \in C$, there is nothing to prove since $\alg C$ is a subalgebra of both  $\alg B \oplus_\beta \alg C$ and $\alg A$. Let then $x, y \in B$ and consider $x \backslash y \in A$, then either $x \backslash y$ is in $B$ or it is in $C -\{1\}$. If $x \backslash y$ is in $B$, then either $x \leq y$ or there is no element $c \in C$ such that $xc \leq y$. In the first case, $x \backslash y = 1$ in both structures; in the latter case, necessarily $\sr_{\alg A}(x) \not\leq y$ hence  $x \backslash y$ is again the same in both structures given the definition of the division in $\alg B \boxplus_\beta \alg C$. If $x \backslash y$ is in $C -\{1\}$ in $\alg A$, then by residuation $\sr_{\alg A}(x) = xc \leq y$ and then $x \backslash y = c_{\alg C}$ both in $\alg A$ and in $\alg A=\alg B \boxplus_\beta \alg C$, given respectively Lemma \ref{lemma: propblockwise gluing} (1a) and the definition of the division in $\alg B \boxplus_\beta \alg C$.
One can show in an analogous way that also $y/x$ in $\alg B \boxplus_\beta \alg C$ coincide with the one of $\alg A$. 
Lastly, the case $x \in B$ and $y \in C$ and viceversa follow from Proposition \ref{prop: lower blocks with c}
 and the definition of the divisions in $\alg B \boxplus_\beta \alg C$.
		
		We can conclude that  $\alg A=\alg B \boxplus_\beta \alg C$.
\end{proof} 
Observe that, unlike the 1-sum construction, the blockwise gluing is not an obviously associative operation, because one has to define each time the operators for the lower block. We hence introduce an iterated version of the blockwise construction.

The idea is to consider a family of (possibly partial) algebras $(\alg B_{i})_{i \in I}$ indexed by some totally ordered set $(I; \leq)$, and a tuple of operators $(\sell, \sr,\gell, \gr)$ for each pair of indices $i,j \in I$, in order to describe products and divisions among the two components. 
\begin{definition}\label{def: (iterated) blockwise gluing}
Let ${\bf I}=(I,\leq)$ be a totally ordered index set, and let $\{\alg B_i\}_{i \in I}$ be a set of partial IRLs. We call a family $\{(\alg B_{i}, \betaij)\}_{i <_{\II} j}$ with $\betaij = (\Sell i j, \Sr i j , \Gell i j, \Gr i j)$ a \emph{family of blocks} if: 
	\begin{enumerate}
		\item $B_{i} \cap B_{j} = \{1\}$. 
		\item $(\alg B_{i}, \betaij)$ is a weak lower block for all $i,j \in I, i < j$. 
		\item For all $i \in I$, and $x,y \in B_i$ such that $x \not\leq y$, if there exists an index $j > i$ such that $\Srx i j x \leq y$ (or $\Sellx i j x \leq y$) then there is a maximum index  $m_r(x,y) \in I$ (or $m_\ell(x,y) \in I$) with such property.
		\item For all $i \in I$, and $x,y \in B_i$:
		\begin{enumerate}
	\item if $x \backslash y$ is undefined then $x \not \leq y$ and there exists a maximum index $m_r(x,y) \in I$ such that $\Srx i {m_r(x,y)} x \leq y$; in this case $\alg B_{m_r(x,y)}$ has a splitting coatom $c_{m_r(x,y)}$;
		\item if $y / x$ is undefined then there exist a maximum index $m_\ell (x,y) \in I$ such that $\Sellx i {m_\ell(x,y)} x \leq y$; in this case $\alg B_{m_\ell(x,y)}$ has a splitting coatom $c_{m_\ell(x,y)}$. 
		\end{enumerate}
		\item For any $x \in B_k$, and $k < j < i$ we have:
		\begin{enumerate}
		\item $\Sellx k i {\Sellx k j x} = \Sellx k j x = \Sellx k j {\Sellx k i x}$;
		\item $\Srx k i {\Srx k j x} = \Srx k j x = \Srx k j {\Srx k i x}$;
		\item $\Sellx k i {\Srx k j x} = \Srx k j {\Sellx k i x}$;
		\item $\Sellx k j {\Sr k i x} = \Srx k i {\Sellx k j x}$.
		\end{enumerate}
		\item If $i < j < k$ in $\alg I$, then $\Srx i j x \leq \Srx{i}{k}{x}$ and $\Sellx i j x \leq \Sellx{i}{k}{x}$. 
		\item If $x, y \in B_i$ are such that $x \lor y = 1$ for some $i \in I$ that is not the maximum in $I$, then there exist $j > i$ that covers $i$ in ${\bf I}$ such that $\alg B_j$ has a bottom element $\bot_j$.
		\item  If there is a top element $i_\top \in I$, $\alg B_{i_\top} $ is a total algebra and we add to the family the pair $(\alg B_{i_\top}, \beta_{i_\top \infty})$ with $\beta_{i_\top\infty} = ({\rm id},{\rm id},{\rm id},{\rm id})$.
	\end{enumerate}   
Given such a family of blocks, we define the \emph{(iterated) blockwise gluing} $\boxplus_{\bf I}(\alg B_{i}, \betaij)$ to be the algebra with domain $\bigcup_{i \in I} B_i$
and operations defined as follows:
\begin{align*}
x \cdot y &=\left\{\begin{array}{ll}
x \cdot^{\alg B_i} y & \mbox{ if } x, y \in B_i\mbox{ for some } i \in I\vspace{0.1cm}\\
\Srx i j x & \mbox{ if } x \in B_i - \{1\}, y \in B_j \mbox{ and } i < j\vspace{0.1cm}\\
\Sellx j i y & \mbox{ if } x \in B_i - \{1\}, y \in B_j \mbox{ and } j < i
\end{array}\right.\\
x\back\, y &=\left\{\begin{array}{ll} 
x\back^{\alg B_i} y & \mbox{ if } x, y \in  B_i, \mbox { and } x \leq y \mbox{ or there is no }  k\in I 	\mbox{ such that }  \Srx i {k} x \leq y\vspace{0.1cm}\\ 
c_{m_r(x,y)} & \mbox{ if } x, y \in  B_i, x \not \leq y, \mbox{ and }  \Srx i {m_r(x,y)} x \leq y\vspace{0.1cm}\\ 
\Gellx j i y  & \mbox{ if } x \in B_i -\{1\} , y \in B_j \mbox{ and } j < i \vspace{0.1cm}\\
1 & \mbox{ if } x \leq y\\ \end{array}\right.\\
y/\, x &=\left\{\begin{array}{ll} 
y/\,^{\alg B_i} x &\mbox{ if } x, y \in  B_i, \mbox { and } x \leq y \mbox{ or there is no } k \in I  \mbox{ such that } \Sellx i {k} x \leq y \vspace{0.1cm}\\ 
c_{m_\ell(x,y)}& \mbox{ if } x, y \in  B_i, x \not \leq y, \mbox{ and } \Sellx i {m_\ell(x,y)} x \leq y \vspace{0.1cm}\\
\Grx j i y & \mbox{ if } x \in B_i - \{1\}, y \in B_j \mbox{ and } j < i \vspace{0.1cm}\\
1 & \mbox{ if } x \leq y\\\end{array}\right.\\
x \meet y &= \left\{\begin{array}{ll}
x \meet_{\alg B_i} y & \mbox{ if } x, y \in B_i\\
x & \mbox{ if } x \in B_j \mbox{ and } y \in B_i \mbox{ with } j < i\\
y & \mbox{ if } x \in B_j \mbox{ and } y \in B_i \mbox{ with } i < j \vspace{0.1cm}\\
\end{array}\right.\\
x \join y &= \left\{\begin{array}{ll}
\bot_j &\mbox{ if }  x, y \in B_i \mbox{ with } i \neq \max\alg I, x \join_{\alg B_i} y=1, j \mbox{ covers } i \\
x \join_{\alg B_i} y & \mbox{ if } x, y \in B_i, x \join_{\alg B_i} y \neq 1\\
x & \mbox{ if } x \in B_i \mbox{ and } y \in B_j \mbox{ with } j < i\\
y & \mbox{ if } x \in B_i \mbox{ and } y \in B_j \mbox{ with } i < j \vspace{0.1cm}\\
\end{array}\right.
\end{align*}
\end{definition}
\begin{theorem}
	Given a family of blocks $\{(\alg B_{i}, \betaij)\}_{i <_{\II} j}$, their blockwise gluing $\boxplus_{\bf I}(\alg B_{i}, \betaij)$ is an IRL.
\end{theorem}
\begin{proof}
	First, as a straightforward consequence of the definitions, $\boxplus_{\bf I}(\alg B_{i}, \betaij)$ has a lattice reduct with top element $1$. Let us then consider $(\bigcup_{i \in { I}} B_i, \cdot, 1)$, we proceed to show that it is a monoid. In particular, we verify that the product is associative. Notice that if $x,y,z$ belong to the same component $B_i$, then associativity holds since $\alg B_i$ is a partial IRL in which products are always defined. If two among $x,y,z$ belong to the same component, then associativity can be proved as in Proposition \ref{prop: proving to be IRL}. Thus, let $x,y,z$ be in three different components $\alg B_i$, $\alg B_j$, and $\alg B_k$ with $i>j>k$, then we need to consider the following cases. 
	\begin{enumerate}
		\item Let $x \in B_i, y \in B_j$ and $z \in B_k$, then: 
		 $(xy)z=\Sellx j i y z=\Sellx k j z$ and $x(yz)=x \Sellx k j z=\Sellx{k}{i}{\Sellx{k}{j}{z}}$; by Definition \ref{def: (iterated) blockwise gluing} (5a), we get $(xy)z=x(yz)$.		
		\item Let $y \in B_i, x \in B_j$ and $z \in  B_k$, then: 
	$(xy)z=\Srx j i x z=\Sellx k j z$ and $x(yz)= x \Sellx k i z=\Sellx k j {\Sellx k i z}$; by Definition \ref{def: (iterated) blockwise gluing} (5a), we get $(xy)z=x(yz)$. 
	
		\item Let $x \in  B_i, z \in B_j$ and $y \in B_k$, then: 
		 $(xy)z=\Sellx k i y z=\Srx k j {\Sellx k i y}$ and $x(yz)= x \Srx k j y=\Sellx k i {\Srx k j y}$; by Definition \ref{def: (iterated) blockwise gluing} (5c), we get $(xy)z=x(yz)$;
		
		\item Let $z \in  B_i, x \in B_j$ and $y \in B_k$, then: 
		 $(xy)z=\Sellx k j y z= \Srx k i {\Sellx k j y}$ and $x(yz)= x \Srx k i y=\Sell k j ({\Srx k i y})$; thus, by Definition \ref{def: (iterated) blockwise gluing} (5d), we get $(xy)z=x(yz)$; 
		
		\item Let $z \in B_i, y \in B_j$ and $x \in B_k$, then: 
		 $(xy)z=\Srx k j x z=\Srx k i {\Srx k j x}$ and $x(yz)= x \Srx j i y= \Srx k j x $; by Definition \ref{def: (iterated) blockwise gluing} (5b), we get $(xy)z=x(yz)$; 
		
		\item Let $y \in  B_i, z \in  B_j$ and $x \in B_k$, then: 
		 $(xy)z=\Srx k i x z=\Srx k j {\Srx k i x}$ and $x(yz)=x \Sellx j i z= \Srx k j x$; agin by Definition \ref{def: (iterated) blockwise gluing} (5b), we get $(xy)z=x(yz)$.
\end{enumerate} 
				Let us now verify that the residuation law holds, i.e., for all $x,y,z$ in the domain: $xy \leq z$ iff $y \leq x \backslash z$ iff $x \leq z/y$. If at least two elements are in the same component, the proof is analogous to the one in Proposition \ref{prop: proving to be IRL}, with some adjustments due to the new divisions in the following cases: 
		\begin{enumerate}
			\item Let $x,z \in B_i, y \in B_j$ and $i < j$ in $\alg I$. If $xy = \Srx i j x \leq z$, then either $x \leq z$ and then $y \leq x \backslash z = 1$, or $y \leq x \backslash z = c_{m_r(x,z)}$; indeed necessarily $j \leq m_r(x,z)$, since the latter is the largest index $m$ such that $\Srx i m x \leq z$. Conversely, if $y \leq x \backslash z$, then either $x \leq z$ or $\Srx i {m_r(x,z)} x \leq z$. In the former case, we get $xy = \Srx i j x \leq x \leq z$, since $\sigma^r_{ij}$ is an interior operator. In the latter, $y \leq x \backslash z = c_{m_r(x,z)}$, and therefore $j \leq m_r(x,z)$ which implies that $xy = \Srx i j x \leq \Srx i {m_r(x,z)} x \leq z$ by Definition \ref{def: (iterated) blockwise gluing}(6). For the other equivalence, $xy = \Srx i j x \leq z$ iff $x \leq \Grx{i}{j}{z} = z / y$ by the fact that $(\sigma^r_{ij}, \gamma^r_{ij})$ is a residuated pair.
			\item Let $x \in B_j, y,z \in B_i$ and $i < j$ in $\alg I$. Then residuation can be shown analogously to the case above.
		\end{enumerate}
		Finally, if $x,y,z$ are in three different components, then residuation can be easily seen to hold as a consequence of the definitions and the order. In particular, let $i < j < k$, then the three inequalities are verified in the following cases: $x \in B_i, y \in B_j, z \in B_k$; $x \in B_i, z \in B_j, y \in B_k$; $y \in B_i, x \in B_j, z \in B_k$; $y \in B_i,z \in B_j, x \in B_k$. The three inequalities are all invalid in the remaining cases: $ z \in B_i, x \in B_j, y \in B_k$ and $z \in B_i, y \in B_j, x \in B_k$.

		We can conclude that $\boxplus_{\bf I} (\alg B_i, \betaij)$ is an IRL.
\end{proof}
We will now show that the class of IRLs which can be characterized by the previous construction is the desired one, in which algebras can be seen as made by components stacked on top of each other, glued at $1$, and multiplying in blocks.
\begin{definition}\label{def:blockwiseclass}
	We define  $\BlG \sse \mathsf{IRL}$ to be the class of IRLs $\alg A$ such that:
\begin{enumerate}
	\item the domain of $\alg A$ can be seen as the union of a set $\{(B_i)\}_{i \in I}$ for some totally ordered index set ${\bf I} = (I, \leq)$, with $B_i \cap B_j = \{1\}$.
	\item The elements of $B_i- \{1\}$ are strictly below the elements of $B_j$ iff $i < j$.
	\item Each $B_i$, for $i \in I$, is the domain of a partial IRL $\alg B_i$ with the inherited operations from $\alg A$, where all operations are defined except possibly some divisions and some joins.
	\item if $x, y \in B_i- \{1\}, z \in B_j$ and $j < i$, then $xz = yz$ and $zx = zy$.
\end{enumerate} 

\end{definition}

\begin{proposition}\label{proposition: a in BG}
$\alg A \in \BlG$ if and only if it is a blockwise gluing of some family of blocks. 
\end{proposition}
\begin{proof}
	Let $\alg A \in \BlG$, we will prove that it is a blockwise gluing of its components $\{\alg B_i \}_{i \in I}$, with suitably defined operators $\betaij$. If $\alg A \in \BlG$, then its domain can be seen as $\bigcup_{i \in { I}} B_i$, where ${\bf I}$ is a totally ordered index set. If there is a top element $i_\top \in I$, we set $\sell_{i_\top \infty} = \sr_{i_\top \infty} = \gell_{i_\top \infty} = \gr_{i_\top \infty} = {\rm id}$. Now, recalling that if $x, y \in B_j$ and $z \in B_i$ with $i < j$, then $xz=yz$ and $zx=zy$, we define the rest of the operators $\Sell i j$, $\Sr i j , \Gell i j$ and $\Gr i j$ in the obvious way: given $c_j \in \alg B_j$, we set
	$$
	\Sellx i j x=c_jx \qquad \Srx i j x=xc_j \qquad \Gellx i j x=c_j \backslash x \qquad \Grx i j x=x/c_j
	$$
	for any $x \in B_i- \{1\}$, and all operators are $1$ at $1$. 
	We proceed to prove that $\{(\alg B_i, \Sell i j, \Sr i j, \Gell i j, \Gr i j)\}_{i <_{\alg I} j}$ is a family of blocks according to Definition \ref{def: (iterated) blockwise gluing}. 
	
	First note, for property (1), that $B_i \cap B_j= \{1\}$ by definition of the class $\BlG$. For (2), one can show that $(\alg B_i, \Sell i j, \Sr i j, \Gell i j, \Gr i j)$ is a weak lower block for all $i, j \in I$ with $i < j$, following the proof of Proposition \ref{prop: lower blocks with c}. 
	
	Let us check (3), i.e., for all $i \in I$, and $x,y \in B_i$ such that $x \not\leq y$, if there exists an index $j > i$ such that $\Srx i j x \leq y$ then there is a maximum index  $m_r(x,y) \in I$ with such property (the proof for the left operators being analogous). Notice that if $x c_j = \Srx i j x \leq y$, we get by residuation that $c_j \leq x \backslash y$. It follows that the index $m$ such that $x \backslash y \in B_m$ is the desired $m_r(x,y)$. 
	
	We now verify (4a), the proof of (4b) being similar. Let $x, y \in B_i$, for some $i \in I$, such that $x \backslash y \not \in B_i$. Then necessarily $x \not\leq y$, since $1 \in B_i$, and $x \backslash y$ belongs to some $B_j$ with $j > i$. Thus, since 
	$\Srx i j x =x (x\backslash y) \leq y$, as shown in the previous point there is a maximum index $m_r(x,y) \in I$ such that $\Srx i {m_r(x,y)} x \leq y$; moreover, since all the elements in $B_{m_r(x,y)}$ multiply the same with respect to the ones in $B_i$, $x \backslash y$ is the splitting coatom of $\alg B_{m_r(x,y)}$.
		
Notice that properties (5) and (6) follow straightforwardly from, respectively, associativity and order preservation of the product in $\alg A$.

For (7), suppose that given $x,y \in \alg B_i$, $x \join y$ is not defined in $\alg B_i$. Then, since $\alg B_i$ is strictly below any $\alg B_j $ with $j >i$, necessarily $x \join y=\perp_j$ where $\perp_j$ is the minimum of the component $\alg B_j$, with $j$ covering $i$. 

Finally, for (8), if there is a top element $i_\top \in I$, $\alg B_{i_\top} $ necessarily is a total algebra, and we already defined $\beta_{i_\top\infty} = ({\rm id},{\rm id},{\rm id},{\rm id})$.

Thus, given $\alg A \in \BlG$, one can define the blockwise gluing $\boxplus_{\bf I}(\alg B_i, \betaij)$ where the $\alg B_i$s are the components of $\alg A$ and the $\betaij$s are defined as above. It is left to show that $\alg A= \boxplus_{\bf I}(\alg B_i, \betaij)$, that is to say, the operations of $\alg A$ coincides with the ones in $\boxplus_{\bf I}(\alg B_i, \betaij)$; this is a straightforward consequence of the definitions. Hence we can conclude that $\alg A= \boxplus_{\bf I}(\alg B_i, \betaij)$. 

For the converse, the fact that a blockwise gluing of a family of blocks is in $\BlG$ is a direct consequence of the definition.
\end{proof}
\begin{remark}
	Proposition \ref{proposition: a in BG} shows that the iterated construction of blockwise gluing characterizes the class of algebras in $\BlG$, which are intuitively gluings of some subcomponents whose products behave in blocks. We wish to point out that the relationship between algebras and family of blocks is not one-one. Precisely, different families of blocks can be used to construct the same algebra. This is due to the fact that some divisions present in the blocks can be ``forgotten'' in the blockwise gluing: it is possible that for $x,y \in B_i$, $x \backslash_{\alg B_i} y$ is defined in $B_i$, but there is a (maximum) $m \in I$ such that $\Srx i m x \leq y$ and hence $x \back y = c_m$ in the gluing. Therefore, substituting $\alg B_i$ with a variation of it, where $x \backslash_{\alg B_i} y$ is not defined, would yield the same blockwise gluing.
	This is more of a technicality than a conceptual point, but it allows one to use total algebras to construct a gluing, since it does not force the initial blocks to be partial algebras even when some of their divisions will be forgotten in the construction. 
\end{remark}

\subsection{Preservation of identities and $\vv{HSP_u}$}
In this subsection we investigate the interaction of the blockwise gluing construction with class operators and equational properties. For simplicity of exposition, we restrict to the non-iterated construction. This investigation is inspired by the analogous ones for 1-sums \cite{AM2003} and gluings of residuated lattices \cite{GalatosUgolini}.

\subsubsection{Preservation of identities} We would like to identify which kind of equations are preserved by the blockwise gluing construction. In particular, we characterize the cases in which commutativity, divisibility, and semilinearity are preserved. Observe that linearity is obviously maintained by the construction. 

Let $\alg B$ be a partial integral residuated lattice. We call it semilinear (or divisible) if it satisfies the semilinearity (or divisibility) identity whenever all the terms involved are defined.
 \begin{proposition}
	Let $\alg B \boxplus_\beta \alg C$ be a blockwise gluing, with $\beta =(\sell, \sr, \gell, \gr)$. Then:
	\begin{enumerate}
		\item $\alg B \boxplus_\beta \alg C$ is commutative if and only if $\alg B$ and $\alg C$ are commutative and $\sell=\sr$. 
		\item $\alg B \boxplus_\beta \alg C$ is divisible if and only if $\alg C$ is divisible, $\alg B$ is divisible, and the blockwise gluing coincides with the 1-sum. 
		\item $\alg B \boxplus_\beta \alg C$ is semilinear if and only if $\alg B$ is linear, $\alg C$ is semilinear, and either $\alg C$ is linear or the gluing is a 1-sum.
	\end{enumerate}
\end{proposition}

\begin{proof}
	\begin{enumerate}
		\item Note that the right-to-left direction is obvious. For the left-to-right, assume $\alg B \boxplus_\beta \alg C$ is commutative; necessarily $\alg B$ and $\alg C$ are commutative since they are subreducts under products. Moreover, for any  $x \in B$ and $y \in C$, $\sr(x)=x \cdot y = y \cdot x = \sell (x)$, by definition of the operations and by the fact that $\alg B \boxplus_\beta \alg C$ is commutative.
		\item Note again that the right-to-left direction is obvious. For the left-to-right direction, assume $\alg B \boxplus_\beta \alg C$ is divisible; then necessarily $\alg C$ is divisible and $\alg B$ is divisible. Moreover, since $\alg B \boxplus_\beta \alg C$ is divisible, for any $x \in B$ and $y \in C-\{1\}$, $$x=x \meet y = y \cdot (y \backslash x)= \sell(\gell(x)).$$ From the fact that $x= \sell(\gell(x)$, one gets $$\sell(x)=\sell(\sell(\gell(x)))=\sell(\gell(x))=x.$$ Similarly, one can show that also $\sr$ coincides with the identity map, yielding that $\alg B \boxplus_\beta \alg C$ is the 1-sum of $\alg B$ and $\alg C$.
		\item Suppose that $\alg B \boxplus_\beta \alg C$ is semilinear, then $\alg C$ is also semilinear since it is a subalgebra of the gluing. We now show that $\alg B$ is necessarily linear. By way of contradiction, consider two incomparable elements $x,y \in B$, i.e. such that $x \not \leq  y$ and $y \not \leq  x$. Then $x \backslash y \neq 1$ and $y \backslash x \neq 1$. Notice that by the definition of the operations in the gluing, either $x \backslash y \in B-\{1\}$ or $x \backslash y = c_{\alg C}$, the coatom of $\alg C$; the same holds for $y \backslash x$. Thus, in any case, $(x \backslash y) \join (y \backslash x) < 1$. But since $\alg B \boxplus_\beta \alg C$ is semilinear, and so in particular prelinear, $(x \backslash y) \join (y \backslash x)=1$ yielding a contradiction. Therefore, $\alg B$ is necessarily linear. For the last part, suppose that $\alg C$ is not linear. Then, there are incomparable elements in $\alg C$, say $x, y \in C$ such that $x \not \leq  y$ and $y \not \leq  x$. We prove that, necessarily, both operators $\sell$ and $\sr$ must coincide with the identity map on $\alg B$, or in other words, the gluing is a 1-sum, which will prove the claim. Suppose that there is an element $u \in B$ such that $\sell(u) < u$. Take then any $v \in B$, given that $\alg B \boxplus_\beta \alg C$ is semilinear it holds that:
		$$(u\backslash(y \backslash x)u) \lor  (v(x \backslash y )/ v)=1.$$
	Notice that $u\backslash(y \backslash x)u = u \backslash \sell(u)$ and $v(x \backslash y )/ v = \sr(v) / v$. Since $\sell(u) < u$, if also $\sr(v) < v$ then both $u \backslash \sell(u)$ and $\sr(v) / v$ are either elements of $B - \{1\}$ or they coincide with the splitting coatom of $\alg C$. In any case, their join is strictly smaller than $1$, a contradiction. Hence necessarily $\sr(v) = v$, for all $v \in B$. This proves that at least one among $\sr$ and $\sell$ is the identity map. We proceed to prove that they must coincide. Indeed, by way of contradiction suppose, without loss of generality, that $\sr$ is the identity map but $\sell$ is not: i.e., there is $u \in B-\{1\}$ such that $\sell(u) < u$. Consider again the incomparable elements of $\alg C$, $x$ and $y$. An application of semilinearity yields that:
	$$(u\backslash(y \backslash x)u) \lor  (x(x \backslash y )/ x)= (u \backslash \sell(u)) \lor (x(x \backslash y )/ x) =1.$$
	Since $\sr(u) = u > \sell(u)$, by the definition of the operations in the gluing we get that $u \backslash \sell(u) \in B-\{1\}$. Hence necessarily, for the above join to be $1$, $x(x \backslash y )/ x = 1$. But then by residuation $x \leq x(x \backslash y ) \leq y$, a contradiction since $x$ and $y$ are incomparable. Hence, there can be no element $u \in B-\{1\}$ such that $\sell(u) < u$, which proves that also $\sell$ is the identity map. This completes the left-to-right direction.
		
		For the converse, notice first that if both $\alg B$ and $\alg C$ are linear, then the gluing is linear itself and hence semilinear. Suppose now that $\alg B$ is linear, $\alg C$ is semilinear, and the gluing is a 1-sum, i.e., $\sell = \sr = {\rm id}$.
		We verify that $\alg B \boxplus_\beta \alg C$ satisfies, for all $x,y,u,v$ in the domain, the semilinearity identity: 
		\begin{equation*}
	        (u\backslash(y \backslash x)u) \lor  (v(x \backslash y )/ v)=1.
       \end{equation*}
       We distinguish the following cases: 
       \begin{enumerate}
       	\item If $x, y, u,v \in C$, then the identity holds since $\alg C$ is semilinear.
       	\item If $x$ and $y$ are comparable, then either $x \leq y$ or $y \leq x$  which implies that the equation holds since either $u\backslash(y \backslash x)u = u \backslash u = 1$ or $v(x \backslash y )/ v = v / v = 1$, hence either way their join is $1$. Notice that this in particular covers the cases where at least one among $x$ and $y$ is in $B-\{1\}$.
       	\item Let now $x, y, v \in C$ with $x$ and $y$ incomparable, and suppose $u \in B-\{1\}$. Then $u\backslash(y \backslash x)u = u \backslash \sell(u) = u \backslash u = 1$ and the identity holds. The proof for $x, y, u \in C$ and $v \in B-\{1\}$ is completely analogous.
       \end{enumerate}
        We covered all the possible cases, hence proving that $\alg B \boxplus_\beta \alg C$ is semilinear.\qedhere
\end{enumerate}
\end{proof}
Let us now observe that, since the upper component of the blockwise gluing is a subalgebra, and the lower component is a subreduct of the blockwise gluing, most of the one-variable identities are preserved. 

\begin{proposition}
	The blockwise gluing $\alg B \boxplus_\beta \alg C$ preserves all the one-variable equations valid both in $\alg B$ and $\alg C$ not involving divisions. In case the gluing is a 1-sum, all the one-variable equations valid both in $\alg B$ and $\alg C$ are preserved. 
\end{proposition}

Considering the previous proposition, we get that the blockwise gluing construction preserves $n$-potency, i.e. $x^n=x^{n+1}$ for $n \geq 1$. It is also important to notice that, in general, the blockwise gluing construction does not preserve all the monoidal identities that are true both in $\alg B$ and $\alg C$. For instance, as shown above, commutativity is only preserved when $\sell=\sr$. In fact: 
 \begin{proposition}
	If  $\sell=\sr$, $\alg B \boxplus_\beta \alg C$ satisfies all the monoidal identities that are true in both $\alg B$ and $\alg C$. 
\end{proposition}
\begin{proof}
	The proof goes as the one of Proposition 5.3 in \cite{GalatosUgolini}.
\end{proof}

\subsubsection{$\vv{HSP_u}$} Constructions such as the blockwise gluing can help one better understand and describe the structure of the algebras. To this end, we now provide a description of the lattice of congruence filters of the blockwise gluing of two algebras, and we characterize when a blockwise gluing is subdirectly irreducible. Moreover, we describe the fundamental algebraic operators of subalgebras, homomorphic images, and ultrapowers of a blockwise gluing.  
 
\begin{definition}
	Let $\alg B$ be a partial integral residuated lattice. A {\em partial filter} $F$ of $\alg B$ is a nonempty upset of $\alg B$ that is closed under existing products and existing conjugates, i.e, for any $a \in B$ and $x, y \in F$, $xy$, $yx$, $ax/a$, and $a\backslash xa$ are in $F$ whenever defined.
\end{definition}

Let us consider the poset of partial filters of $\alg B$ ordered by inclusion; we observe that it is a lattice with intersection as meet, and the smallest partial filter containing the union as join. 
Consider now a lower block $(\alg B, \beta)$ with $\beta = (\sell, \sr, \gell, \gr)$. We call {\em $\beta$-filter} of $(\alg B, \beta)$ a partial filter $F$ such that: for any $b \in B$ such that $\sell(b) \not\leq\sr(b)$ it holds that $\sr(b) /b \in F$, and similarly, for any $b \in B$ such that $\sr(b) \not\leq \sell(b)$ it holds that $b \backslash \sell(b) \in F$. $\beta$-filters also form a lattice, which we denote by $\betafil(\alg B)$. Notice that if $\sell=\sr$, then all partial filters are $\beta$-filters.

We are now ready to describe the lattice of congruence filters of a blockwise gluing. 
\begin{theorem}\label{theorem: description of filters}
	Let $\alg B \boxplus_\beta \alg C$ with $\beta=(\sell, \sr, \gell, \sr)$. Then 
	\begin{enumerate}
		\item if $\sell=\sr$, $\fil(\alg B \boxplus_\beta \alg C)\cong \fil(\alg C) \oplus (\betafil(\alg B) -\{1\})$, where $\oplus$ is the ordinal sum of posets;
		\item if $\sell \neq \sr$, $\fil(\alg B \boxplus_\beta \alg C)\cong\betafil(\alg B)$.
	\end{enumerate}
\end{theorem}
\begin{proof}
	(1)  Let us consider the map $\varphi: \fil(\alg B \boxplus_\beta \alg C) \to \fil(\alg C) \oplus (\betafil(\alg B)-\{1\})$ defined by:
\begin{align*}
	\varphi(F) = &\left\{\begin{array}{l l}
F &\mbox{ if } F \cap B = \{1\}, \vspace{0.1cm}\\
F \cap B &\mbox{ otherwise.} \vspace{0.1cm}
\end{array}\right.
\end{align*}
The map $\varphi$ is well defined. Indeed, let $F \in \dfil(\alg B \boxplus_\beta \alg C)$, then if $F \cap B = \{1\}$, clearly $F \in \dfil(\alg C)$. 
	  If instead there exists an $ x \in F \cap  B$ with $x \neq 1$, then notice that $C \subseteq F$, since $F$ needs to be closed upwards, and it is easy to see that $F \cap  B \in \dbetafil(\alg B)-\{1\}$. 
	  
	  We prove that $\varphi$ is the desired lattice isomorphism. It is a straightforward consequence of the definition that $\phi$ is injective and preserves the order. To prove surjectivity, let $F \in \dfil(\alg C) \oplus (\dbetafil(\alg B)-\{1\})$, then either $F \in \dfil(\alg C)$ or $F \in \dbetafil(\alg B)-\{1\}$. 
	  
	  If $F \in \dfil(\alg C)$, we show that $F$ is also a congruence filter of $\alg B \boxplus_\beta \alg C$, and hence $\varphi(F) = F$. Since $\alg C$ is a subalgebra of the gluing, it suffices to check that $F$ is closed under conjugates with respect to elements of $\alg B$. Suppose there is $x \in F, x \neq 1$, and let $b \in B-\{1\}$. Then $b \backslash xb = b \backslash \sell(b) = c_{\alg C}$ by the definition of the operations, since $\sr(b) = \sell(b)$ by hypothesis. Similarly, $bx/b = c_{\alg C}$. Notice that $x \leq c_{\alg C}$ hence $c_{\alg C} \in F$. Thus, $F$ is a filter of the gluing as desired. 
	  
	  Suppose now $F \in \dbetafil(\alg B)$, $F \neq\{1\}$, and let $F' = F \cup C$. We prove that $F'$ is a congruence filter of the gluing. It suffices to check closure under conjugates involving elements of $\alg B$, since $F'$ is clearly closed under products and upwards, and $\alg C$ is a subalgebra of the gluing. We need to consider the following cases:
	  \begin{itemize}
	  	\item If $x \in C, b \in B$, then $bx/b = \sr(b) / b \in C\sse F'$ by the definition of the operations, since $\sell(b) = \sr(b)$. Similarly $b \back xb \in C \sse F'$.
	  	\item  Let now $x \in F, b \in B$, then $bx/b$ and $b \backslash xb$ are either in $B$, hence in $F$ given that $F$ is a partial filter, or in $C$, hence in $F'$.
	  	\item Consider then $x \in F, c \in C$. Then $cx / c = \sell(x) / c = \gr\sell(x)$. Now, since $\sr(x) = \sell(x)$ by hypothesis, we get $x \leq  \gr\sell(x) \in F\sse F'$, using that $(\sr, \gr)$ is a residuated pair. 
	  \end{itemize}
	   Hence $F'$ is a filter of $\alg B \boxplus_\beta \alg C$, and clearly $\varphi(F') = F$. We have shown that $\varphi$ is an order-preserving bijection. To see that it is a lattice isomorphism, it suffices to observe that, given what we have shown above, the map $\varphi': \fil(\alg C) \oplus (\betafil(\alg B)-\{1\}) \to  \fil(\alg B \boxplus_\beta \alg C)$ defined by:
	    \begin{align*}
	\varphi'(F) = &\left\{\begin{array}{l l}
F &\mbox{ if } F \in \dfil(\alg C), \vspace{0.1cm}\\
F \cup C &\mbox{ if } F \in \dbetafil(\alg B) \vspace{0.1cm}
\end{array}\right.
\end{align*}
	   is the inverse of $\varphi$, and it is also order preserving.
	   This proves that $\varphi$ is the desired lattice isomorphism.
	 
	 (2) Suppose now that $\sell \neq \sr$. First, we notice that every nontrivial filter $F \in  \dfil(\alg B \boxplus_\beta \alg C)$ intersects $B - \{1\}$. Indeed, let $ x\in F-\{1\}$. Since  $F$ is closed under conjugation, for any $y \in B$ and $x \in F $, $y x / y$ and $y \backslash x y $ are in $F$. Recalling that $\sell \neq \sr$, there exists $b \in B$ such that $\sell(b)\neq \sr(b)$, i.e. either $\sell(b)\not \leq \sr(b)$ or $\sr(b)\not \leq \sell(b)$. If $\sell(b)\not \leq \sr(b)$, then, by definition of the operations, $b x / b=\sr(b)/b \in B -\{1\}$. Similarly, if $\sr(b)\not \leq \sell(b)$, then $b \backslash x b=b \backslash \sell(b) \in B-\{1\}$. 

Let us then consider the map $\psi: \fil(\alg B \boxplus_\beta \alg C) \to \betafil(\alg B)$ sending $F \mapsto F \cap B$. To see that $\psi$ is indeed a map from the filters of the gluing to the $\beta$-filters of the lower block $(\alg B, \beta)$, consider any nontrivial $F \in \dfil(\alg B \boxplus_\beta \alg C)$. It is clear that $F \cap B$ is a partial filter of $\alg B$. To see that it is a $\beta$-filter, notice first that, as argued above, any $c \in C$ is in $F$. Thus, by closure under conjugates, for every $b \in B$, both $b \backslash cb = b \backslash \sell(b)$ and $bc / b = \sr(b) / b$ are in $F$. Therefore, if $\sell(b) \not\leq\sr(b)$ we have that $bc / b = \sr(b) /b \in F \cap B$, and similarly, for any $b \in B$ such that $\sr(b) \not\leq \sell(b)$ we have that $b \backslash \sell(b) \in F \cap B$. Hence, $F \cap B$ is a $\beta$-filter.

 We move to show that $\psi$ is the desired lattice isomorphism. As above, $\psi$ is clearly injective and order preserving. To prove surjectivity, Note first that $\psi(\{1\}) = \{1\}$. Now
 let $F \in \dbetafil(\alg B) - \{1\}$, and consider $F' = F \cup C$. We prove that $F' \in \dfil(\alg B \boxplus_\beta \alg C)$, by showing that it is closed under conjugates. We need to consider the following cases:
 \begin{itemize}
  	\item Let first $x \in C, b \in B$, then $bx /b = \sr(b) / b$. If $\sell(y) \leq \sr(y)$ then $\sr(b) / b \in C \sse F'$. If instead $\sell(y) \not\leq \sr(y)$, then $\sr(b) / b \in F$ since $F$ is a $\beta$-filter. The proof that $b \backslash xb \in F$ is similar. 
 	\item Let now $x \in F, b \in B$, then $bx /b$ and $b \backslash xb$ are either in $B$, hence in $F \sse F'$ because $F$ is closed under existing conjugates, or in $C$, hence in $F'$.
 	\item Finally, let $x \in F, c \in C$. Then $cx/c = \sell(x)/c \geq \sell(x) \geq x^2 \in F$, and then $cx/c \in F$. Similarly $c \backslash xc \in F$. 
 \end{itemize}

Thus $F'$ is a congruence filter of the gluing, and since $\psi(F') = F$, we proved surjectivity of $\psi$. It also follows from what we have shown that the map $\psi': \betafil(\alg B) \to \fil(\alg B \boxplus_\beta \alg C)$ defined by $\psi'(\{1\}) = \{1\}$ and $\psi'(F) = F \cup C$ is the inverse of $\psi$, and $\psi'$ is also clearly order-preserving. Therefore, $\psi$ provides the lattice isomorphism showing $\fil(\alg B \boxplus_\beta \alg C)\cong\betafil(\alg B)$.
\end{proof}
As a consequence:
\begin{corollary}
	Let $\alg B \boxplus_\beta \alg C$ be a blockwise gluing with $\beta=( \sell, \sr, \gell, \gr)$, then: 
	\begin{enumerate}
		\item if $\sell=\sr$, then $\alg B \boxplus_\beta \alg C$ is subdirectly irreducible if and only if $\alg C$ is subdirectly irreducible; 
		\item if $\sell \neq \sr$, then $\alg B \boxplus_\beta \alg C$ is subdirectly irreducible if and only if $\betafil( B)$ has a monolith.
	\end{enumerate}
\end{corollary}
Let us now describe the homomorphic images (or equivalently, the quotients) of a blockwise gluing. Given a congruence filter $H$ of a blockwise gluing $\alg B \boxplus_\beta \alg C$, we write $\alg B / H$ for the partial subalgebra of $(\alg B \boxplus_\beta \alg C)/H$ with domain $B/H$. 
\begin{proposition}
	Let $H$ be a congruence filter of a blockwise gluing $\alg B \boxplus_\beta \alg C$ with $\beta = (\sell, \sr, \gell, \gr)$. Let $\beta_H=(\sell_H, \sr_H, \gell_H, \gr_H)$ where $\alpha_H[x]=[\alpha(x)]$ with $\alpha \in \{ \sell, \sr, \gell, \gr\}$.
	Then:
	\begin{enumerate}
		\item  $(\alg B \boxplus_\beta \alg C)/H\cong \alg B / H \boxplus_{\beta_H} \alg C /H$, if $\sell=\sr$;
		\item $(\alg B \boxplus_\beta \alg C)/H\cong \alg B/ H$, if $\sell \neq \sr$.
	\end{enumerate}
\end{proposition}
\begin{proof}
It is easy to check that $(\alg B \boxplus_\beta \alg C)/H$ is a blockwise gluing of $\alg B /H$ and $\alg C/H$ in the sense of Definition \ref{def: block IRL}. Indeed:
\begin{enumerate}
	\item the domain of the gluing is clearly $B/H \cup C/ H$, and $B/H \cap C/ H = \{[1]\}$.
	\item the elements of $B/H - \{[1]\}$ are strictly below those of $C/ H$. Indeed, either $C/H = \{[1]\}$, or $B \cap H = \{1\}$. Hence, given $b \in B-\{1\}$ and $c \in C$ such that $[c] \neq [1]$, $c \backslash b \in B$ and then $c \notin H$, which implies that $[b] \neq [c]$.
	\item  $\alg C/H$ is a subalgebra of the gluing since it is closed under all operations.
	\item Given $[x], [y] \in C/H$ and $[z] \in B/H-\{[1]\}$, $[x][z] = [xz] = [yz]= [y][z]$ and similarly $[z][x] = [zx] = [zy]= [z][y]$.
\end{enumerate}
It then follows directly by Theorem \ref{thm:blockwise} that $(\alg B \boxplus_\beta \alg C)/H \cong \alg B/H \boxplus_{\beta_H} \alg C/H$. Indeed, fixing $[c] \in C/H$, we have that $\sell_{\alg B \boxplus_\beta \alg C/H}([x]) = [c][x] = [cx] = [\sell(x)] = \sell_H[x]$. The proof for the other operators is completely analogous. The points (1) and (2) then follow from the description of congruence filters in the gluing, Theorem \ref{theorem: description of filters}.
		\end{proof}
	
Let us proceed to describe subalgebras. 
Let $(\alg B, \beta)$ be a lower block, with $\beta = (\sell, \sr, \gell, \gr)$.
A partial IRL $\alg S$ is a {\em $\beta$-subreduct} of the lower block if its domain is a subset $S$ of $B$ closed under all operations defined in $\alg B$, and such that for all $x \in B$, $\alpha(x) \in S$ for $\alpha \in \{ \sell, \sr, \gell, \gr\}$. We call $\alg S$ a {\em $\beta$-subalgebra} if it is a $\beta$-subreduct that is a total algebra. Moreover, let $\alg A$ be an $\vv{IRL}$ with a splitting coatom, then a subalgebra $\alg T$ of $\alg A$ is called {\em coatomic} if it contains the coatom of the algebra $\alg A$. 
\begin{proposition}
	Let $\alg B \boxplus_\beta \alg C$ be the blockwise gluing of $(\alg B, \sell, \sr, \gell, \gr)$ and $\alg C$. Then an IRL $\alg S$ is a subalgebra of $\alg B \boxplus_\beta \alg C$ iff it is one of the following: 
	\begin{enumerate}
		\item a subalgebra of $\alg C$;
		\item a (total) subalgebra of $\alg B$;
	    \item $\alg B_1 \boxplus_{\beta_1} \alg C_1$, where
		\begin{itemize}
			\item $\alg C_1$ is a non-trivial coatomic subalgebra of $\alg C$
			\item $\alg B_1$ is a $\beta$-subreduct of $\alg B$
			\item $\beta_1$ is given by the restrictions of the operators in $\beta$ to $\alg B_1$.
		\end{itemize}
		\item $\alg B_2 \boxplus_{\beta_2} \alg C_2$ where
		\begin{itemize}
			\item $\alg C_2$ is a subalgebra of $\alg C$
			\item $\alg B_2$ is a $\beta$-subalgebra of $\alg B$
			\item $\beta_2$ is given by the restrictions of the operators in $\beta$ to $\alg B_2$.
		\end{itemize}
	\end{enumerate}
\end{proposition}
\begin{proof}
It is easy to check that all the structures in points (1)--(4) are subalgebras of $\alg B \boxplus_\beta \alg C$. Conversely, let $\alg S$ be a subalgebra of $\alg B \boxplus_\beta \alg C$. We can distinguish the following cases: 
\begin{enumerate}
	\item if $S \subseteq C$, then $\alg S$ is a subalgebra of $\alg C$ and we get case (1). 
	\item If $S \subseteq B$, then $\alg S$ is a total algebra whose domain is contained in $B$, getting case (2). 
	\item If $S$ intersects nontrivially both $B$ and $C$, then for any $x \in S \cap B$, we get that $\sr(x) = xc$, $\sell(x) = cx$, $\gell(x) = c \backslash x$, and $\gr(x) = x / c$ are all in $S$. If all such divisions are in $B$, we get case (4). Otherwise, if there exists $x, y \in S \cap B$ such that either $x \backslash y$ or $y / x$ end up in the coatom $c_\alg C \in \alg C$, then $c_\alg C \in S$ and we get case (3).
\end{enumerate}

Note that these are all the possible cases for $\alg S$ to be a subalgebra of $\alg B \boxplus_{\beta}  \alg C$, and this concludes the proof.  
\end{proof}

We are now going to describe the ultrapowers of a blockwise gluing $\alg B \boxplus_\beta \alg C$.
We denote by $\vv{P_u}(\alg B)\boxplus \vv{P_u}(\alg C)$ the class of algebras which are blockwise gluings of an algebra in $\vv{P_u}(\alg B)$\footnote{It makes sense to consider ultrapowers of partial IRLs, defined in the obvious way (see for instance \cite{Burmeister}).} and an algebra in $\vv{P_u}(\alg C)$.
\begin{proposition}
	$\vv{P_u}(\alg B \boxplus_\beta \alg C)\subseteq \vv{P_u}(\alg B)\boxplus \vv{P_u}(\alg C)$.
\end{proposition}
\begin{proof}
	Let $\alg A = \prod_{j \in J} \alg B \boxplus_\beta \alg C$ and let $ U$ be an ultrafilter on $J$; moreover, let $\alg A_U=\alg A / U$. We prove that $\alg A_U$ belongs to $\vv{P_u}(\alg B)\boxplus \vv{P_u}(\alg C)$. 
	
	For $x=(x_j)_{j \in J}$ in $\alg A$, and for each $j \in J$, either $x_j \in B$ or $x_j \in C$. Let us then define the following two sets:
	\begin{align*}
		J_B(x)=\{ j \in J :  x_j \in B - \{1\}\} \qquad J_C(x)=\{ j \in J : x_j \in C \}.
	\end{align*}
	Note that $J_B(x)$ and $J_C(x)$ partition $J$, hence exactly one of the two sets belongs to $U$ for each $x \in A$. Moreover, if $[x]_U=[y]_U$ then the corresponding sets are either both in $U$ or not in $U$. This allows us to define the following two sets:
	$$
	B_U= \{[x]_U : J_B(x) \in U \}\cup \{[1]\} \qquad C_U=\{[x]_U : J_C(x)\in U \}.
	$$
	Note that $C_U$, with the inherited operations, is an $\vv{IRL}$, while $B_U$, with the inherited operations, is a partial $\vv{IRL}$; precisely, note that $\alg C_U$ is a subalgebra of $\alg A_U$ while $\alg B_U$ is a partial subalgebra of $\alg A_U$. Let $\beta_U=(\sell_U, \sr_U, \gell_U, \gr_U)$, with $\alpha_U[x]_U = [x^\alpha]_U$ where:
	\begin{align*}
	x^\alpha_i &= \left\{\begin{array}{ll}
\alpha(x_i) &\mbox{ if } i \in J_B(x) \vspace{0.1cm}\\
x_i &\mbox{ otherwise } \vspace{0.1cm}
\end{array}\right.
\end{align*}
and $\alpha \in \{\sell, \sr, \gell, \gr \}$. In words, $\alpha_U[x]_U$ is the equivalence class of an element whose $i$-th component is exactly $\alpha(x_i)$ for all indexes $i$ for which the $\alpha$-image is defined. We proceed to show that $\alg A_U$ coincides with $\alg B_U \boxplus_{\beta_U} \alg C_U$. Note that $B_U\cup C_U = A_U$ and $B_U \cap C_U = \{[1]\}$.

 Let us prove that $\beta_U$ is well defined. First, consider any $\alpha_U \in \{\sell_U, \sr_U, \gell_U, \gr_U \}$, then $\alpha_U: B_U \to B_U$ since 
 $$J_B(x^\alpha)=\{ i \in I : x^\alpha_i \in B_i \}=J_B(x).$$
  Let now $[x]_U, [y]_U \in B_U$ and suppose that $[x]_U=[y]_U$, we verify that $\alpha_U[x]=\alpha_U[y]$. Since $[x]_U=[y]_U$, $\{i \in I : x_i=y_i \} \in U$. Note that for any $i \in I$ such that $x_i=y_i$, $x^\alpha_i=y^\alpha_i$ which implies that 
  $$\{i \in I : x_i = y_i \} \subseteq \{ j \in I : x^\alpha_j=y^\alpha_j \}$$ and so since ultrafilters are closed upwards, $\{ j \in I : x^\alpha_j=y^\alpha_j \} \in U$ getting that $\alpha_U[x]=\alpha_U[y]$ as desired.
  
Let us then show that $(\alg B_U, \beta_U)$ is a lower block, i.e., the properties in Definition \ref{def: lower block} hold. Note that if $[x]_U, [y]_U \in B_U$, then $J_B(x)$ and $J_B(y)$ are in $U$, hence also $J_B(x)\cap J_B(y) \in U$. Therefore, the properties of Definition \ref{def: lower block} apply to all elements in the $i$-th component of the direct product for any $i \in J_B(x)\cap J_B(y)$. With this observation in mind, it is easy to see that $(\alg B_U, \beta_U)$ is a lower block. For instance, let us show the proof of the last property in Definition \ref{def: lower block}. Suppose that $[x]_U, [y]_U \in B_U$ and $[x]_U \backslash [y]_U$ is not defined in $B_U$, we prove that $\sr_U[x]_U\leq [y]_U$ and $[x]_U \not \leq [y]_U$. Note that we have that $[x]_U \backslash [y]_U=[x\backslash y ]_U \in C_U$, i.e. 
$$J_C(x \backslash y)=\{ i \in I : x_i \backslash y_i \in C_i \} \in U.$$ Since $J_B(x), J_B(y)$ and $J_C(x \backslash y) \in U$, also $$K=J_B(x)\cap J_B(y)\cap J_C(x \backslash y) \in U$$ and for any $i \in K$ $x_i, y_i \in B_i $ and $x_i \backslash y_i \in C_i$ which implies that for any $i \in K$, $\alpha(x_i) = x_i^\alpha\leq y_i$, and $x_i \not \leq y_i$. Thus, if $[x]_U\backslash [y]_U$ is not defined in $B_U$, then $\sr_U[x]_U\leq [y]_U$ and $[x]_U \not \leq [y]_U$. The other properties of Definition \ref{def: lower block} can be proven similarly. Hence, the blockwise gluing $\alg B_U \boxplus_{\beta_U} \alg C_U$ is well defined. 

Now, it is straightforward to check that both the elements and the operations among them in $\alg B_U \boxplus_{\beta_U} \alg C_U$ coincide exactly with those of $\alg A_U$. Hence, $\alg A_U=\alg B_U \boxplus_{\beta_U} \alg C_U$. 
	It is also clear that 	$\alg B_U = (\prod _{j \in J} \alg B \big )/ U$ and $\alg C_U= (\prod_{j \in J} \alg C )\big / U$. Indeed, take for instance any $[x]_U \in B_U$, then $J_B(x) \in U$ or $[x]_U=[1]_U$. If $y$ is such that: 
	\begin{align*}
	y_j &=\left\{\begin{array}{ll}
x_j &\mbox{ if } j \in J_B, \vspace{0.1cm}\\
1 &\mbox{ otherwise,} \vspace{0.1cm}
\end{array}\right.
\end{align*}
then clearly $y \in \prod_{j \in J} B$ and $[y]_U=[x]_U$ getting that $[x]_U \in \prod_{j \in J} B/U$ and so $B_U \subseteq \prod_{j \in J} B / U$. The converse inclusion is clear, so $\alg B_U = (\prod _{j \in J} \alg B \big )/ U$. The proof for $\alg C_U= (\prod_{j \in J} \alg C )\big / U$ is completely analogous.

Thus, we can conclude that $\alg A_U \in \vv{P_u}(\alg B)\boxplus \vv{P_u}(\alg C)$ which yields the thesis. 
\end{proof}

\section{Varieties generated by blockwise gluings: axiomatizations}
We now focus our attention on totally ordered algebras obtained by blockwise gluings, and on the axiomatization of the varieties they generate. 
In order for this to make sense, let us first observe that given a residuated chain $\alg A$, there exists a {\em finest} decomposition in terms of blockwise gluings. Let us be more precise.

Given an IRL chain $\alg A$, let us call a family of partial IRLs $\cc D = \{\alg A_{i}\}_{i \in I}$ ordered by a totally ordered index set $\II = (I, \leq)$ a {\em decomposition} of $\alg A$ if $\alg A = \boxplus_{\II}(\alg A_{i}, \betaij)$ for some family of blocks $\{(\alg A_{i}, \betaij)\}_{i <_{\II} j}$ based on $\cc D$.
Let us call an IRL $\alg A$ {\em gluing-indecomposable} if it is not isomorphic to a blockwise gluing of two nontrivial IRLs.

We can define a partial order on the set of decompositions of $\alg A$, $\vv{D}_{\alg A}$, as follows: if $\cc D, \cc D' \in \vv{D}_{\alg A}$, $$\cc D \preceq \cc D' \mbox{ iff for each } \alg B \in \cc D  \mbox{ there is a } \alg B' \in \cc D' \mbox{ such that } B \sse B'.$$
Intuitively, $\cc D \preceq \cc D'$ iff the decomposition $\cc D'$ is finer than $\cc D$. The proof of the fact that there is a maximum element in $(\vv{D}_{\alg A}, \preceq)$ is analogous to the one for 1-sums in \cite[Theorem 14]{Montagnaetal2006}, we adapt it here to blockwise gluings.
\begin{proposition}
	For every IRL chain $\alg A$, there is a finest decomposition in blockwise gluings, i.e., a maximum element in $(\vv{D}_{\alg A}, \preceq)$.
\end{proposition}
\begin{proof}
The proof uses Zorn's Lemma. Consider a chain $\mathscr{C} = \{\cc D_k : k \in  K\}$ in $(\vv{D}_{\alg A}, \preceq)$, we start by showing that it has an upper bound in $(\vv{D}_{\alg A}, \preceq)$. 
Define the following equivalence relation on $A-\{1\}$: for every $a,b \in  A-\{1\}$, $a \equiv b$ iff $a$ and $b$ belong to the same class of $\cc D_k$ for every $k \in  K$; denote by $[a]_\equiv$ the equivalence class of $a$ w.r.t $\equiv$. We prove that $A_\equiv = \{[a]_\equiv \cup 1: a \in A-\{1\}\} \in \vv{D}_{\alg A}$ and it is an upper bound of $\mathscr{C}$. Take $(I, \leq)$ to be a totally ordered index set for $A_\equiv$ (one could take a representative for each equivalence class, ordered as in $\alg A$), and let us refer to the elements of $A_\equiv$ as $A_i$ for $i \in I$. It is easy to check that $\alg A$ belongs to $\BlG$ according to Definition \ref{def:blockwiseclass}. Indeed, the domain of $\alg A$ is the union of the sets $A_i$, for $i \in I$, and the elements of $A_i- \{1\}$ are strictly below the elements of $A_j$ iff $i < j$. Moreover, it is straightforward that each $A_i$ is the domain of a partial IRL $\alg A_i$ with the inherited operations from $\alg A$, where all operations are defined except possibly some divisions (that end up in higher components).
	Finally, if $x, y \in A_i- \{1\}, z \in A_j$ and $j < i$, then $xz = yz$ and $zx = zy$; indeed, if $j < i$, there is one of the decompositions $\cc D_k$ such that $x,y$ belong to the same block (since they are in the same block in any decomposition), while $z$ belongs to a lower block, and then the claim follows. By Proposition \ref{proposition: a in BG}, $\alg A$ can be written as a blockwise gluing of a family of blocks based on  $\{\alg A_i\}_{i \in I}$, thus $\cc A = \{\alg A_i\}_{i \in I} \in \vv{D}_{\alg A}$.

We claim that $\cc A$ is an upper bound for $\mathscr{C}$. Indeed, for any $a \in A-\{1\}$, all the elements of $[a]_\equiv$ are in the same component of $\cc D_k$ for all $k \in K$, which means that $\cc A \preceq \cc D_k$.
Therefore, by Zorn's Lemma, for every $\cc D \in \vv{D}_{\alg A}$, there exist a maximal decomposition $\cc M \in \vv{D}_{\alg A}$ such that $\cc D \preceq \cc M$. 

Finally, we prove that there is only one maximum decomposition.
Suppose that $\cc M_1, \cc M_2 \in \vv{D}_{\alg A}$ are two different maximal elements. Then there is $\alg B \in \cc M_1$ whose domain is not included in any element of $\cc M_2$. Moreover, by the maximality of $\cc M_1$, $\alg B$ is indecomposable, so in particular it is not a union of elements of $\cc M_2$. Hence, there is $\alg B' \in \cc M_2$ such that $B \cap B' \neq \emptyset$ and $B' \nsubseteq B$. But then one could further refine $\cc M_2$ by splitting $B$ in $B \cap B'$ and $B-B'$, a contradiction. The thesis follows.
\end{proof}
\begin{notation}\label{notation1}
	From now on, if we consider an IRL chain $\alg A$ as $\alg A = \boxplus_{\bf I} \alg A_i$ with no further description, we mean that the decomposition given by the family $\{\alg A_i\}_{i \in I}$ is the maximum one. Moreover, for a finite index set of cardinality $k \in \mathbb{N}, k \geq 1$, we may write $\alg A = \boxplus_{\bf k} \alg A_i$ to denote the ordered index set given by $1 < \ldots < k$.
\end{notation}
A particularly interesting case, which we will analyze in the rest of this section, is the maximum decomposition being given by the archimedean components.

Let $\alg A$ be an IRL chain. Given $a,b \in A$, we write $a \sim b$ if $a$ and $b$ are in the same archimedean component, that is to say, if either $a^n\leq b \leq a$ or $b^n \leq a \leq b$ for some $n \in \mathbb{N}$. It can be easily shown that $\sim$ is an equivalence relation; its equivalence classes are the archimedean classes. Given $a \in A$, we write $\arc a$ for its archimedean class. Note that one can associate a totally ordered index set $(I, \leq )$, by picking an element $a_i$ from each equivalence class, and let $A_i={\arc {a_i}} \cup \{1\}$. Observe that $\alg A_i$ with the restricted operations from $\alg A$ is always a partial IRL, only missing those divisions which end up in some $A_j$ with $j>i$.

Notice that, in general, decomposing a chain in its archimedean component does not yield a blockwise gluing. However, when it does, it yields the finest decomposition.
\begin{corollary}
	Let $\alg A$ be an IRL-chain. If the partition given by the archimedean classes gives a decomposition as a blockwise gluing, then it is the maximum one.
\end{corollary}
\begin{proof}
	The proof follows from the fact that the components of a blockwise gluing are closed under products.
\end{proof}
Given this fact, we now focus on chains that are blockwise gluings of their archimedean components, and the varieties they generate.
The next subsections show that with this restriction, one can obtain finite axiomatizations for varieties generated in a natural way by blockwise gluings. 

The reader may observe that tackling a more general, unrestricted, scenario is out of reach. For comparison, we do not even have techniques to axiomatize varieties generated by 1-sums of chains, if one removes the conditions of divisibility \cite{AM2003} or some weaker forms of it \cite{GalatosUgolinimanuscript}.

 \subsection{$n$-potent gluings of archimedean components}\label{subsection:axiom-npotent}
Let us call an IRL chain that is a blockwise gluing of its archimedean components an {\em archimedean gluing}, and let 
\begin{equation*}
	\BlockG = \{\alg A \in \vv{IRL_c}: \alg A \mbox{ is an archimedean gluing}\}.
\end{equation*}

In this subsection we axiomatize, for all $n \in \mathbb{N}$, the $n$-potent varieties generated by chains which are archimedean gluings, and we show that they all have the FEP. Moreover, we show that their join is the variety generated by all chains that are archimedean gluings, which has the FMP.
 
Let us denote with $\vv{IRL_c}$ the class of IRL-chains. We define the class of algebras:
$$
\BlockG_{n}=\{\alg A \in \vv{IRL_c} : (x^n=x^{n+1} > y^n) \mbox{ implies } (xy=x^ny \mbox{ and } yx=yx^n) \}.
$$
\begin{lemma}\label{lemma: blockwise gluing of archim comp}
	If $\alg A \in \BlockG_{n}$ then $\alg A$ is an archimedean gluing. 
\end{lemma}
\begin{proof}
It is easy to see that if $\alg A \in \BlockG_{n}$ then it is a blockwise gluing of its archimedean classes, according to Definition \ref{def:blockwiseclass}. Indeed, if one considers $x, y, z$ with $\arc x < \arc y = \arc z$, this means that $y^n=z^n > x^n$, and so by definition of $\BlockG_{n}$, $xy=xz$ and $yx=zx$.
\end{proof}
 Let now $\BG_{n}$ be the variety axiomatized relatively to $n$-potent $\vv{SemIRL}$ by: 
\begin{equation*} \label{eq: equation Bn}
  	(x^n \backslash y^n) \join (((xy\backslash x^ny)\meet (yx \backslash yx^n)))=1
  	\tag{block$_n$}
 \end{equation*} 
 
 The goal is to show that $\vv V(\BlockG_{n})=\BG_{n}$. 
 \begin{lemma}\label{lemma: equation holding in Bn}
 	Let $\alg A$ be an $n$-potent IRL-chain, then $\alg A \models$ {\rm (\ref{eq: equation Bn})} iff $\alg A \in \BlockG_{n}$. 
 \end{lemma}
 \begin{proof}
 	Suppose that $\alg A \in \BlockG_{n}$ and let $x,y \in A$. If $x^n\leq y^n$, then $(x^n \backslash y^n)=1$; if $x^n > y^n$,  $xy=x^ny$ and $yx=yx^n$, therefore both $xy \backslash x^ny $ and $yx \backslash yx^n$ are equal to $1$. Hence (\ref{eq: equation Bn}) holds in $\alg A$. 
 	
 	 	Vice versa, suppose that $\alg A \not \in \BlockG_{n}$, then there exists $x,y \in \alg A$ such that $x^n > y^n$ but $xy > x^ny$ or $yx > yx^n$. Hence both $x^n \backslash y^n$ and $(xy \backslash x^ny)\meet (yx \backslash yx^n)$ are strictly less than $1$, meaning that $\alg A \nvDash$ (\ref{eq: equation Bn}).	
 \end{proof}
 As a consequence of the previous lemma, 
 \begin{proposition}
 	$\BG_{n}=\vv V(\BlockG_n)$.
 \end{proposition}
 \begin{proof}
 	$\BG_{n}$ is generated by the class of its chains, i.e. $\BG_{n}=\vv V((\BG_{n})_{\vv c})$. By Lemma \ref{lemma: equation holding in Bn}, $(\BG_{n})_{\vv c}=\BlockG_{n}$; hence $\BG_{n}=\vv V(\BlockG_{n})$.
 \end{proof}
 Thus, $\BG_n$ is the variety whose chains are exactly the $n$-potent archimedean gluings. In the following lemma we demonstrate that, in order to axiomatize $\BG_n$, one can equivalently consider the following variation of the above axiom {\rm (\ref{eq: equation Bn})}, which will be used in some proofs:
 \begin{equation*} \label{eq: equation 2Bn}
  	(x^n \backslash y^n) \join (((xy\backslash x^2y)\meet (yx \backslash yx^2)))=1
  	\tag{block$_n^2$}. 
  	 \end{equation*} 

 \begin{lemma}\label{rem: semplification axioms}
Let $\alg A$ be an $n$-potent IRL-chain, then $\alg A \models$ {\rm (\ref{eq: equation 2Bn})} iff $\alg A \models$ {\rm (\ref{eq: equation Bn})}.
 \end{lemma}
 \begin{proof}
 Let $\alg A$ be an $n$-potent IRL chain.
 	For the right-to-left direction, assume that $\alg A \models$ (\ref{eq: equation Bn}). Let $x, y \in A$, then if $x^n \leq y^n$, $x^n \backslash y^n=1$. If $y^n < x^n$, then $x \cdot y= x^n \cdot y$ and $y \cdot x = y \cdot x^n$, hence, $$x^2 \cdot y = x \cdot x \cdot y = x \cdot x^n \cdot y = x^{n+1}\cdot y = x^n \cdot y= x \cdot y$$ and similarly $y \cdot x^2 = y \cdot x$. In either case, $(x^n \backslash y^n) \join (((xy\backslash x^2y)\meet (yx \backslash yx^2)))=1$, hence $\alg A \models$ (\ref{eq: equation 2Bn}). 
 	
  Vice versa, let $\alg A \models$ (\ref{eq: equation 2Bn}) and let $x, y \in A$. If $x^n \leq y^n$, $x^n \backslash y^n=1$. If $y^n < x^n$, then $x^2y=xy$ and $yx^2=yx$. Thus, one gets the following chains of equalities:
$$
x y = x^2 y= \dots = x^{n-1}  y = x^n y \qquad y \cdot x= y \cdot x^2= \dots = y \cdot x^{n-1} =y \cdot x^n.
$$
Hence, in both cases $(x^n \backslash y^n) \join (((xy\backslash x^ny)\meet (yx \backslash yx^n)))=1$, which means that $\alg A \models $ (\ref{eq: equation Bn}).
 \end{proof}

We observe that the varieties $\BG_n$ form a denumerable chain ordered by inclusion.
\begin{lemma}
	For any $n \in \mathbb{N}$ with $n \geq 1$, $\BG_{n} \subseteq \BG_{n+1}$.
\end{lemma}
\begin{proof}
	Let $\alg A \in \BG_{n}$, then $\alg A \models (x^n \backslash y^n) \join ((x(x \meet y) \backslash x^n(x \meet y))\meet ((x \meet y)x \backslash (x \meet y)x^n))$. Moreover, $\alg A$ is $n$-potent for any $x \in \alg A$, i.e. $x^n=x^{n+1}$ and so $\alg A \models (x^{n+1} \backslash y^{n+1}) \join ((x(x \meet y) \backslash x^{n+1}(x \meet y))\meet ((x \meet y)x \backslash (x \meet y)x^{n+1}))$, meaning that $\alg A \in \vv \BG_{n+1}$.
\end{proof}
It is then natural to wonder about the union of these hierarchy of varieties. We proceed to prove that the join of all the varieties $\BG_{n}$, so the variety generated by their union $\bigvee_{n \in \mathbb N} \BG_{n}$, is exactly the variety generated by chains which are archimedean gluings. Moreover, the latter is strictly contained in the variety of all semilinear IRLs.

In order to do so, we will see that the residuated frame construction described in the preliminaries preserves the decomposition in archimedean gluings. For the sake of the reader, we recall only the information that is going to be useful for the following proofs, while more details are given in the preliminary section.

Given an IRL chain $\alg A$ and a partial subalgebra $\alg B$, it is possible to construct a finite IRL chain $\alg W_{\alg A, \alg B}^+ = (\gamma_N[\mathcal{P}(W)], \cap, \cup_{\gamma_N}, \circ_{\gamma_N}, \backslash, /)$ into which $\alg B$ embeds (Proposition \ref{prop:galatosjipsen} and Lemma \ref{lemma: finite chain}). In particular, this chain is constructed using: 
\begin{itemize}
	\item $(W, \cdot, 1)$ as the submonoid of $\alg A$ generated by $\alg B$; 
	\item the sections $S_W$ on $W$, i.e. the maps of the kind $u(x)=v \cdot x \cdot w$ for $v, w \in W$, and then $W' = S_W \times B$; 
	\item the nuclear relation $N \sse W \times W'$ defined by $x N (u,b)$ iff $u(x)\leq_{\alg A} b.$ 
\end{itemize}
To the nuclear relation one associates the operator $\gamma_N$, which is a nucleus on $(\mathcal{P}(W), \circ)$, with $\circ$ being the complex product. Then the domain $\gamma_N(\mathcal{P}(W))$ is given by the following downsets of $W \cap A$:
$$\{(u,b)\}^\triangleleft = \{w \in W : u(w)\leq_{\alg A} b\}.$$
For the reader's convenience, we also recall the product operation:
$$X\circ_{\gamma_N} Y = \gamma_N(X \circ Y).$$
Before we proceed, a technical lemma. In what follows, for $m \geq 1$, we write $Y \circ^m Y$ for the complex product $Y \circ \ldots \circ Y$ of $Y$ with itself $m$ times (so in particular $Y \circ^1 Y = Y$).

 \begin{lemma}\label{lemma: powers of sets}
Let $Y \in \gamma_N(\mathcal{P}(W))$, then $Y^m=\gamma_N(Y \circ^m Y)$.
\end{lemma}
\begin{proof}
	We prove the claim by induction on $m$. If $m =1$, then $Y = \gamma_N(Y)$, since $Y \in \gamma_N(\mathcal{P}(W))$ and $\gamma_N$ is a nucleus hence idempotent. 
	
	Suppose now that the claim holds for $m-1$, let us prove that it holds for $m$. 
	\begin{align*}
	Y^m &= Y \circ_{\gamma_N} Y^{m-1}\\
	&=Y \circ_{\gamma_N} \; \gamma_N(Y \circ^{m-1}Y)\\
	&= \gamma_N(Y \circ \; \gamma_N(Y \circ^{m-1}Y))\\
	&= \gamma_N (\gamma_N(Y) \circ \;\gamma_N(Y \circ^{m-1}Y))\\
	&=\gamma_N(Y \circ (Y \circ^{m-1} Y))\\
	&= \gamma_N(Y \circ^{m}Y)
	\end{align*}
	where the second equality holds by inductive hypothesis, the third by definition of $\circ_{\gamma_N}$, the fourth since $Y \in \gamma_N(\mathcal{P}(W))$, and the second to last one follows from the fact that $\gamma_N$ is a nucleus on $(\gamma_N(\mathcal{P}(W)), \circ)$. 
\end{proof}

We are now ready to prove that the presented residuated frame construction applied to IRL chains preserves the blockwise gluing decomposition in archimedean components. 
Given a class of algebras $\vv K$ let $\vv K_{\rm fin}$ be the class of its finite members. 
The next lemma shows that $\BlockG$, the class of IRL-chains that are archimedean gluings, has the FEP.

\begin{lemma}\label{theorem: preparation to FMP}
Let $\alg A \in \BlockG$ and let $\alg B$ be any finite partial subalgebra of $\alg A$, then $\alg W_{\alg A, \alg B}^+\in (\BlockG)_{\rm fin}$. 
\end{lemma}
\begin{proof}
	Let $\alg A \in \BlockG$ and let $\alg B$ be a finite partial subalgebra of $\alg A$. As shown in Lemma \ref{lemma: finite chain}, $\alg W_{\alg A, \alg B}^+$ is a finite chain. Then the goal is to show that it is in $\BG_{n}$ for some $n \in \mathbb{N}$. Since it is finite, it is $n$-potent for some $n \in \mathbb{N}$; hence, by Lemma \ref{rem: semplification axioms}, it suffices to show that if $X^n \supset Z$ then $X^2 Z= X Z$ and $Z X^2= Z X$, for any $X, Z \in W_{\alg A, \alg B}^+$. We prove the former, since the latter is similar.
	
	 Let then $$X=\{w \in W : u (w ) \leq_{\alg A} d\}, \;\;Z=\{w \in W : u' (w )\leq_{\alg A} d'\}$$  for some $u, u' \in W'$ and $ d, d' \in B$. Hence $XZ=\gamma_N(\{ab : a \in X, b \in Z\})$. Observe that by order preservation $X^2Z \subseteq XZ$. Hence, it suffices to show the converse inclusion. To do so it is enough to check that $$\{ab : a \in X \mbox{ and } b \in Z \} \subseteq \{cb : c \in X^2 \mbox{ and } b \in Z \},$$ i.e., for any $ a \in X \mbox{ and } b \in Z $, there exists a $c \in X^2$ such that $ab=cb$. Let us then fix $a \in X$ and $b \in Z$.

 If $a \in X^2 \subseteq X$, then we take $c=a$. Thus, the only case left to consider is when $a \in X - X^2$. We set $c=a^2$, then $a,a^2 \in W \subseteq A$. So if one shows that $b$ is in a lower component than $a$ in $\alg A$, since $\alg A$ is a blockwise gluing of its archimedean components, $ab=a^2b=cb \in X^2 Z$ which would complete the proof. 
 
 Recall that $\arc x$ denotes the archimedean component of an element $x$. Suppose by way of contradiction that $\arc b$ is above or equal to $\arc a$, then there exists $m \in \mathbb{N}$ such that $a^m \leq_{\alg A} b$. Observe that for any $x \in X^2$, $x <_{\alg A} a$; indeed  if $a \leq x$, since $X^2$ is a downset, $a \in X^2$ contradicting the fact that $a \in X- X^2$. Thus, for any $x \in X^2$, $x^m \leq_{\alg A} a^m \leq_{\alg A} b$ meaning that 
 $$X^2 \circ^m X^2\;\subseteq\; {\downarrow }b.$$
  Indeed, let $x_1, \dots ,x_m \in X^2$, then 
  $$x_1 \cdot \ldots \cdot x_m\leq_{\alg A }(\max\{x_i\}_{i=1, \dots ,m})^m\leq_{\alg A} a^m \leq_{\alg A} b.$$
   Now, ${\downarrow} b \subseteq Z$, hence $X^2 \circ^m X^2\subseteq Z$ and so, 
   $$X^{2m}=\gamma_N(X \circ^{2m} X)=\gamma_N(X^2 \circ^m X^2) \subseteq \gamma_N(Z)=Z,$$ 
   where the first equality follows from Lemma \ref{lemma: powers of sets}, the inclusion from the fact that $\gamma_N$ is a nucleus hence order preserving, and the last since $\gamma_N(Z)=Z$ given that $Z$ is an element of the domain. Now, since $\alg W_{\alg A, \alg B}^+$ is $n$-potent, necessarily $X^n \subseteq X^{2m}\subseteq Z$ which contradicts the hypothesis. 
   
   Hence, $\arc a$ is strictly above $\arc b$ in $\alg A$, thus $ab = a^2b = cb$. We proved that if $X^n \supset Z$, then $X^2Z=XZ$ as desired. Similarly, $ZX^2=ZX$, meaning that $\alg W_{\alg A, \alg B}^+ $ is a finite chain in $\BG_{n}$ for some $n \in \mathbb{N}$ and so $\alg W_{\alg A, \alg B}^+ \in (\BlockG)_{\rm fin}$.
\end{proof}
The above lemma has several interesting consequence. Firstly, let us call $\BG$ the variety generated by $\BlockG$, $\BG = \vv V (\BlockG)$.
\begin{theorem}\label{corollary: FMP for V(BG)}	
$\BG$ has the finite model property.
\end{theorem}
\begin{proof}
	Suppose $\varepsilon$ is an equation failing in $\BG$, then there exists a chain $\alg A \in \BlockG$ such that $\alg A \nvDash \varepsilon$. Note that $\varepsilon$ has finitely many variables $x_1, \dots , x_n$. Let $a_1, \dots a_n$ be the assignment of the variables witnessing the failure of $\varepsilon$ in $\alg A$. Moreover, let $\Phi$ be the set of subterms appearing in $\varepsilon$, which is finite. We define $B=\{\phi_i(a_1, \dots, a_n) : \phi_i \in \Phi \}$. Take $\alg B$ as the partial subalgebra of $\alg A$ having as universe $B \cup \{1\}$. Since $\alg B$ is finite, by the previous theorem it embeds into a finite algebra in $\BlockG$ which does not satisfy $\varepsilon$. Thus, $\BG$ has the finite model property.
\end{proof}

Also, with a slight adaptation of the proof of Lemma \ref{theorem: preparation to FMP} one can prove that any partial subalgebra of a chain in $\BG_n$ can be embedded into a finite $\BG_n$-chain. Hence:
\begin{theorem}\label{corollary: FEP e FMP for Bn}
	$\BG_{n}$ has the finite embeddability property. 
\end{theorem}
We are ready to state the preannounced result that $\BG=\bigvee_{n \in \mathbb{N}}\BG_{n}$, which is the variety generated by the union of all $\BG_{n}$. 
\begin{theorem}\label{lemma: partial gluing of archimedean components}
	$\BG= \bigvee_{n \in \mathbb{N}}\BG_{n}$.
\end{theorem}
\begin{proof}
Since all chains in each $\BG_n$ are in $\BG$, we get that ${\BG} \supseteq \bigvee_{n \in \mathbb{N}}\BG_{n}$. 

For the left-to-right inclusion, we want to prove that if an equation holds  in $\bigvee_{n \in \mathbb{N}}\BG_{n}$, it holds in $\BG$; equivalently, if an equation fails in $\BG$, it fails in $\bigvee_{n \in \mathbb{N}}\BG_{n}$. The latter is an immediate consequence of Lemma \ref{theorem: preparation to FMP}.
\end{proof}

It is important to notice that, unlike what happens in each $\BG_n$, in $\BG$ not all chains are archimedean gluings, despite the latter being the generators of the variety. The key observation is that archimedeanity is not a first-order property, and this results in the fact that the class $\BlockG$ is not closed under ultraproducts. Let us provide an example of a chain in $\BG$ that is not an archimedean gluing.
\begin{example}
	The variety of MV-algebras, $\vv{MV}$, is contained in $\BG$. Indeed, $\vv{MV}$ is generated by its finite chains, which are all simple, hence, trivially, archimedean gluings. Not all MV-chains, however, are blockwise gluings of their archimedean components. An example of this is Chang's MV-algebra, with domain $$C= \{1, c, c^2, \ldots c^n \ldots \neg c^n, \ldots, \neg c^2, \neg c, 0\},$$
where $C^+ = \{1, c, c^2, \ldots c^n \ldots \neg c^n\}$ is a copy of the negative cone of the lattice-ordered group of the integers $\alg Z^-$, the elements in $C^- = \{\ldots \neg c^n, \ldots, \neg c^2, \neg c, 0\}$ all multiply to $0$, and $c^n \cdot \neg c^m = \neg (c^{m - n})$ if $m > n$ and $0$ otherwise. Here, $C^-$ and $C^+$ are two different archimedean classes, and if $m < n$ 
$$c^n \cdot \neg c^m = \neg (c^{m - n}) > 0 = c^m \cdot \neg c^m,$$
hence Chang's MV-algebra is not an archimedean gluing, but it belongs to $\BG$ (in fact, it can be obtained as an ultraproduct of the finite MV-chains).
\end{example}
One might then wonder whether all IRL-chains end up in the variety $\BG$, but this does not hold.
To show this, we first observe that all algebras in $\BG$ validate all of the following formulas, for any $n \geq 1$:
\begin{equation*}
	(x^n=x^{n+1}~\&~y \leq x^n) \Rightarrow (xy=x^ny ~\&~ yx=yx^n), \tag{$\beta_n$}
\end{equation*}
where $\&$ and $\Rightarrow$ are first-order conjunction and implication respectively. Note that each $\beta_n$ is equivalent to the two quasi-equations: $(x^n=x^{n+1}~\&~y \leq x^n) \Rightarrow (xy=x^ny)$ and $(x^n=x^{n+1}~\&~y \leq x^n) \Rightarrow (yx=yx^n)$.
\begin{lemma} \label{lemma: equations ast}
	$\BG\models (\beta_n)$ for all $n \geq 1$.
\end{lemma}
\begin{proof}
Since $\BG = \vv V(\BlockG)$, by J\'onsson Lemma, all the subdirectly irreducible algebras are in $\mathsf{HSP}_u(\BlockG)$. Thus it suffices to check that each $(\beta_n)$ holds in all algebras in $\mathsf{HSP}_u(\BlockG)$. 

Fix any $n \in \mathbb{N}$. It is straightfoward that for all $\alg A \in \BlockG$, $\alg A \models (\beta_n)$, and given any $\alg B \in \mathsf{ISP}_u(\BlockG)$, also $\alg B \models (\beta_n)$ since quasiequations are preserved by $\mathsf{ISP}_u$, and $(\beta_n)$ is equivalent to two quasiequations. Thus it is left to show that if  $\alg A \in \mathsf{H}(\alg B)$ with $\alg B \in \mathsf{ISP}_u(\BlockG)$, then $\alg A \models (\beta_n)$.

Since $\alg A \in \mathsf{H}(\alg B)$, then $\alg A \cong \alg B / F$ for some congruence filter $F$. 
Let $[a]$ and $[b]$ in $B/F$ such that $[a]^n=[a]^{n+1}$ and $[b] \leq [a]^n$. If $[a] = [1]$ then ($\beta_n$) holds with the assignment $x \mapsto [a]$ and $y \mapsto [b]$, so assume $[a] \neq [1]$. We observe that in $\alg B$, $a^{n + 1} = a^{n+2}$. Indeed, suppose by way of contradiction that $a^{n+1}> a^{n+2}$.  Since $[a^n]=[a^{n+2}]$, $a^n \backslash a^{n+2} \in F$. Moreover,
since $a^2 a^n = a^{n+2}$ but $a a^n = a^{n+1} > a^{n+2}$ by the absurdum hypothesis, we get that
 $$a^{2}\leq a^n\backslash a^{n+2}<a.$$ Since $a^n \backslash a^{n+2} \in F$, we get that also $a \in F$ and then $[a] = [1]$, a contradiction. Hence $a^{n + 1} = a^{n+2}$ in $\alg B$.
 
 Moreover, notice that if $[b] < [a]^{n+1}$, necessarily $b < a^{n +1}$ in $\alg B$, since quotients preserve the order, and then $a b = a^{n+1}b$ and $ba = ba^{n+1}$.  
Therefore: $$[a] \cdot [b] = [a \cdot b]=[a^{n+1} \cdot b]= [a]^{n+1} \cdot [b] = [a]^n [b]$$ and similarly $[b] \cdot [a] =[b] \cdot [a]^n$. We can conclude that $\alg A \models (\beta_n)$. 

Hence, all subdirectly irreducible algebras in $\bigvee_{n \in \mathbb{N}} \vv{B_n}$ satisfy $(\beta_n)$, which implies the thesis.
\end{proof}
However, the formulas $\beta_n$ do not hold in all semilinear IRL chains. Hence:
\begin{theorem}\label{theorem: not all SemIRL}
	$\BG \subsetneq \vv{SemIRL}$.
\end{theorem}
\begin{proof}
	We can find a (finite) IRL chain that does not satisfy ($\beta_2$), and hence, by Lemma \ref{lemma: equations ast}, it does not belong to $\BG$. Let us consider the IRL chain $\alg A$ uniquely determined by the following stipulations:
\begin{center}
	\begin{tikzpicture}
		\node[right] (A) at (2,0) {$1$};
		\filldraw (2,0) circle (1pt);
		\filldraw (2, -0.5) circle (1pt);
		\node [right] (B) at (2,-0.5) {$a$};
		\filldraw (2, -1) circle (1pt);
		\node [right] (C) at (2,-1) {$a^2=a^3$};
		\filldraw (2, -1.5) circle (1pt);
		\node [right] (D) at (2,-1.5) {$y$};
		\filldraw (2, -2) circle (1pt);
		\node [right] (E) at (2,-2) {$z=a \cdot y$};
		\filldraw (2, -2.5) circle (1pt);
		\node [right] (F) at (2,-2.5) {$y^2=a^2 \cdot y=y \cdot a= y \cdot a^2$};
	\end{tikzpicture}
\end{center}
	Since $y < a^2=a^3$ but $a \cdot y \neq a^2 \cdot y$, $\alg A \nvDash (\beta_2)$ which implies that $\alg A \not \in \BG$. Hence, we found a finite chain belonging to $\vv{SemIRL}$ but not to $\BG$ which gives us the thesis.
	\end{proof}
We observe that all the results in this section can be obtained also if one considers commutative structures. In particular, for $n \geq 1$, let $\vv{C}\BG_n$ be the commutative subvariety of $\BG_n$:
\begin{corollary}
	$\vv C\BG_n$ is the variety generated by archimedean gluings of $n$-potent commutative chains, and it has the finite embeddability property, for all $n \geq 1$.
	The join $\bigvee_{n \in \mathbb{N}} \vv C\BG_n$ is generated by all archimedean gluings of commutative integral chains, and it is strictly contained in $\vv{SemCIRL}$.
\end{corollary}
\subsection{Axiomatizations of subvarieties via archimedean gluings}\label{subsection:axioms-subvar}
In order to gain a better understanding of the subvariety lattice of $\BG$, we show how one can obtain finite axiomatizations for certain varieties arising naturally from finite archimedean gluings of simple IRL-chains. It is worth observing that not all subvarieties of $\BG$ are finitely based. In fact, there are already non-finitely based subvarieties of basic hoops $\vv{BH}$ (\cite[Example 7.2]{AM2003} shows this in the $0$-bounded case, but it can be straightforwardly adapted), which is a proper subvariety of $\BG$. Indeed, $\vv{BH}$ is generated for instance by the 1-sum of $\omega$ copies of the $0$-free reduct of the standard MV-algebra \cite[Theorem 4.8]{AM2003}, which can be seen as an archimedean gluing. Moreover, there are uncountably many subvarieties of $\vv{BH}$ (see the analogous $0$-bounded case in \cite[Theorem 4.12]{AM2003}, the same proof works for basic hoops), hence also of $\BG$. 

The results in this section take inspiration from the axiomatizations of semilinear varieties generated by 1-sums in \cite{AM2003,GalatosUgolinimanuscript}. In contrast to these works, we here consider the diverse class of constructions given by gluings, hence we do not need the assumption of divisibility or its weakenings, and we do not assume commutativity. However, we do need to restrict to $n$-potent varieties.

We are able to prove in our context the equivalents to the main axiomatization results in \cite{AM2003,GalatosUgolinimanuscript}. In particular, we can:
\begin{enumerate}
	\item define equations that characterize the varieties whose subdirectly irreducible algebras have a bounded number of components;
	\item under certain assumptions, given varieties $\vv V_1, \ldots, \vv V_m$, we can axiomatize the variety generated by archimedean gluings of the kind $\alg A_1 \boxplus \dots \boxplus \alg A_m$ for $\alg A_i\in \vv V_i$.
\end{enumerate} 

Let us start by showing that it is possible to describe equationally the cardinality of the set $I$ over which one decomposes in blockwise gluings of the archimedean components. Recall that by $(\BG_n)_{\vv c}$ we mean the class of chains in the variety $\BG_{n}$. Let then $\alg A = \boxplus_{\alg I}\alg A_i$ for $\alg A \in (\BG_n)_{\vv c}$. We write $a \ll b$ if and only if $a$ is in a lower component than $b$, i.e. 
$$
a \ll b \mbox{ if and only if } a \in A_i-\{1\} \mbox{ and } b \in A_j \mbox{ where } i <j.
$$
The following technical lemma is the key to ``talk about'' the different components in the gluing.
\begin{lemma} \label{lemma: definition of parts of lamndak}Let $\alg A$ be a chain in $\BG_n$, $\alg A = \boxplus_{\alg I} \alg A_i$. The following holds: 
\begin{enumerate}
	\item $(x^n \backslash y ^n)^n \backslash (x \join y)=1$ if and only if $x \gg y$ or at least one between $x$ and $y$ is $1$; 
	\item $(x^n \backslash y ^n)^n \backslash (x \join y)=x \join y$ if and only if $x \ll y$, $x^n=y^n$, or $x = 1$.
\end{enumerate}	
\end{lemma}
\begin{proof}
	(1) For the right to left direction, notice first that if one between $x$ and $y$ is $1$, then the equation trivially holds. While if $x \gg y$, since $\alg A \in (\BG_{n})_{\vv c}$, $x^n \backslash y^n \in \arc {y^n}$ hence $(x^n \backslash y^n)^n=y^n$, and so $(x^n \backslash y^n)^n < (x \join y)=x$, implying the thesis. For the left to right direction, we argue by contraposition. Suppose that none between $x$ and $y$ is $1$, and that $x^n=y^n$ or $x \ll y$. Then, $$(x^n \backslash y ^n)^n \backslash (x \join y)= 1 \backslash (x \join y)= (x \join y) < 1,$$ getting the thesis.

	(2) For the right to left direction, if $x^n=y^n$ or $x \ll y$ then $x^n \backslash y^n=1$ hence the identity becomes $x \lor y = x \lor y$. If $x = 1$ we get $1 = 1$. 
	For the left to right one, we argue by contraposition. Suppose that $x \gg y$ with $x \neq 1$. Then, by point (1), $$(x^n \backslash y ^n)^n \backslash (x \join y)=1 > x \join y,$$ proving the thesis. 
	\end{proof}
The idea now is to use the above lemma to find an equation $\lambda_k^n$ such that a subvariety $\vv V $ of $\BG_{n}$ satisfies it iff every chain in $\vv V$ is an archimedean gluing of at most $k$ components. Let $k \in \mathbb{N}$ with $k > 0$, we set:
\begin{equation*}\tag{$\lambda_k^n$}
		\bigwedge_{i=0}^{k-1} (x_{i+1}^n \backslash x_i^n)^n \backslash (x_{i+1} \vee x_i)\leq \bigvee_{i=0}^{k}x_i.
	\end{equation*}
	
\begin{lemma}\label{lemma: axiomatization via lambdak}
	Let $\alg A$ be a chain in $\BG_n$, $\alg A = \boxplus_{\alg I} \alg A_i$. $\alg A \models \lambda_k^n$ if and only if $|I| \leq k$.
\end{lemma}
\begin{proof}
	Suppose that $|I| \leq k$, i.e., there are at most $k$ components in the blockwise gluing. Thus, if one takes $k+1$ elements in $\alg A$, say $a_0, \dots a_k$, there is $j \in \{0, \ldots, k\}$ such that one of the following holds:
	\begin{itemize}
		\item $a_j = 1$,
		\item $\arc {a_{j}} = \arc {a_{j+1}}$, i.e. $a_j^n = a_{j+1}^n$,
		\item $a_{j +1} \ll a_j$.
	\end{itemize}
	In any case, $(a_{j+1}^n \backslash a_j^n)^n \backslash (a_{j+1} \vee a_j)=a_{j+1}\join a_j$ by Lemma \ref{lemma: definition of parts of lamndak}. Then $$\bigwedge_{i=0}^{k-1} (a_{i+1}^n \backslash a_i^n)^n \backslash (a_{i+1} \vee a_i)\leq a_{j+1}\join a_j \leq \bigvee_{i=0}^{k}a_i,$$ meaning that  $\lambda_k^n$ holds.
	
	Vice versa, suppose $| I | > k$. Then we have at least $k+1$ components in the blockwise gluing. Hence, one can pick $k+1$ elements different from $1$, all from distinct components, and order them from the smallest to the biggest. Precisely, let $a_0, a_1, \dots , a_k \in A$ such that $a_0 \ll a_1 \ll \dots \ll a_k < 1$; then by Lemma \ref{lemma: definition of parts of lamndak} $$\bigwedge_{i=0}^{k-1} (a_{i+1}^n \backslash a_i^n)^n \backslash (a_{i+1}\vee a_i)=1$$ while $\bigvee_{i=0}^{k}a_i < 1$, which implies that $\alg A \nvDash \lambda_k^n$. 
\end{proof}
 Let us recall that an algebra $\alg A$ is {\em generic} for a variety $\vv V$ if it generates $\vv V$; as a consequence of the above lemma, we can get the analogous of the result \cite[Corollary 4.3]{AM2003} for BL-algebras.
 \begin{theorem}
	Let $\alg A$ be a chain in $\BG_n$, $\alg A = \boxplus_{\alg I} \alg A_i$. If $\alg A$ is generic, then $I$ must be infinite. 
\end{theorem}
Moreover, we observe that $\lambda_1^k$ axiomatizes the subvariety $\ssBG_n$ generated by chains with only one archimedean component. Let us call algebras in $\ssBG_n$  {\em semiarchimedean}. Notice that, in commutative chains, having one archimedean component is equivalent to being simple (we remind the reader that an algebra is simple if it has only two congruences). This is not the case without commutativity, since a chain may have more archimedean components but still be simple, because of closure under conjugates of the congruence filters.

We proceed to show how one can obtain finite axiomatizations for varieties generated by finite archimedean gluings, where the components belong to specific varieties of IRLs. The ideas are again inspired by the results in \cite{AM2003,GalatosUgolinimanuscript}, where the considered varieties are generated by 1-sums, and satisfy commutativity and divisibility, or a weak form of the latter. However, while we do restrict to $n$-potent varieties, we do not assume commutativity, and consider a  wider class of constructions which include, but is not restricted to, 1-sums. This introduces some technical complications, and we will need to add some restrictions to the considered varieties. Let us now fix an $n \in \mathbb{N}$, $n \geq 1$. 
\begin{definition}\label{def:gluable}
	Consider $\vv V_1, \ldots, \vv V_m$ to be subvarieties of $\BG_n$. We call  $\vv V_1, \ldots, \vv V_m$ {\em gluable} if, for $i = 1, \ldots, m$:
\begin{enumerate}
	\item $\vv V_i$ is semiarchimedean, i.e., $\vv V_i \sse \ssBG_n$;
	\item $\vv V_i$ is axiomatizable, relatively to $\ssBG_n$, by an identity in one variable $t_i(x) = 1$, with $t_i$ being a $(\backslash,/)$-free term.
\end{enumerate}
\end{definition} 

The first condition ensures that it makes sense to take archimedean gluings where each component belongs to a specific variety, since chains of each $\vv V_i$ will only have one archimedean component. Notice that this condition is verified, for instance, by $n$-potent varieties of MV-algebras or Wajsberg hoops. The second condition has to do with the nature of the blockwise gluing construction; recall indeed that the divisions are not necessarily preserved by the gluing. The need to restrict to 1-variable terms is purely technical (notice that all varieties of MV-algebras can be axiomatized by 1-variable terms).

Let then $\vv V_1, \ldots, \vv V_m$ be gluable subvarieties of $\BG_n$, we define: 
\begin{equation}
	\vv V_1 \boxplus_c \dots  \boxplus_c \vv V_m =\vv V (\{ \boxplus_{\bf m} \alg  A_i : \alg A_i \in (\vv V_i)_{\vv c} \mbox{ for every } i= 1, \dots , m \}).
\end{equation}
We remind the reader from Notation \ref{notation1} that we write $\boxplus_{\bf m} \alg A_i$ for a gluing over the ordered index set $1 < \ldots < m$. 
Notice that chains in $\vv V_1 \boxplus_c \dots  \boxplus_c \vv V_m$ may have less than $m$ components. In order to be able to describe equationally what happens in this situation, we need to consider subsequences of $\{1, \ldots, m\}$.
Let us then define $\vv{SEQ_k}$ as the set of strictly increasing sequences of $k$ natural numbers between $1$ and $m$, i.e. 
$$
\vv{SEQ_k} = \{s=(s_1, \dots , s_k) : 0 < s_1< \dots < s_k \leq m \}.
$$
Then for each $k=2, \dots m$, we consider $k$ variables $x_1, \dots x_k$ and define two terms $p_k$ and $q_k$ as follows: 
\begin{align*}
	& p_k^n := \bigwedge_{i=1}^{k}(x_{i+1}^n \backslash x_i^n)^n \backslash (x_{i+1}\join x_i)\\
	& q_k^n := \bigvee_{i=1}^k x_i \join \bigvee_{s \in \vv{SEQ_k}} \Biggl ( \bigwedge_{j=1}^kt_{s_j}(x_j) \Biggr) 
\end{align*}
Finally, we define the equations $\varepsilon_k$ for $k=1, \dots , m$ as:
\begin{equation}
	\varepsilon_1^n(x):= \bigvee_{i=1}^m t_i(x)=1, \qquad \varepsilon_k^n:= p_k^n \leq q_k^n \;\; \mbox{ for } k \geq 2,
\end{equation}
where we recall that $t_i(x) = 1$ is the axiomatization of $\vv V_i$ relatively to $\ssBG_n$.

While $\varepsilon_1$ is quite self-explanatory, the following lemma provides an intuitive understanding of the identities $\varepsilon_k^n$ for $k \geq 2$. 
\begin{lemma}\label{lemma:epsilonk}
	Let $k \geq 2$. Then a chain $\alg A \in \BG_n$ satisfies $\varepsilon_k^n$ iff, whenever one considers elements $a_1, \ldots, a_k \in A$ such that $a_1 \ll \ldots \ll a_k < 1$, there is a sequence $s=(s_1, \dots, s_k)\in \vv{SEQ_k}$ such that $t_{s_j}(a_j)=1$ for $j=1, \dots , k$.
\end{lemma}
\begin{proof}
Let $\alg A \in \BG_n$. We first show that if one takes $a_1, \dots , a_k \in A$ that do not satisfy the hypothesis $a_1 \ll \ldots \ll a_k < 1$,  $\varepsilon_k^n$ holds with the assignment $x_i \mapsto a_i$.  
	First, if one of the $a_i$ is equal to $1$, then $\bigvee_{i=1}^ka_i=1$, and so $p_k(a_1, \dots, a_k)\leq q_k(a_1, \dots , a_k)=1.$
	Moreover, if there is an $a_i$ such that either $a_{i+1} \ll a_i$ or $a_{i+1}^n=a_i^n$, then, by Lemma \ref{lemma: definition of parts of lamndak}, $$p_k(a_1, \dots, a_k)\leq a_i \join a_{i+1}\leq \bigvee_{i=1}^k a_i \leq q_k(a_1, \dots, a_k).$$ 
	
	Hence, $\alg A \models \varepsilon_k^n$ iff the identity holds whenever assigning the variables to elements $a_1, \dots , a_k \in A$ such that $a_1 \ll a_2 \ll \dots \ll a_k < 1$. Notice that in this case, $p_k(a_1, \dots, a_k) = 1$ by Lemma \ref{lemma: definition of parts of lamndak}, and since none of the variables is set to $1$, $\bigvee_{i=1}^k a_i < 1$. Thus, $p_k(a_1, \dots, a_k) \leq q_k(a_1, \dots, a_k)$ iff $\bigvee_{s \in \vv{SEQ_k}}( \bigwedge_{j=1}^kt_{s_j}(x_j) )  = 1$, which is equivalent to the existence of a sequence $s=(s_1, \dots, s_k)\in \vv{SEQ_k}$ such that $t_{s_j}(a_j)=1$ for $j=1, \dots , k$. This completes the proof.
	\end{proof}
In order to be able to prove the main results of this subsection, it will be important to establish that the subvariety of $\BG_n$ axiomatized by $\lambda_m^n$ and $\varepsilon_k^n$ for $k=1, \dots, m$ has the finite model property, and as such, it is generated by its finite chains. To do so, we need a technical lemma.
\begin{lemma}\label{lemma: FEP K}
Let $\vv V_1, \ldots, \vv V_m$ be gluable subvarieties of $\BG_n$. Let $\alg A$ be a chain in $\BG_n$, $\alg A = \boxplus_{\bf m} \alg A_i$, such that $\alg A_i \in \vv V_i$ for $i=1 , \dots , m$. Then for any finite partial subalgebra $\alg B$ of $\alg A$, $\alg W_{\alg A, \alg B}^+ \in (\BG_n)_{\rm fin}$ and it satisfies $\lambda_m^n$ and $\varepsilon_k^n$ for $k=1, \dots, m$.
\end{lemma}
\begin{proof}
	The fact that $\alg W_{\alg A, \alg B }^+ \in (\BG_n)_{\rm fin}$ follows from Lemma \ref{theorem: preparation to FMP}. It is left to show that $\alg W_{\alg A, \alg B}^+ \models \lambda_m^n$ and $\alg W_{\alg A, \alg B}^+ \models \varepsilon_k^n$ for $k= 1 , \dots, m$. We prepare.	
	First recall that the elements in $\alg W_{\alg A, \alg B}^+$ are of the form $$\{(u,b)\}^\triangleleft=\{ x \in W : u(x) \leq_{\alg A} b \}$$ which are downsets of $W$ seen in $\alg A$. It can be easily shown that they are principal downsets: $$\{(u,b)\}^\triangleleft={\downarrow}_{\alg A} m \cap W, \mbox{ with } m =\max \{ x \in W : u(x) \leq_{\alg A} b \}.$$ 
	Indeed, $W$ is the submonoid of $\alg A$ generated by $\alg B$, so if $B = \{b_1, \ldots, b_q\}$, all its elements are of the form $b_1^{k_1} \cdot \ldots \cdot b_q^{k_q}$. Hence $W$ is a discretly totally ordered set whose elements are all below those of $B$, and therefore satisfies the ascending chain condition. Hence every nonempty set of its elements has a maximum, thus in particular the maximum of $\{x \in W : u(x) \leq b \}$ exists; let us call such maximum $m$. The fact that $\{(u,b)\}^\triangleleft = {\downarrow m}\cap W$ then follows from the definition of $\{(u,b)\}^\triangleleft$. 
		
	Moreover, it can be shown that if $X, Z \in \alg W_{\alg A, \alg B}^+$ where $X={\downarrow}_{\alg A} x \cap W$ and $Z= {\downarrow}_{\alg A} z \cap W$ with $Z \ll X$, then $z \ll x$ in $\alg A$. Indeed if $x^n \leq z$ in $\alg A$, then $X \circ^{n} X \subseteq Z$. This implies that $$X^n= \gamma_N(X \circ^{n} X) \subseteq \gamma_N(Z)=Z,$$ where the first equality follows from Lemma \ref{lemma: powers of sets}, and the inclusion from the fact that $\gamma_N$ is a nucleus hence order preserving. This contradicts the hypothesis, hence $Z \ll X$ implies $z \ll x$.
	
Let us indicate with $\top$ the $1$ of the algebra $\alg W_{\alg A, \alg B}^+$ to avoid confusion in what follows. We are ready to show that $\alg W_{\alg A, \alg B}^+ \models \lambda_m^n$. Suppose by way of contradiction that $\alg W_{\alg A, \alg B}^+\nvDash \lambda_m^n$, then $\alg W_{\alg A, \alg B}^+$ is blockwise gluing of at least $m+1$ components. Hence, we can find a chain of $k+1$ elements all belonging to different components and all different from $\top$, i.e. $W_1 \ll \dots \ll W_k < \top$ where $W_i={\downarrow}_{\alg A}w_i \cap W$ with $w_i \in W$. Thus, projecting back to $\alg A$, we have a the sequence $w_1 \ll w_2 \ll \dots \ll w_k < 1$ having $k+1$ elements all belonging to different components and all different from $1$ which implies that $\alg A \nvDash \lambda_m^n$, contradicting the hypothesis.

Let us now verify that $\alg W_{\alg A, \alg B}^+ \models \varepsilon_k^n$ with $k = 1 , \dots, m$. First suppose $k=1$ and let $Y={\downarrow}_\alg A y\cap W$ be an element of $\alg W_{\alg A, \alg B}^+$. Projecting $Y$ back to $\alg A$, we have that $y$ belongs to a certain archimedean component of $\alg A$, say $\alg A_i \in \vv V_i$; then $t_i(y)=1$. We prove that $t_i(y)=1$ implies $t_i(Y)=\top$ by induction on the complexity of the term $t_i$.

For the base case, if $t_i(x)=1$, then there is nothing to prove; if $t_i(x)=x$, then $t_i(y)=1$ means that $y=1$ and since $Y={\downarrow}_\alg A 1 \cap W$, $t_i(Y)=\top$.

For the inductive case, suppose that $t_i(y)=1$ implies that $t_i(Y)=\top$ for terms of complexity $n$; we prove it for $n+1$. Recalling that the terms $t_i$ of interest are divisions free, we need to consider the following cases:
\begin{itemize}
\item If $t_i(x)=t_1(x)\star t_2(x)$, for $\star \in \{\cdot, \land\}$, then $t_i(y)=t_1(y)\star t_2(y)=1$ implies that $t_1(y)=1$ and $t_2(y)=1$, so, by inductive hypothesis, $t_i(Y)=t_1(Y)\star t_2(Y)=\top \star \top =\top$.
\item If $t_i(x)=t_1(x)\join t_2(x)$, then $t_i(y)=t_1(y)\join t_2(y)=1$ implies that either $t_1(y)=1$ or $t_2(y)=1$. Then by inductive hypothesis, either $t_1(Y)=1$ or $t_2(Y)=1$, in any case $t_i(Y)=t_1(Y)\join t_2(Y)=\top$.
\end{itemize}
Hence, $\alg W_{\alg A, \alg B}^+ \models \varepsilon_1(x)$. 

Let now $2 \leq k \leq m$ and let $W_1, \dots, W_k \in \alg W_{\alg A,\alg B}^+$. By Lemma \ref{lemma:epsilonk}, it suffices to check that $\varepsilon_k^n$ holds for an assignment $x_i \mapsto W_i$ with $W_1 \ll \dots \ll W_k < \top$; precisely, it suffices to check that there is a sequence $s \in \vv{SEQ}_k$ such that $\bigwedge_{j=0}^k t_{s_j}(W_j)=\top$. Projecting back to $\alg A$, we have a sequence $w_1 \ll w_2 \ll \dots \ll w_k < 1$ of $k$ elements all belonging to different archimedean components and since $\alg A \models \varepsilon_k^n$, there exists a sequence $s \in \vv{SEQ}_k$ such that $t_{s_i}(w_{s_i})=1$. We have already shown above that $t(w_i)=1$ implies $t(W_i)=\top$. 
Hence, we have found a sequence $s \in \vv{SEQ}_k$ such that $\bigwedge_{j=0}^k t_{s_j}(W_j)=\top$, which implies that $\alg W_{\alg A, \alg B}^+ \models \varepsilon_k^n$ for $k=1, \dots , m$. This completes the proof.
\end{proof}
As a direct consequence:
\begin{proposition}\label{corollary: FMP K}
	The subvariety of $\BG_n$ axiomatized by $\lambda_m^n$ and $\varepsilon_k^n$ for $k=1, \dots, m$ has the finite embeddability property.
\end{proposition}
We are now ready to state the main result of this subsection. The following proof is a non-trivial adaptation of the one in \cite{GalatosUgolinimanuscript}, where the authors axiomatize varieties that satisfy a weak notion of divisibility, and which fixes some (minor) issues in the proof of the axiomatizations for subvarieties of BL-algebras in \cite[Theorem 5.1]{AM2003}. We notice in particular that the situation in this case is complicated by the fact that we are dealing with (possibly) partial algebras. To bypass this issue, we utilize the above Proposition \ref{corollary: FMP K}. 
\begin{theorem}\label{theorem: axiomatization of unbounded subvarieties}
	Let $\vv V_1, \dots , \vv V_m$ be gluable subvarieties of $\BG_n$. Then, $\vv V_1 \boxplus_c \dots \boxplus_c \vv V_m$ is axiomatized, relatively to $\BG_n$, by $\lambda_m^n$ and $\varepsilon_k^n$ for $k=1, \dots, m$.
\end{theorem}
\begin{proof}
	First observe that, by Lemma \ref{lemma: axiomatization via lambdak}, $\lambda_m$ is a necessary condition to belong to the variety $\vv V_1 \boxplus_c \dots \boxplus_c \vv V_m$. Note that then, one only needs to check that for every algebra $\alg A \in \BG_n$ such that $\alg A \models \lambda_m$, 
	\begin{center}
		$\alg A$ belongs to $\vv V_1 \boxplus_c \dots \boxplus_c \vv V_m$ if and only if $\alg A \models \varepsilon_k^n$ for all $k = 1 \ldots m$. 
	\end{center}
	Observe that it is sufficient to check the left-to-right direction on the generators of the variety. Let $\alg A$ be a generator, then it is of the form $\alg A= \boxplus_{\bf m}\alg A_i$, where each $\alg A_i$ is a chain in $\vv V_i$ for $i=1, \dots , m$. 
	
	Since each $\vv V_i \models \lambda_1$, each $\alg A_i$ only has one archimedean component. Moreover, since each $\alg A_i \in \vv V_i$, it is clear that $\varepsilon_1^n$ holds.
	Let now $2 \leq k \leq m$. By Lemma \ref{lemma:epsilonk}, it suffices to consider $a_1 \ll a_2 \ll \dots \ll a_k < 1$, and show that there is a sequence of $k$ increasing natural numbers $s=(s_1, \dots, s_k)\in \vv {SEQ_K}$ such that $t_{s_j}(a_j)=1$ for $j=1, \dots , k$. This clearly holds, again due to the fact each $\alg A_i \in \vv V_i$. We have shown that $\alg A \models \varepsilon_k^n$ for $k=1, \dots, m$, for each generator $\alg A$. Hence
	 $\vv V_1 \boxplus_c \dots \boxplus_c \vv V_m\models \varepsilon_k^n$, for $k=1. \dots, m$.
	
	For the right-to-left direction, let us call $\vv K$ the subvariety of $\BG_n$ axiomatized by $\lambda_m^n$ and $\varepsilon_k^n$ for $k = 1 \ldots m$. It suffices to check that a set of generators for $\vv K$ belong to $\vv V_1 \boxplus_c \dots \boxplus_c \vv V_m$. Since $\vv K$ has the finite embeddability property (Proposition \ref{corollary: FMP K}), and therefore the finite model property, one only needs to check that if $\alg B$ is a finite chain belonging to $\vv K$, then $\alg B \in \vv V_1 \boxplus_c \dots \boxplus_c \vv V_m$. We prove this by contrapositive.
	
 Suppose $\alg B \in \BG_n$ is a finite chain such that $\alg B \models \lambda_m$ but $\alg B \not \in \vv V_1 \boxplus_c \dots \boxplus_c \vv V_m$. The goal is to show that $\alg B \nvDash \varepsilon_k^n$ for some $k \in \{1, \ldots, m\}$. Say that $\alg B = \boxplus_{\bf k} \alg B_i$, for some $k \in \{1, \ldots, m\}$, since $\alg B \models \lambda_m$. The idea is to show that for any $i= 1, \dots, k$, one can pick a $b_i \in \alg B_i$ such that if each $x_i$ is assigned to $b_i$, $\varepsilon_k^n$ is falsified by this assignment.
	
	Since $\alg B$ is finite, then every $\alg B_i$, for $i= 1, \dots , k$, is also finite. Now, each $\alg B_i$ seen as a subreduct of $\alg B$ may be a partial algebra, but each $B_i$ is in fact the domain of a residuated lattice whose operations extend those of $\alg B_i$. Indeed, given that $\alg B_i$ is a finite chain whose product is order preserving, one can define the missing divisions in the obvious way:
	$$x \backslash y := \max \{ z \in B_i : x \cdot z \leq y \}, \;\; y /x := \max \{ z \in B_i : z \cdot x \leq y \}.$$
	Let us denote with $\alg B_i^{+}$ this (total) residuated lattice associated to $\alg B_i$ over the same domain $B_i$.
	  Consider the family of blocks $\{(\alg B_i, \betaij)\}_{i<_{\bf k}j}$ associated to $\alg B$ provided by Proposition \ref{proposition: a in BG}. Then, the same operators yield a family of blocks with the new total algebras over the same domains, $\{(\alg B^+_i, \betaij)\}_{i<_{\bf k}j}$, and it is straighforward that $\alg B =  \boxplus_{\bf k} \alg B_i^+$. In other words, we can consider $\alg B$ as the blockwise gluing of the total algebras $\alg B_1^+, \dots , \alg B_k^+$. 
	  
	  We are now ready to pick the elements $b_i$ so that the assignment $x_i \mapsto b_i$ makes $\varepsilon_k^n$ fail. First, for each $i=1, \dots, k$, let 
	$$N_i=\{\vv V_j : \alg B^+_i \not \in \vv V_j \mbox{ for some } j= 1 , \dots , m \}.$$ 
	Observe that it is possible that some of the $N_i$ are empty, but surely not all of them since $\alg B \not \in \vv V_1 \boxplus_c \dots \boxplus_c \vv V_m$. In case $N_i= \emptyset$, than one can choose whatever $b_i\in B_i$ different from $1$. Otherwise, consider the join of the varieties in $N_i$, $\vv N_i$ in what follows, and define
	$$
	n_i(x)= \bigvee \{ t_j(x) : \vv V_j \in N_i \}
	$$
	where $t_j(x)$ is the term axiomatizing $\vv V_j$. $\vv N_i$ is the variety axiomatized by $n_i(x)=1$ \cite[Corollary 4.3]{Galatos04}. Note that $\alg B^+_i \not \in \vv N_i$. Indeed since $\alg B^+_i$ is subdirectly irreducible, as a consequence of J\'onsson Lemma, if it belongs to the join of finitely many varieties then it should belong to one of them; but since we know that $\alg B_i$ does not belong to any of the $\vv V_j \in N_i$, we get that $\alg B_i \not \in \vv N_i$. Thus there exists a $b_i\neq 1 \in \alg B^+_i$ such that $n_i(b_i)\neq 1$, and so such that $t_j(b_i) \neq 1$ for all $j$ such that $\alg B_i^+ \notin \vv V_j$; let us choose exactly this $b_i$. 
	
We claim that $\varepsilon_k^n$ fails with the assignment mapping each $x_i$ to $b_i$. Notice first that if $k=1$, $N_1$ necessarily contains all $\vv V_j$ for $j = 1 \ldots m$, which means that $\varepsilon_1(b_1)=n_1(b_1) \neq 1$.

Let now $k \geq 2$. Observe that $b_1 \ll \dots \ll b_k <1$, thus, by Lemma \ref{lemma:epsilonk}, it suffices to show that $\bigvee_{s \in \vv{SEQ_k}}\Bigl (\bigwedge_{j=1}^k t_{s_j}(b_j) \Bigl ) \neq 1$. Since $\alg B \not \in \vv V_1 \boxplus_c \dots \boxplus_c \vv V_m$, for any sequence $s=(s_1, \dots , s_k) \in \vv{SEQ_k}$ there is an $s_i$ such that $\alg B_i^+ \not \in \vv V_{s_i}$, and so $t_{s_i}(b_i)\neq 1$. This implies that also $\bigvee_{s \in \vv{SEQ_k}}\Bigl (\bigwedge_{j=1}^k t_{s_j}(b_j) \Bigl ) \neq 1$, hence $\varepsilon_k^n$ fails with the chosen assignment. 
This completes the proof. 
\end{proof}
\begin{remark}
	We observe that one could rephrase the same results if all the gluable varieties $\vv V_1, \ldots, \vv V_m$ belong to a common subvariety $\vv W$ of $\BG_n$; then the requirements of Definition \ref{def:gluable} for $\vv V_1, \ldots, \vv V_m$ have to be satisfied relatively to $\vv W$. So for instance, the same results hold if one consider commutative varieties. We did not phrase the results in this more general case to avoid making the statements and definitions harder to read.
\end{remark}
\subsection{The $0$-bounded case}
All the results contained in this section can be adapted to the $0$-bounded case. Notice that here, if $\alg A$ is a chain in $\FL_w$ and it is a blockwise gluing $\alg A = \boxplus_{\bf I} \alg A_i$, the index set $\bf I$ has a minimum $i_0$ in the order, and $0 \in A_{i_0}$. We call such a gluing {\em $0$-bounded}. With this in mind, we can obtain all the same results as in Subsection \ref{subsection:axiom-npotent}. In particular, let $\BG_n^0$ be the subvariety of $\vv{SemFL}_w$ axiomatized by (block$_n$) for $n \geq 1$, and $\vv{C}\BG_n^0$ its commutative subvariety. Then:
\begin{corollary}
	$(\vv C)\BG_n^0$ is the variety generated by $0$-bounded archimedean gluings of $n$-potent (commutative) chains, and it has the finite embeddability property, for all $n \geq 1$.
	The join $\bigvee_{n \in \mathbb{N}} (\vv C)\BG_n^0$ is generated by all $0$-bounded archimedean gluings of (commutative) chains, and it is strictly contained in $\vv{SemFL}_w$ ($\vv{MTL}$).
\end{corollary}

The results in Subsection \ref{subsection:axioms-subvar} can also be adapted to the $0$-bounded case, but some proofs are required. 
In particular, another important difference with respect to 1-sums becomes apparent here. In $0$-bounded 1-sums, the lowest component can be easily identified by the fact that it contains all and only the elements such that $\neg\neg x = x$ (where we mean $\neg x:= x \backslash 0$). This is not the case for blockwise gluings. Nonetheless, we show that one can still identify equationally the lowest component in archimedean gluings. Let us fix an $n \in \mathbb{N}$. First, notice that:
\begin{proposition}
	Let $\alg A$ be a chain in $\BG_n^0$, $\alg A = \boxplus_{\bf I} \alg A_i$; then $\alg A \models \lambda_k^n$ if and only if $|I| \leq k$. Moreover, if $\alg A$ is generic for  $\BG_n^0$, $I$ must be infinite.
\end{proposition}
Let us then call $\ssBG_n^0$ the subvariety of $\BG^0_n$ axiomatized by $\lambda_1$, whose chains only have one archimedean component.
\begin{lemma}\label{lemma: preservation of the first component}
Let $\vv V$ be a subvariety of $\BG_n^0$ generated by a class of $0$-bounded blockwise gluings whose lowest component belongs to a variety $\vv W \sse \ssBG_n^0$, and suppose that $\vv W$ is axiomatized, relatively to $\ssBG_n^0$, by an equation in one variable $t_0(x)=1$. Then $\vv V \models (\neg x^n)^n \backslash t_0(x) = 1$.
\end{lemma}
\begin{proof}
It suffices to show that the generators of $\vv V$ satisfy the desired identity. Let then $\alg A$ be a $0$-bounded blockwise gluing, $\alg A = \boxplus_{\bf I} \alg A_i$, where the lowest component $\alg A_{i_0}$ belongs to $\vv W$. Since $\vv V \sse \BG_n^0$, recall that the $A_i$s are the archimedean components of $\alg A$.

Let then $x \in A$. If $x \in A_j$ with $j \neq i_0$, we have that $0 \ll x$. Then also $0 \ll x^n$, and so necessarily $\arc {\neg x^n} = \arc 0$. Therefore, given the $n$-potency of $\alg A$, $(\neg x^n)^n = 0$ hence $$(\neg x^n)^n \backslash t_0(x) = 0 \backslash t_0(x) = 1.$$ 

If instead $x \in A_{i_0}$, then since $\alg A_{i_0} \in \vv W$, $t_0(x) = 1$ hence $(\neg x^n)^n \backslash t_0(x) = 1$. This proves that $\alg A \models (\neg x^n)^n \backslash t_0(x) = 1$, which implies the thesis.
\end{proof}

Thanks to the above lemma, we can adjust the machinery of the previous subsection to the $0$-bounded case. 
\begin{definition}\label{def:gluable0}
	Consider $\vv V_1, \ldots, \vv V_m$ to be subvarieties of $\BG_n$ and $\vv V_0$ a variety of $\BG_n^0$. We call  $\vv V_0, \ldots, \vv V_m$ {\em gluable} if, for $i = 1, \ldots, m$:
\begin{enumerate}
	\item $\vv V_1 \ldots \vv V_m$ are gluable in the sense of Definition \ref{def:gluable};
	\item $\vv V_0\sse \ssBG^0_n$ and it is axiomatized, relatively to $\ssBG_n^0$, by an identity in one variable $t_i(x) = 1$, with $t_i$ being a $(\backslash,/)$-free term.
\end{enumerate}
\end{definition} 
Fix then $\vv V_1, \ldots, \vv V_m$ subvarieties of $\BG_n$ and $\vv V_0$ a subvariety of $\BG_n^0$, with $\vv V_0, \ldots, \vv V_m$  gluable as in the above definition.
We now consider increasing sequences of $k+1$ natural numbers between $0$ and $m$:
$$
\vv{SEq_{k,0}^n}=\{ s=(s_0, \dots, s_k) : 0 \leq s_0 <s_1 < \dots < s_k \leq m \}.
$$
Then, for $k=1, \dots , m$, let
\begin{align*}
	& p_{k,0}^n := \bigwedge_{i=0}^{k}(x_{i+1}^n \backslash x_i^n)^n \backslash (x_{i+1}\join x_i)\\
	& q_{k,0}^n := \bigvee_{i=0}^k x_i \join \bigvee_{s \in \vv{SEQ_k}} \Biggl ( \bigwedge_{j=0}^kt_{s_j}(x_j) \Biggr) 
\end{align*}
We are ready to define the equations $\varepsilon_k$ for $k=0, \dots, m$:
\begin{equation}
	\varepsilon_{0,0}^n(x):=(\neg x^n)^n \backslash t_0(x) = 1; \qquad \varepsilon_{k,0}^n:= p_{k,0}^n \leq q_{k,0}^n \;\;\mbox{ for } 1\leq k \leq m.
\end{equation} 
We prove the analogous result to Proposition \ref{corollary: FMP K}.
\begin{proposition}\label{corollary: FMP KO}
	The subvariety of $\BG_n^0$ axiomatized by $\lambda_{m+1}^n$ and $\varepsilon_{k,0}^n$ for $k=0, \dots, m$ has the finite embeddability property.
\end{proposition}
\begin{proof}
Let $\vv K$ be the subvariety of $\BG_n^0$ axiomatized by $\lambda_{m+1}^n$ and $\varepsilon_{k,0}^n$ for $k=0, \dots, m$.
It suffices to prove that if one considers a chain $\alg A\in \vv K$, and any finite partial subalgebra $\alg B$ of $\alg A$ which contains $0$, then $\alg W_{\alg A, \alg B}^+ \in \vv K_{\rm fin}$. 

	The fact that $\alg W_{\alg A, \alg B}^+$ is a chain in $(\BG_n^0)_{\rm fin}$ follows from Lemma \ref{theorem: preparation to FMP}. Hence one only needs to prove that $\alg W_{\alg A, \alg B}^+$ satisfies $\lambda_{m+1}^n$ and $\varepsilon_{k,0}^n$ for $k=0, \dots, m$. 
	The proof of the fact that $\alg W_{\alg A, \alg B}^+ \models \lambda_{m+1}^n$, and that $\alg W_{\alg A, \alg B}^+ \models \varepsilon_{k,0}^n$ for $k=1, \dots, m$, is as in Lemma \ref{lemma: FEP K}.
	
	It is left to prove that $\alg W_{\alg A, \alg B}^+\models \varepsilon_{0,0}^n(X)$.	
	Notice that since $0 \in \alg B$, there is at least an $X \in \alg W_{\alg A, \alg B}^+$ which is of the form $X={\downarrow}_\alg A x \cap W$ with $x^n=0$ in $\alg A$. Moreover, the $0$-bound of the algebra $\alg W_{\alg A, \alg B}^+$ is given by ${\downarrow}_\alg A 0 \cap W= \{0\}$; let us denote this element by $\bot$. 
	
	Consider an element $X \in \alg W_{\alg A, \alg B}^+$ where $X={\downarrow}_\alg A x \cap W$. Then either $X^n=\bot$ or $X^n \gg \bot$. 
	In the first case, projecting back to $\alg A$, we get that $x^n=0$. Indeed, notice that $$X^n= \gamma_N(X \circ^{n} X)$$ by Lemma \ref{lemma: powers of sets}, hence, since $\bot = \{0\}$, it follows that for all $a \in X \circ^{n} X$, $a = 0$. In particular, $x^n = 0$.
	 Therefore, $x \in A_0 \in \vv V_0$ and so $t_0(x)=1$; as in the proof of Lemma \ref{lemma: FEP K} it can be shown that this implies that $t_0(X)=\top$ and so $(\neg X^n)^n \backslash t_0(X)=\top$. 
	 
	In the case $X^n \gg \bot$, since $\alg W_{\alg A, \alg B}^+ \in \BG_n^0$, we get that $(\neg X^n)^n=\bot$. Therefore, $(\neg X^n)^n \backslash t_0(X)=\top$ and we have proved that $\alg W_{\alg A, \alg B}^+ \models \varepsilon_{0,0}^n$. The proof is complete.
\end{proof}
Finally:
\begin{theorem}
Let $\vv V_1, \ldots, \vv V_m$ be subvarieties of $\BG_n$ and $\vv V_0$ a subvariety of $\BG_n^0$, with $\vv V_0, \ldots, \vv V_m$  gluable.
Then, $\vv V_0 \boxplus_c \dots \boxplus_c \vv V_m$ is axiomatized, relatively to $\BG_n^0$, by $\lambda_{m+1}^n$ and $\varepsilon_{k,0}^n$ for $k=0, \dots, m$.
\end{theorem}
\begin{proof}
	First observe that by Lemma \ref{lemma: axiomatization via lambdak}, $\lambda_{m+1}$ is a necessary condition to belong to the variety $\vv V_0 \boxplus_c \dots \boxplus_c \vv V_m$. Hence it suffices to check that given any $\alg A \in \BG_n^0$ such that $\alg A \models \lambda_{m+1}^n$, $\alg A$ belongs to $\vv V_0 \boxplus_c \dots \boxplus_c \vv V_m$ if and only if $\alg A \models \varepsilon_{k,0}^n$ for $k=0, \dots, m$. 
	
	For the left-to-right direction, the proof is the same as in Theorem \ref{theorem: axiomatization of unbounded subvarieties}, modulo Lemma \ref{lemma: preservation of the first component} which covers the case $k = 0$. 
	For the converse direction, call $\vv K_0$  the subvariety of $\BG_n^0$ axiomatized by $\lambda_{m+1}^n$ and $\varepsilon_{k,0}^n$ for $k=0, \dots, m$. It suffices to check that the generators of $\vv K_0$ belong to  $\vv V_0 \boxplus_c \dots \boxplus_c \vv V_m$. In particular, since $\vv K_0$ has the finite embeddability property (Corollary \ref{corollary: FMP KO}) it is enough to show that if $\alg B$ is a finite chain in $\vv K_0$, then it belongs to $\vv V_0 \boxplus_c \dots \boxplus_c \vv V_m$. We prove it by contrapositive.
	
	Consider $\alg B = \boxplus_{\bf I} \alg B_i \in \BG_n^0$ to be a finite chain which satisfies $\lambda_{m+1}^n$ but $\alg B \not \in \vv V_0 \boxplus_c \dots \boxplus_c \vv V_m$. As in the proof of Theorem \ref{theorem: axiomatization of unbounded subvarieties}, since $\alg B$ is finite, it can be seen as a gluing of total algebras $\alg B = \boxplus_{\bf I}\alg B_i ^+ $ with $|I| \leq m + 1$, where each $\alg B_i^+$ is the finite total algebra uniquely associated to $\alg B_i$. 
	
	Once again the idea is to find an assignment $x_i \mapsto b_i$ which makes $\varepsilon_{k,0}^n$ fail. If $\alg B_0^+ \not \in \vv V_0$, then there is an element $b_0 \in B_0$ such that $t_0(b_0) \neq 1$. Notice that in $\alg B$, $(\neg b_0^n)^n = (\neg 0)^n = 1$, thus: $$(\neg b_0^n)^n \backslash t_0(b_0) = 1 \backslash t_0(b_0) = t_0(b_0) \neq 1.$$
	Thus $\alg B \nvDash \varepsilon_{0,0}^n(x)$.
If instead $\alg B_0^+ \in \vv V_0$, one can proceed as in the proof of Theorem \ref{theorem: axiomatization of unbounded subvarieties} to obtain an assignment that makes $\varepsilon_{|I|,0}^n$ fail. This completes the proof.
\end{proof}
\begin{remark}
	We observe that, similarly to the unbounded case, one can obtain the same axiomatizations results if, considering an equation $\varepsilon$ in the $0$-free signature, $\vv V_0 \in \BG_n^0 + \varepsilon$ and $\vv V_1 \ldots, \vv V_n \in \BG_n + \varepsilon$. In this case the requirements of Definition \ref{def:gluable0} have to be satisfied relatively to $\BG_n + \varepsilon$ and $\BG_n^0 + \varepsilon$. So for instance, the same results hold if one consider commutative varieties. We did not phrase the results in this generality to help readability.
\end{remark}

We conclude this section by highlighting another important consequence of the FEP for the varieties for which we also provided a finite axiomatization. 
Indeed, all finitely axiomatizable varieties with the FEP have a decidable universal theory (see for instance \cite[\S 6.5]{GJKO}). Therefore:
\begin{corollary}
	The following varieties have a decidable universal theory:
	\begin{enumerate}
		\item $\BG_n$ for all $n \geq 1$, their $0$-bounded versions, and commutative subvarieties;
		\item $\vv V_1 \boxplus_c \ldots \boxplus_c \vv V_n$ for any gluable $\vv V_1 \ldots \vv V_n$ in $\BG_n$, for all $n \geq 1$, their $0$-bounded versions, and commutative subvarieties.
	\end{enumerate}
\end{corollary}

\section{Amalgamation failures}
In this final section we use blockwise gluings to construct some chains which provide a counterexample to the amalgamation property in several important varieties of residuated lattices. We solve in the negative several open problems in the literature; in particular, we show that the following varieties do not have the amalgamation property: $\vv{MTL}$, $\vv{SemFL}_w$, and their 0-free subreducts, $\vv{Sem(C)IRL}$. For the commutative varieties this also entails the failure of the deductive interpolation property in the corresponding substructural logics. We proceed by first recalling some needed definitions and facts about the amalgamation property and its connection to deductive interpolation.

\subsection{Amalgamation, essential amalgamation, and deductive interpolation}
First, let us recall the definition of an amalgam, and the amalgamation property for a class of algebras. 

Let $\vv K$ be a class of algebras of the same signature. A {\em V-formation in $\vv K$} is a tuple $(\alg A, \alg B,\alg C,i,j)$ where $\alg A, \alg B, \alg C \in \vv K$ and $i,j$ are embeddings of $\alg A$ into $\alg B,\alg C$, respectively. Given a V-formation $(\alg A, \alg B,\alg C,i,j)$ in $\vv K$, $(\alg D,h,k)$ is said to be an \emph{amalgam of $(\alg A, \alg B,\alg C,i,j)$ in $\vv K$} if $\alg D \in \vv K$ and $h, k$ are embeddings of $\alg B, \alg C$, respectively, into $\alg D$ such that the compositions $h \circ i$ and $k\circ j$ coincide: 
\begin{figure}[h!]
\begin{center}
\begin{tikzpicture}
\node at (0,0) {$\alg A$};
\node at (2,1) {$\alg B$};
\node at (2,-1) {$\alg C$};
\node at (4, 0) {$\alg D$};
\draw [>->] (0.3, 0.1) -- (1.8, 0.8);
\draw [>->] (0.3, -0.1) -- (1.8, -0.8);
\draw [>->] (2.2, 0.8) -- (3.7, 0.1);
\draw [>->] (2.2, -0.8) -- (3.7, -0.1);
\node at (1, 0.65) {\footnotesize{$i$}};
\node at (3, 0.65) {\footnotesize{$h$}};
\node at (1, -0.65) {\footnotesize{$j$}};
\node at (3, -0.7) {\footnotesize{$k$}};
 \end{tikzpicture}\end{center} \end{figure}
 
$\vv K$ has the {\em amalgamation property} (AP) if each V-formation in $\vv K$ has an amalgam in $\vv K$.

In general, establishing the amalgamation property of a class of algebras can be a difficult task, and it is necessary to have a good understanding of the structures involved. In the framework of residuated lattices, some of the most relevant varieties with the AP are the following: C(I)RLs and FL$_{e(w)}$-algebras (since they have the Craig interpolation property, see \cite{Ono} and \cite[Lemma 8.4.2, Theorem 8.4.3]{MPT2023}), abelian lattice-ordered groups \cite{Pierce}, BL-algebras \cite{Montagna06}. However, several important varieties are known to lack the AP: semilinear RLs \cite{GilferezLeddaTsinakis2015}, semilinear CRLs \cite{FussnerSantschi24}, lattice-ordered groups \cite{Pierce}, and also the variety of all residuated lattices, as shown in the recent preprint \cite{JipsenSantschi}. In this work, we show that also MTL-algebras and semilinear (C)(I)RLs lack the AP.

In order to obtain our results, it is convenient to make use of a necessary condition for a variety to have the amalgamation property extracted in \cite{FussnerSantschi24}. The authors, further elaborating on results in \cite{MMT2014,FussnerMetcalfe2024}, show that in some cases the study of the amalgamation property of a variety can be reduced to the study of a weaker property, namely {\em essential amalgamation}, in the class of its subdirectly irreducible members.

An embedding $e: \alg A \to \alg B$ is called {\em essential} if for each homomorphism $h: \alg B \to \alg C$, whenever $h \circ e $ is an embedding, then $h$ is an embedding. If $e : \alg A \to \alg B$ is the inclusion map, then $e$ is essential if and only if for each $\theta \in \vv{Con}(\alg B)$ with $\theta \neq \Delta_{\alg B}$, also $\theta \cap A^2 \neq \Delta_{\alg A}$. A V-formation $(\alg A, \alg B, \alg C, i, j)$ is called an {\em essential span} if $j$ is an essential embedding. Moreover, we say that a class $\vv K$ has the {\em essential amalgamation property} (or EAP) if each essential span in $\vv K$ has an amalgam in $\vv K$.
Given a variety $\vv V$, let $\vv V_{\si}$ be the class of subdirectly irreducible members of $\vv V$. 
\begin{theorem}[{{\cite[Corollary 3.10]{FussnerSantschi24}}}]\label{thm:fussnersantchi}
	Let $\vv V$ be a variety. If $\vv V_{\si}$ does not have the essential amalgamation property, then $\vv V$ does not have the amalgamation property.
\end{theorem}
We will exhibit a V-formation, constituting of subdirectly irreducible integral chains, that has no essential amalgam in residuated chains; we will use this fact together with the theorem above to disprove the amalgamation property in several varieties of semilinear residuated lattices.

In varieties with the {\em congruence extension property} (CEP), such as all varieties of commutative residuated lattices, the amalgamation property corresponds to the {\em deductive interpolation property} of the associated logic. 
We say that a logic, associated to a consequence relation $\vdash$, has the deductive interpolation property if for any set of formulas $\Gamma \cup \{\psi \}$, if $\Gamma \vdash \psi$ there exists a formula $\delta$ such that $\Gamma\vdash \delta$, $\delta \vdash \psi$ and the variables appearing in $\delta$ belong to the intersection of the variables appearing both in $\Gamma$ and in $\psi$, in symbols $Var(\delta) \subseteq Var(\Gamma) \cap Var(\psi)$. 
If a logic $\cc L$ has a variety $\vv V$ as its equivalent algebraic semantics, and $\vv V$ satisfies the congruence extension property, $\cc L$ has the deductive interpolation property if and only if $\vv V$ has the amalgamation property. Without the CEP, the amalgamation property corresponds to the stronger Robinson property, see \cite{MMT2014}. 

Since all commutative varieties of residuated lattices have the CEP \cite[Lemma 3.57]{GJKO}, our results will also demonstrate that several important substructural logics lack the deductive interpolation property.

\subsection{Amalgamation failures}
In this subsection we will use the blockwise gluing construction to build a V-formation $(\alg A, \alg B, \alg C, f,g )$ that has no amalgam in chains. We exhibit an example given by three small 2-potent integral chains; note that idempotent integral chains do have the AP \cite{Gabbay, Maksimova}.

We take $\alg A$ to be the three-element G\"odel chain $A=\{1, a, \bot\}$ with $1 > a = a^2 >\bot$. We now construct two chains $\alg B$ and $\alg C$ of which $\alg A$ is a subalgebra. The intuition is that while there is quite some freedom in defining products between two blocks, the total order imposes some restrictions that will cause the failure of the AP. 

Let $\alg 2$ be the two-element Boolean algebra with domain $\{1, \bot\}$, and $\alg \L_3$ be the three-element \L ukasiewicz chain with domain $\L_3 = \{1, b, \bot\}$, where $1 > b > \bot = b^2$.
We take $\alg B$ to be the 1-sum of $\alg \L_3 \oplus \alg 2$, or equivalently, the blockwise gluing $\alg \L_3 \boxplus_\beta \alg 2$ with $\beta = ({\rm id}, {\rm id}, {\rm id}, {\rm id})$. 
In order to define $\alg C$, let us consider the following triple $(\alg K, \sigma, \gamma)$, where: $\alg K$ is the simple CIRL with domain $\{1, c, d, \bot \}$ with $1 > c > d > \bot = c^2$, and $\sigma$ and $\gamma$ are defined as follows:
\begin{align*}
	\sigma (1) &= 1, \;\;\sigma(c)=\sigma(d)=d, \;\; \sigma(\bot)=\bot\\
	\gamma (1) &= 1, \;\;\gamma(c) = \gamma(d)=c,\;\; \gamma(\bot)=\bot.
\end{align*}
Easy calculations show the following.
\begin{lemma}
	$(\alg K, \sigma, \gamma)$ is a lower block. 
\end{lemma}
We set $\alg C: = \alg K \boxplus_\beta \alg 2$ with $\beta = (\sigma,\sigma,\gamma, \gamma)$.
Let us call $\VS$ the V-formation $(\alg A, \alg B, \alg C, {\rm id}_{\alg A}, {\rm id}_{\alg A})$. To help the intuition, Figure \ref{figure:VSformation} draws the algebras $\alg A, \alg B, \alg C$ with the relevant operations.
\begin{figure}[h!]
\begin{center}
\begin{tikzpicture}
\draw (-1,1) -- (-1,0);
     \fill (-1,1) circle (0.05);
 \node[right] at (-1,1) {$1$}; 
 \fill (-1,0.5) circle (0.05);
 \node[right] at (-1,0.55) {$a = a^2$};
 \fill (-1,0) circle (0.05);
 \node[right] at (-1,0) {$\bot$};
\node at (-0.9,-0.7) {$ \alg A$}; 

\draw (3,1.5) -- (3,0);
 \fill (3,1.5) circle (0.05);
 \node[right] at (3,1.5) {$1$}; 
 \fill (3,1) circle (0.05);
 \node[right] at (3,1) {$a = a^2$};
 \fill (3,0.5) circle (0.05);
 \node[right] at (3,0.5) {$b=ab=b \to \bot$};
 \fill (3,0) circle (0.05);
 \node[right] at (3,0) {$\bot = b^2$};
\node at (3.2,-0.7) {$ \alg B$}; 

\draw (7.5,2) -- (7.5,0);
   \fill (7.5,2) circle (0.05);
 \node[right] at (7.5,2) {$1$}; 
 \fill (7.5,1.5) circle (0.05);
 \node[right] at (7.5,1.5) {$a = a^2$};
 \fill (7.5,1) circle (0.05);
 \node[right] at (7.5,1) {$c = c \to \bot$};
 \fill (7.5,0.5) circle (0.05);
 \node[right] at (7.5,0.5) {$d = ac = ad$};
  \node[right] at (7.5, 0) {$\bot = c^2$};
  \fill (7.5, 0) circle (0.05);
  \node at (7.7,-0.7) {$ \alg C$};
  \end{tikzpicture}\caption{The algebras in the V-formation $\VS$.}\label{figure:VSformation}
   \end{center}  
\end{figure}

 We will show that there is no amalgam to $\VS$ in the class of $\mathsf{RL}$-chains.

\begin{lemma}\label{lemma:noamalgam}
	$\VS$ is an essential span and there is no  amalgam $(\alg D, h,k)$ to $\VS$ in the class of $\mathsf{RL}$-chains.
\end{lemma}
\begin{proof}
First, let us show that $\VS$ is an essential span, i.e. that the identity map from $\alg A$ to $\alg C$ is an essential embedding. This follows from the fact that $\alg A$ and $\alg C$ share the same coatom, and that congruence filters are upwards closed; hence, $a$ belongs to all nontrivial filters of both $\alg A$ and $\alg C$. In other words, let $\theta$ be any nontrivial congruence of $\alg C$, i.e. $\theta \neq \Delta_{\alg C}$, then the pair $(1, a) \in \theta$. If one considers $\theta \cap A^2$, then $(1,a) \in \theta \cap A^2$, getting that $\theta \cap A^2 \neq \Delta_{\alg A}$. Thus, the identity map is an essential embedding and $\VS$ is an essential span.

It is left to prove that there is no amalgam for $\VS$ in residuated chains. Assume by way of contradiction that there is an amalgam $(\alg D, h,k)$ to $\VS$ where $\alg D$ is a totally ordered residuated lattice. Note that both $\alg B$ and $\alg C$ must embed into $\alg D$, and in particular, the embeddings must respect the elements of $\alg A$, i.e. for all $x \in A$, $h(x) = k(x)$. To ease the notation, let us consider $h$ and $k$ to be the identity maps. Notice that since $ab=b$ in $\alg B$ but $ac=d \neq c$ in $\alg C$, necessarily $b \neq c$ in $\alg D$. 

Now, since $\alg D$ is a chain, $b$ and $c$ are comparable.
However, if one assumes that $b < c$, by order preservation: $$bc \leq c^2=\bot$$ and therefore by residuation, and given that $b \backslash \bot=b$ in $\alg B$ (hence also in $\alg D$), $$c \leq b \backslash \bot =b,$$ which contradicts the fact that $b < c$. 
Similarly, if one assumes that $c < b$, then we get $cb \leq b^2=\bot$ and thus $b \leq c \backslash \bot = c$, where the last equality holds in $\alg C$; this contradicts the fact that $c < b$.

Therefore, there is no amalgam $(\alg D, h,k)$ to $\VS$ in $\mathsf{RL}$-chains. 
\end{proof}
A first direct consequence of the above lemma is the following.
\begin{theorem}
	The class of totally ordered residuated lattices does not have the AP.
\end{theorem}
\begin{remark}
	A possibly illuminating way of thinking about the previous example in FL$_{ew}$-algebras is the following. Consider the algebras in the $\VS$ formation as $0$-bounded, where the constant $0$ is then evaluated at $\bot$. Then both $\alg B$ and $\alg C$ have a negation fix-point, namely $b$ and $c$ respectively. Now, both $\alg B$ and $\alg C$ must be subalgebras of the potential amalgam $\alg D$. But an FL$_{ew}$-chain can have only one negation fix-point: therefore in $\alg D$ the elements $b$ and $c$ must coincide. However, $b$ and $c$ multiply differently with respect to $a \in \alg A$, since $ab = b$ but $ac \neq c$. Thus, there can be no amalgam in chains. 
\end{remark}
If one drops the requirement of the total order, an amalgam to $\VS$ can be found;
indeed, $\vv{FL}_{ew}$ has the AP. For instance, one can consider the commutative residuated lattice whose operations uniquely extend those in the following diagram:
\begin{center}
\begin{tikzpicture}
\draw (6.5, 3.5) -- (6.5,2.5) -- (6,1.75) --(6.5, 1) -- (7,1.5) -- (7, 2) -- (6.5,2.5);
\draw (6.5,0.5) -- (6.5,1);
 \fill (6.5,3.5) circle (0.05);
 \node at (6.65,3.6) {$1$};
 \fill (6.5,3) circle (0.05);
 \node at (7.17,3.1) {$a = a^2$}; 
 \fill (6.5,2.5) circle (0.05);
 \node at (7.15,2.55) {$e = ae$};
 \fill (6,1.75) circle (0.05);
 \node at (5.4,1.8) {$b  = ab$};
 \fill (7, 2) circle (0.05);
 \node at (7.2, 2.05) {$c$};
 \fill (7,1.5) circle (0.05);
 \node at (8.05,1.55) {$d = ac = ad$};
  \node at (7.55,0.5) {$\bot = c^2 = b^2$};
  \fill (6.5, 1) circle (0.05);
  \node at (7.6,1) {$f = e^2 = bc$};
  \fill (6.5,0.5) circle (0.05);
\end{tikzpicture}\end{center}
And if one keeps the total order, but restricts to the division-free signature $\{\cdot, \land, \lor, 1\}$, then once again one can find an amalgam. See the following chain, which is an amalgam for the V-formation given by the $\{\cdot, \land, \lor, 1\}$-reducts of the algebras in $\VS$:
\begin{center}
\begin{tikzpicture}
 \draw (7.5,2.5) -- (7.5,0);
 \fill (7.5,2.5) circle (0.05);
 \node[right] at (7.5,2.5) {$1$}; 
   \fill (7.5,2) circle (0.05);
 \node[right] at (7.5,2) {$a = a^2$};
 \fill (7.5,1.5) circle (0.05);
 \node[right] at (7.5,1.5) {$b = ab$};
 \fill (7.5,1) circle (0.05);
 \node[right] at (7.5,1) {$c$};
 \fill (7.5,0.5) circle (0.05);
 \node[right] at (7.5,0.5) {$d = ac = ad$};
  \node[right] at (7.5, 0) {$\bot = b^2$};
  \fill (7.5, 0) circle (0.05);
\end{tikzpicture}\end{center} 
Note that this chain above is the reduct of a residuated lattice, but here $c \to \bot = b \neq c$.

Let us proceed to the main results of this section. The above Lemma \ref{lemma:noamalgam} provides a failure of the necessary condition in Theorem \ref{thm:fussnersantchi} for any semilinear variety to which the algebras in the $\VS$ formation belong.
\begin{theorem}\label{theorem:noAP}
	Let $\vv V$ be a variety of semilinear residuated lattices such that the algebras $\alg A, \alg B, \alg C$ in $\VS$ belong to $\vv V$. Then $\vv V$ does not have the AP. 
\end{theorem}
\begin{proof}
	Since $\vv V$ is semilinear, its subdirectly irreducible members are chains. By Theorem \ref{thm:fussnersantchi}, since $\alg A, \alg B, \alg C \in \vv V$ and by Lemma \ref{lemma:noamalgam} there is no essential amalgam to $\VS$ in $\vv V_{\si}$, $\vv V$ does not have the AP.
\end{proof}
Observe that by evaluating the constant $0$ at $\bot$ in the algebras of $\VS$, one obtains the same results for semilinear varieties of $\vv{FL}$. Let us call $\VS_0$ the corresponding V-formation of bounded chains.
\begin{theorem}\label{theorem:noAP0}
	Let $\vv V$ be a variety of semilinear FL-algebras such that the (0-bounded versions of the) algebras $\alg A, \alg B, \alg C$ in $\VS_0$ belong to $\vv V$. Then $\vv V$ does not have the AP. 
\end{theorem}
The following collects the main results that are consequences of the above theorems.
\begin{corollary}
	The following varieties of residuated lattices do not have the AP:
	\begin{enumerate}
	    \item $\mathsf{SemCIRL}$
	   \item $\mathsf{SemIRL}$
	   \item $\mathsf{SemCRL}$ \cite{FussnerSantschi24}
	   \item $\mathsf{SemRL}$ \cite{GilferezLeddaTsinakis2015}
	   \item all $n$-potent subvarieties of the above varieties for $n \geq 2$
	   \item  $\BG_{n}$ for all $n \geq 2$, and commutative subvarieties.
	\end{enumerate}
\end{corollary}
\begin{corollary}
	The following varieties of FL-algebras do not have the AP:
	\begin{enumerate}
	   \item $\mathsf{MTL}$
	   \item $\mathsf{SemFL_w}$
	   \item $\mathsf{SemFL_e}$ \cite{FussnerSantschi24}
	   \item $\mathsf{SemFL}$
	   \item all $n$-potent subvarieties of the above varieties for $n \geq 2$
	   \item $\BG_{n}^0$ for all $n \geq 2$, and commutative subvarieties.
	\end{enumerate}
\end{corollary}
We note that our counterexample provides an alternative proof of the failure of the AP in $\mathsf{SemCRL}$ and $\mathsf{SemFL_e}$, proven in \cite{FussnerSantschi24}, and $\mathsf{SemRL}$ proved in \cite{GilferezLeddaTsinakis2015}.
We will now modify the algebras in the $\VS$ formation in order to show the failure of the AP in other interesting subvarieties of semilinear RLs.

First, consider the {\em lifting} of the algebras in $\VS$; i.e., the 1-sums $\alg 2 \oplus \alg A, \alg 2 \oplus \alg B, \alg 2 \oplus \alg C$. Then clearly $(\alg 2 \oplus \alg A, \alg 2 \oplus \alg B, \alg 2 \oplus \alg C,{\rm id}, {\rm id})$ is a V-formation, which we call $\VS_\ell$. The proof of Lemma \ref{lemma:noamalgam} straightforwardly adapts to $\VS_\ell$. 

A different example is obtained by taking the {\em disconnected rotations} of the algebras in $\VS$. This construction was inspired by Wro\'nski's reflection construction for BCK-algebras \cite{W83}, and was developed by Jenei in semigroups and residuated structures \cite{Je00,Je03}. 
Precisely, the disconnected rotation $\alg A^\delta$ of an IRL $\alg A$ is the $\sf{FL}_{ew}$-algebra whose lattice reduct is given by the union of $A$ and its disjoint copy $A' = \{a' : a \in A\}$ with dualized order, placed below $A$: for all $a, b \in A$,  $$ a' < b, \mbox{ and } a' \leq b' \mbox{ iff } b \leq a.$$
In particular, the top element of $\alg A^\delta$ is the top $1$ of $\alg A$ and the  bottom element of $\alg A^\delta$ is the copy $0 := 1'$ of the top $1$. 
$\alg A$ is a $0$-free subreduct, the products in $A'$ are all defined to be the bottom element $0=1'$, and furthermore, for all $a,b \in A$, $$a\cdot b' = (b/a)', \quad b'\cdot a = (a \back b)';$$ 
$$a \back b' = a' /b = (b\cdot a)', \quad a' \back b' = a/b, \quad b'/a' = b \back a.$$
For a pictorial intuition of the lifting and disconnected rotation see Figure \ref{fig:rotation}. 

 \begin{figure}
\begin{center}
\begin{tikzpicture}
\footnotesize{
 \draw (-2.5,1.2) arc (0:180: 0.5 and 1.1);
 \draw [dotted] (-3.5, 1.2) -- (-2.5, 1.2);
  \fill (-3,2.3) circle (0.05);  
     \fill (-3,0.5) circle (0.05);
  \node at (-3,0.2) {$0$}; 
    \node at (-3,2.5) {$1$}; 
      \node at (-3,1.7) {$A$};   
 \draw (0.5,1.2) arc (0:180: 0.5 and 1.1);
 \draw [dotted] (0.5, 1.2) -- (-0.5, 1.2);
  \draw (0.5,0.4) arc (0:-180: 0.5 and 1.1);
 \draw [dotted] (0.5, 0.4) -- (-0.5, 0.4);
  \fill (0,2.3) circle (0.05);
    \fill (0,-0.7) circle (0.05);
 \node at (0,2.5) {$1$}; 
  \node at (0.35,-1) {$1' = 0$};  
  \node at (0,1.7) {$A$}; 
  \node at (0,0) {$A'$}; 
    }
 \end{tikzpicture}\end{center} \caption{The lifting and disconnected rotation of an IRL $\alg A$.} \label{fig:rotation}
 \end{figure}
Let us call $\VS_\delta$ the V-formation $(\alg A^\delta, \alg B^\delta, \alg C^\delta, {\rm id}, {\rm id})$. Then once again the proof of Lemma \ref{lemma:noamalgam} easily adapts to $\VS_\delta$. 
\begin{lemma}
For $* \in \{\ell, \delta\}$, $\VS_*$ is an essential span and there is no amalgam to $\VS_*$ in the class of FL-chains.
\end{lemma}
Therefore:
\begin{theorem}
 Let $* \in \{\ell, \delta\}$.
	Let $\vv V$ be a variety of semilinear FL-algebras such that the algebras in $\VS_*$ belong to $\vv V$. Then $\vv V$ does not have the AP. 
\end{theorem}
Notice that the algebras in $\VS_\delta$ are all involutive, and so belong to the involutive subvariety of $\vv{MTL}$, named $\vv{IMTL}$. For the case of the lifting, we get that all the algebras in the V-formation are pseudocomplemented, i.e. it holds that $x \land \neg x = 0$ for all $x$. Equivalently (in semilinear varieties) it holds the Stone equation $\neg x \lor \neg \neg x = 1$. The subvariety of pseudocomplemented (or Stonean) MTL-algebras is usually called $\vv{SMTL}$.
We get the following main consequences.
\begin{corollary}
	The following varieties do not have the AP:
	\begin{enumerate}
		\item $\mathsf{IMTL}$
		\item $\mathsf{SMTL}$
		\item their $n$-potent subvarieties for all $n \geq 2$.
	\end{enumerate} 
\end{corollary}
\begin{remark}
	We observe that one can obtain other failures of the AP by applying appropriate constructions to the algebras in $\VS$. For instance, the lifting and disconnected rotation constructions can be seen uniformly as {\em generalized $n$-rotations}. This more general construction was introduced in \cite{BMU19}, extending the {\em generalized rotation} of \cite{AFU}, and  later adapted to the non-commutative case in \cite{GalatosUgolini}. Intuitively, generalized $n$-rotations are obtained by taking an IRL $\alg A$, attaching below it a rotated (possibly proper) \emph{nuclear image} of $\alg A$, and then adding a \L ukasiewicz chain of $n$ elements, $n-2$ of which are between the original structure and its rotated nuclear image. This construction has as another particular case  also the {\em connected rotation} construction. As above, one can obtain essentially the same proof as in Lemma \ref{lemma:noamalgam}, and derive that the AP fails in all varieties generated by generalized $n$-rotations of (C)IRLs, for all $n \geq 2$. 
\end{remark}

To conclude, let us rephrase some of our results in terms of interpolation properties. Since, in the presence of the CEP, the amalgamation of a variety of residuated lattices $\vv V$ corresponds to the deductive interpolation property of the corresponding logic $\cc L_{\vv V}$, we obtain the following consequences.
\begin{corollary}
	Let $\vv V$ be a semilinear variety of residuated lattices or FL-algebras with the CEP, such that the algebras (over the appropriate signature) in one among $\VS, \VS_0$, $\VS_\ell$, $\VS_\delta$, belong to $\vv V$. Then the corresponding logic $\cc L_{\vv V}$ does not have the deductive interpolation property. 
\end{corollary}
To highlight the most relevant novel results:
\begin{corollary}
	The following logics do not have the deductive interpolation property: MTL, IMTL, SMTL, and their $n$-potent extensions.
\end{corollary}


\end{document}